\documentclass[12pt]{amsart}
\usepackage{amssymb,amsmath,amsfonts,amsthm,bbm,mathrsfs,times,graphicx,color,comment}
\usepackage{tikz}
\usetikzlibrary{arrows,shapes}
\usetikzlibrary{fadings}
\usetikzlibrary{intersections}
\usetikzlibrary{decorations.pathreplacing}

\usepackage{extarrows}% http://ctan.org/pkg/extarrows

\newcommand{\BR}{\mathbb{R}}
\newcommand{\BC}{\mathbb{C}}
\newcommand{\BZ}{\mathbb{Z}}

\DeclareMathOperator{\re}{\mathbb{R}e}
\DeclareMathOperator{\im}{\mathbb{I}m}

\newcommand{\lnorm}[1]{ \left\| #1 \right\|}

\newcommand{\spr}[2]{\langle #1,#2 \rangle}

\renewcommand{\i}{\mathrm i}

\newcommand{\INF}{{\infty}}

\newcommand{\tta}{\theta}

\newcommand{\OM}{\Omega}

\newcommand{\sph}{{{\mathbb S}^ 1}}
\newcommand{\del}{\partial}

\newcommand{\Gam}{\varGamma}
\newcommand{\ol}{\overline}
\newcommand{\ds}{\displaystyle}
\newcommand{\dba}{\overline{\partial}} 
\newcommand{\fii}{{\varphi}}
\newcommand{\bu}{{\bf u}}
\newcommand{\bv}{{\bf v}}

\newcommand{\bg}{{\bf g}}

\newcommand{\bbf}{{\bf f}}
\newcommand{\bh}{{\bf h}}

\newcommand{\bF}{{\bf F}}

\newcommand{\B}{\mathcal{B}}

\newcommand{\HT}{\mathcal{H}}

\newcommand{\bzero}{\mathbf 0}

\newcommand{\btheta}{\boldsymbol \theta}
\newcommand{\balpha}{\boldsymbol \alpha}
\newcommand{\bbeta}{\boldsymbol \beta}

\newcommand{\jpj}{\langle j\rangle}

\newtheorem{theorem}{Theorem}[section]
\newtheorem{prop}{Proposition}[section]
\newtheorem{lemma}{Lemma}[section]
\newtheorem{remark}{Remark}[section]

\title[On the $X$-ray transform of symmetric higher order tensors]{On the $X$-ray transform of planar symmetric tensors}
\begin{document}
	\date{\today}
	\author{David Omogbhe}
	\address{Faculty of Mathematics, Computational Science Center,  University of Vienna, Oskar-Morgenstern-Platz 1, 1090 Vienna, Austria}
	\email{david.omogbhe@univie.ac.at}
	
	\author{Kamran Sadiq}
	%\address{Faculty of Mathematics, Computational Science Center,  University of Vienna, Oskar-Morgenstern-Platz 1, 1090 Vienna, Austria}
	\address{Johann Radon Institute for Computational and Applied Mathematics (RICAM), Altenbergerstrasse 69, 4040 Linz, Austria}
	\email{kamran.sadiq@ricam.oeaw.ac.at}
	\maketitle
	\begin{abstract}
		In this article we characterize the range of the attenuated and non-attenuated $X$-ray transform of compactly supported symmetric tensor fields in the Euclidean plane. 
		%In this article we characterize the range of the attenuated and non-attenuated $X$-ray transform of a symmetric tensor of compact support in the plane. 
		The characterization is in terms of a Hilbert-transform associated with $A$-analytic maps in the sense of Bukhgeim. 
		%In this article we study the non-attenuated and attenuated $X$-ray transform of higher order symmetric tensors supported in convex bounded subsets with sufficiently smooth boundary in the Euclidean plane.  We characterize its range and the characterization is in terms of a Hilbert-transform associated with $A$-analytic maps in the sense of Bukhgeim. 
	\end{abstract}
	
	\section{Introduction} \label{sec:intro}
	We consider here the problem  of the range characterization of (non)-attenuated $X$-ray transform of a real valued symmetric $m$-tensors in a strictly convex bounded domain in the Euclidean plane.
	As the $X$-ray and Radon transform \cite{radon1917} for planar functions (0-tensors)  differ merely by the way lines are parameterized, the $m=0$ case is the classical Radon transform \cite{radon1917}, for which  the range characterization has been long established independently by Gelfand and Graev \cite{gelfandGraev}, Helgason  \cite{helgason65}, and Ludwig \cite{ludwig}. 
	Models in the presence of  attenuation have also been considered in the homogeneous case \cite{kuchmentLvin,aguilarEhrenpreisKuchment}, and in the non-homogeneous case in the breakthrough works \cite{ABK, novikov01,novikov02}, and subsequently \cite{natterer01, bomanStromberg, bal04, kazantsevBukhgeimJr04, monard17}. 
	The references here are by no means exhaustive.
	
	The  interest in the range characterization problem in the $0$-tensors case stems out from their applications to data enhancement  in medical imaging methods such as Single Photon Emission Computed Tomography or Positron Emission Computed Tomography  \cite{nattererBook,finch}. 
	The $X$-ray transform of $1$-tensors  (Doppler transform \cite{nattererWubbeling, sharafutdinov_book94}) appears in the investigation of velocity distribution in a flow \cite{braunHauk},
	% in the investigation to detect tumors using blood flow measurements \cite{WHSWW77}, 
	in ultrasound tomography \cite{sparSLP95,schuster08}, and also in non-invasive industrial measurements for reconstructing the velocity of a moving fluid \cite{norton88,norton}.  The $X$-ray transform of second order tensors arises as the linearization of the boundary rigidity problem \cite{sharafutdinov_book94}.
	%; see \cite{derevtsovSvetov15} for a survey on the tensor tomography in the Euclidean plane.
	The case of tensor fields of rank four describes the perturbation of travel times of compressional waves propagating in slightly anisotropic elastic media \cite[Chapters 6,7]{sharafutdinov_book94}. Thus, due to the various applications the range characterization problem has been a continuing subject of research.
	
	Unlike the scalar case, the $X$-ray transform of tensor fields has a non-zero kernel, and the null-space becomes larger as the order of the tensor field increases.  For tensors of order $m\geq 1$, it is easy to check that injectivity can hold only in some restricted class: e.g., the class of solenoidal tensors, and it is possible to reconstruct uniquely (without additional information of moment ray transforms \cite{sharafutdinov_book94}) only the solenoidal part of a tensor field.  The non-injectivity of the $X$-ray transform makes  the range characterization problem even more interesting.
	
	For the attenuating  media in planar domains,  interesting enough, the $1$-tensor field can be recovered in the regions of positive absorption as shown in \cite{kazantsevBukhgeimJr06, bal04,tamasan07,sadiqTamasan02}, without using some additional data information \cite{sharafutdinov86,derevtsovPickalov11,mishra}. It is due to a surprising fact that the two-dimensional attenuated Doppler transform with positive attenuation is injective while the non-attenuated Doppler transform is not.
	
	The systematic study of tensor tomography in non-Euclidean spaces originated in \cite{sharafutdinov_book94}. On simple Riemannian surfaces, the range characterization of the geodesic $X$-ray of compactly supported $0$ and $1$-tensors has been established in terms of the scattering relation in \cite{pestovUhlmann04}, and the results were extended in \cite{AMU, denisiuk,venke20} to symmetric tensors of arbitrary order. Explicit inversion approaches in the Euclidean case have been proposed in \cite{kazantsevBukhgeimJr04,derevtsovSvetov15,monard16}. 
	In the attenuating media, tensor tomography was solved for the cases $m = 0, 1$ in \cite{saloUhlmann11}.  Inversion for the attenuated $X$-ray transform for solenoidal tensors of rank two and higher can be   found in \cite{paternainSaloUhlmann13}, with a range characterization in \cite{paternainSaloUhlmann14,monard17,AMU}.
	%;see \cite{paternainSaloUhlmann14} for a comprehensive survey. 
	% The references here are by no means exhaustive.
	%, and to attenuating media; see \cite{paternainSaloUhlmann14} for a comprehensive survey. 

	The original characterization in \cite{gelfandGraev, helgason65, ludwig} was extended to arbitrary symmetric $m$-tensors in \cite{pantyukhina}; see \cite{derevtsovSvetov15} for a partial survey on the tensor tomography in the Euclidean plane. 
	The connection between the Euclidean version of the characterization in \cite{pestovUhlmann04} and the characterization in \cite{gelfandGraev, helgason65, ludwig} was established in \cite{monard16}. Recently, in \cite{sadiqTamasan22} the connection between the range characterization result in \cite{sadiqTamasan01} and the original range characterization in \cite{gelfandGraev, helgason65, ludwig} has been established.

	%Different from the range characterization \cite{paternainSaloUhlmann14} in terms of the scattering relation, in this paper the range conditions are in terms of the Bukhgeim-Hilbert transform for $A$-analytic
	%maps introduced in \cite{sadiqtamasan01}. 

	In here we build on the results in \cite{sadiqTamasan01,sadiqTamasan02, sadiqScherzerTamasan}, and extends them to symmetric tensor fields of any arbitrary order.  In particular, the range characterization therein are given in terms of the Bukhgeim-Hilbert transform \cite{sadiqTamasan01} (the Hilbert-like transform associated with $A$-analytic maps in the sense of Bukhgeim \cite{bukhgeimBook}). 
	The characterization in here can be viewed as an explicit description of the scattering relation in \cite{paternainSaloUhlmann13, paternainSaloUhlmann14} particularized to the Euclidean setting. In the sufficiency part we reconstruct all possible $m$-tensors yielding  identical $X$-ray data; see \eqref{NART_mEvenPsiClass} and \eqref{NART_mOddPsiClass} for the non-attenuated case and  \eqref{ART_mEvenPsiClass} and \eqref{ART_mOddPsiClass} for the attenuated case.

	This article is organized as follows: All the details establishing notations and basic properties of symmetric tensor fields needed here are in Section \ref{sec:prelim}. In Section \ref {sec:A-analytic} we briefly recall existing results on $A$-analytic maps that are used in the proofs. 
	In Section \ref{sec:RangeNART_mEven} and Section \ref{sec:RangeNART_mOdd}, we provide range characterization of  symmetric tensor field $\bbf$ of even order, respectively, odd order in the non-attenuated case.
	In Section \ref{sec:RangeART_mEven} and Section \ref{sec:RangeART_mOdd}, we provide range characterization of  symmetric tensor field $\bbf$ of even order, respectively, odd order in the attenuated case.
	
	%Section 1 - Section \ref{sec:intro}
	%Section 2 - Section \ref{sec:prelim}
	%Section 3 - Section \ref {sec:A-analytic}
	%Section 4 - Section \ref{sec:RangeNART_mEven}
	%Section 5 - Section \ref{sec:RangeNART_mOdd}
	%Section 6 - Section\ref{sec:RangeART_mEven}
	%Section 7 - Section \ref{sec:RangeART_mOdd}
	
	%%%%%%%%%%%%%%%%%%%%%%%%%%%%%%%%%%%%%%%%%%%%%%%%%%%%%%%%
	\section{Preliminaries} \label{sec:prelim}
	%Let us consider the Cartesian coordinate system $(x^1,x^2)$ on the plane $\BR^2$.
	Given an integer $m \geq 0$, let $\mathbf{T}^m (\BR^2)$ denote the space of all real-valued covariant tensor fields of rank $m$: 
	\begin{align}
		\bbf(x^1,x^2) = %\sum_{(i_1,i_2, \cdots, i_m)\in \{1,2\}^m}  
		f_{i_1 \cdots i_m}(x^1,x^2) dx^{i_1} \otimes dx^{i_2} \otimes \cdots  \otimes dx^{i_m}, \quad i_1, \cdots, i_m  \in\{1,2\},
	\end{align} 
	where $\otimes$ is the tensor product, $f_{i_1 \cdots i_m}$ are the components of tensor field $\bbf$ in the Cartesian basis $(x^1,x^2)$, and where by repeating superscripts and subscripts in a monomial a summation from $1$ to $2$ is meant. 
	
	We denote by $ \mathbf{S}^m(\BR^2)$ the space of symmetric covariant tensor fields of rank $m$ on $\BR^2$.
	Let $\sigma:\mathbf{T}^m(\BR^2) \to \mathbf{S}^m(\BR^2)$ be the canonical projection (symmetrization) defined by $\ds 
	%\begin{align}\label{sigma_eq}
	(\sigma \bbf)_{i_1 \cdots i_m} = \frac{1}{m!} \sum_{\pi \in \Pi_m} f_{ i_{\pi(1)} \cdots i_{\pi(m)}}, $
	%\end{align}
	where the summation is over the group $\Pi_m$ of all permutations of the set $\{1,\cdots,m\}$.
	
	A planar covariant symmetric tensor field of rank $m$ has $m+1$ independent component, which we denote by
	\begin{align}\label{eq:cov_tensor}
		\tilde{f}_k := f_{ \underbrace{1\cdots1}_{m-k} \underbrace{2\cdots2}_{k}}, \quad (k = 0, \cdots ,m),
	\end{align}
	in connection with this, a symmetric tensor $\bbf = (f_{i_1 \cdots i_m}, \; i_1, \cdots, i_m = 1, 2 )$ of rank $m$ will be given by a pseudovector of size $m + 1$
	\begin{align*}%\label{pseudovector}
		\bbf = (\tilde{f}_{0}, \tilde{f}_{1},  \cdots ,\tilde{f}_{m-1}, \tilde{f}_{m}). 
	\end{align*} 
	
	We identify the plane $\BR^2$ by the complex plane $\BC$, $z^1 \equiv z = x^1+\i x^2, z^2 \equiv \bar{z} = x^1-\i x^2$. We consider the Cauchy-Riemann operators 
	\begin{align}\label{eq:ddbar}
		\frac{\del}{\del z^1} \equiv \frac{\del}{\del z}  := \frac{1}{2} \left( \frac{\del}{\del x^1}- \i \frac{\del}{\del x^2} \right), \quad   
		\frac{\del}{\del z^2} \equiv \frac{\del}{\del \bar{z}} := \frac{1}{2} \left( \frac{\del}{\del x^1}+ \i \frac{\del}{\del x^2} \right),
	\end{align}
	and the inverse relation by $\ds 
	%\begin{align*}%\label{eq:inv-ddbar}
	\frac{\del}{\del x^1} =  \frac{\del}{\del z} + \frac{\del}{\del \bar{z}}, \quad \frac{\del}{\del x^2} = \i \frac{\del}{\del z} - \i \frac{\del}{\del \bar{z}}.$
	%\end{align*}

	Let $\bbf = ( f_{i_1 \cdots i_m}(x^1, x^2), \; i_1, \cdots, i_m = 1, 2 )$ be real valued symmetric $m$-tensor field in Cartesian coordinates $(x^1, x^2)$, then in complex coordinates $(z^1, z^2)$ it will have
	new components $(F_{i_1 \cdots i_m}(z, \bar{z}))$, which are formally expressed by the covariant tensor law:
	\begin{equation} \label{symtensor_CartesianComplex}
		\begin{aligned}
			F_{i_1 \cdots i_m}(z, \bar{z}) &= \frac{\del x^{s_1}}{\del z^{i_1}} \cdots \frac{\del x^{s_m}}{\del z^{i_m}} f_{s_1 \cdots s_m}(x^1, x^2), \quad \text{and}
			% \quad 
			\\ f_{i_1 \cdots i_m}(x^1, x^2) &= \frac{\del z^{s_1}}{\del x^{i_1}} \cdots \frac{\del z^{s_m}}{\del x^{i_m}} F_{s_1 \cdots s_m}(z, \bar{z}),
			% f_{i_1 \cdots i_m}(x,y) = \frac{\del z^{s_1}}{\del x^{i_1}} \cdots \frac{\del z^{s_m}}{\del x^{i_m}} F_{s_1 \cdots s_m}(z, \bar{z}),
		\end{aligned}
	\end{equation} where the Jacobian matrix has the form 
	\begin{align*}
		J :=  
		\begin{pmatrix} \frac{\del x^1}{\del z^1} &  \frac{\del x^1}{\del z^2}\\  \frac{\del x^2}{\del z^1} &  \frac{\del x^2}{\del z^2} \end{pmatrix} 
		= \frac{1}{2}\begin{pmatrix} 1 &  1\\  -\i  & \i \end{pmatrix}, \quad \text{and} \quad
		J^{-1} =
		\begin{pmatrix} \frac{\del z^1}{\del x^1} &  \frac{\del z^1}{\del x^2}\\  \frac{\del z^2}{\del x^1} &  \frac{\del z^2}{\del x^2} \end{pmatrix} 
		= \begin{pmatrix} 1 &  \i\\  1  & -\i \end{pmatrix}.
	\end{align*}
	Adopting the notation in \cite{kazantsevBukhgeimJr04}, we shall write the transformations \eqref{symtensor_CartesianComplex} as 
	\begin{equation} \label{tensor_transformation}
		\begin{aligned}
			\bbf = \{f_{i_1 \cdots i_m}(x^1, x^2)\} \quad  &\rightarrowtail \quad  \bF = \{F_{i_1 \cdots i_m}(z, \bar{z})\} , \quad \text{and} \\
			\bF = \{F_{i_1 \cdots i_m}(z, \bar{z})\}\quad  &\rightarrowtail \quad \bbf = \{f_{i_1 \cdots i_m}(x^1, x^2)\}.
		\end{aligned}
	\end{equation}

	A symmetric tensor $\bF$ of rank $m$, obtained from the real symmetric tensor $\bbf$ by passing to complex variables, we also define a pseudovector $(F_{0}, F_{1}, \cdots, F_{m-1}, F_{m})$ with components
	\begin{align}\label{eq:symtensor}
		F_k = F_{ \underbrace{1\cdots1}_{m-k} \underbrace{2\cdots2}_{k}}, \quad k = 0, \cdots ,m,
	\end{align}
	and  subject to the conditions
	\begin{align}\label{reality_cond}
		F_{k} = \overline{F}_{m-k}, \quad k = 0, \cdots ,m.
	\end{align}
	Taking into account the tensor law \eqref{symtensor_CartesianComplex}, we obtain formulas relating the components of pseudovectors in \eqref{eq:cov_tensor} and pseudovectors in \eqref{eq:symtensor}:
	\begin{align}\label{eq:F_k}
		F_k &= \frac{(-\i)^{m-k}}{2^m} \sum_{q=0}^{m-k} \sum_{p=0}^{k} {m-k \choose q} {k \choose p} \i^{k-p+q} \tilde{f}_{p+q}, \quad k = 0,1, \cdots, m, \\ \label{eq:tilde_fk}
		\tilde{f}_k &=  \i^{k} \sum_{q=0}^{m-k} \sum_{p=0}^{k} {m-k \choose q} {k \choose p} (-1)^{k-p} F_{p+q}, \quad k = 0,1, \cdots ,m.
	\end{align}

	In Cartesian coordinates covariant and contravariant components are the same, and thus contravariant components of the tensor field $\bbf$ coincide with its corresponding covariant components, $f_{i_1 \cdots i_m} = f^{i_1 \cdots i_m}$.
	The dot product on $\mathbf{S}^m(\BR^2)$ induced by the Euclidean metric is defined by 
	\begin{align}\label{innerprod}
		\spr{\bbf}{\bh} :=%\sum_{(i_1,i_2, \cdots, i_m)\in \{1,2\}^m}
		f_{i_1 \cdots i_m} h^{i_1 \cdots i_m}. %\quad (\bbf, \bh \in \mathbf{S}^m(\BR^2)),
	\end{align}
	%where by repeating super- and subscripts in a monomial a summation from $1$ to $2$ is meant. 
	Note that if $ \bbf_1 \rightarrowtail \bF_1$ and $ \bbf_2 \rightarrowtail \bF_2$, then the pointwise inner product of tensors is invariant:
	\begin{align}\label{invariant_ptwise_innerprod}
		\spr{\bbf_1}{\bbf_2} = \spr{\bF_1}{\bF_2}.
	\end{align}
	
	For $\btheta = (\theta^1, \theta^2)=(\cos \theta, \sin \theta) \in \sph$, we  denote by $\btheta^m$ the tensor product $\ds \btheta^m := \underbrace{ \btheta  \otimes \btheta \otimes \cdots \otimes \btheta}_m$  and $\btheta^m$ will be an $m$-contravariant tensor in Cartesian coordinates.
	According to the tensor law for contravariant components its representation
	in complex coordinates will look like
	\begin{align*}
		\btheta \rightarrowtail \Theta, \qquad \Theta^k = \frac{\del z^{k}}{\del x^{s}} \theta^s, \qquad \Theta = ( \Theta^1,  \Theta^2) = ( e^{\i \tta},  e^{-\i \tta}),
	\end{align*}
	and $\ds  \Theta^m := \underbrace{ \Theta \otimes \Theta \otimes \cdots \otimes \Theta}_m$ 
	be an $m$-contravariant tensor, and we also have $\ds \btheta^m \rightarrowtail \Theta^m$. \\
	Using \eqref{invariant_ptwise_innerprod}, we get
	\begin{equation}\label{eq:innerprod_fThetaM}
		\begin{aligned}
			\langle \bbf, \btheta^m \rangle 
			% = f_{i_1 \cdots i_m} \theta^{i_1} \cdot \theta^{i_2} \cdots \theta^{i_m} 
			% \xlongequal{  \text{By }\eqref{invariant_ptwise_innerprod} }
			&=  \langle \bF, \Theta^m \rangle 
			%  = F_{i_1 \cdots i_m} \Theta^{i_1} \cdot \Theta^{i_2} \cdots \Theta^{i_m} \\
			= \sum_{k=0}^{m} {m \choose k} F_{k} \; e^{\i \tta (m-k)}e^{-\i \tta k}   = \sum_{k=0}^{m} {m \choose k} \; F_{k} e^{\i (m-2k)\tta}\\
			&= %\sum_{k=0}^{m} {m \choose k} \; \bF_{k} e^{\i (2k-m)\varphi} =
			\begin{cases} 
				%   \ds \sum_{k=-\frac{m}{2}}^{\frac{m}{2}} f_{2k} e^{\i (2k)\varphi},  & (m \;\text{even}), \\
				%   \ds \sum_{k=-\frac{m-1}{2}}^{\frac{m-1}{2}} f_{2k-1} e^{\i (2k-1)\varphi}, & (m \;\text{odd}),
				% \ds \sum_{k=-q}^{q} f_{-2k} e^{\i (2k)\tta},   & (m =2q,\; q \geq 0), \\
				\ds \sum_{k=0}^{q}f_{-2k} e^{\i (2k)\tta}  + \sum_{k=1}^{q} f_{2k} e^{-\i (2k)\tta},   & (\text{if } m =2q,\; q \geq 0), \\
				% \ds \sum_{k=-q-1}^{q} f_{-(2k+1)} e^{\i (2k+1)\tta},  & (m =2q+1,\; q \geq 0),
				\ds \sum_{k=0}^{q} f_{-(2k+1)} e^{\i (2k+1)\tta}+f_{2k+1} e^{-\i (2k+1)\tta},  & (\text{if } m =2q+1,\; q \geq 0),
			\end{cases}
		\end{aligned}
	\end{equation}
	where 
	\begin{align}\label{eq:Evenmodes_fk}
		f_{-2k}   &= {2q \choose q-k} \; F_{q-k}, &&  0 \leq k \leq q, \, q \geq 0, \quad  \left( q =\frac{m}{2}, m \, \text{even} \right ), \\ \label{eq:Oddmodes_fk}
		f_{-(2k+1)} &=  {2q+1 \choose q-k} \; F_{q-k}, &&  0 \leq k \leq q,\, q \geq 0, \quad  \left( q =\frac{m-1}{2}, m \, \text{odd} \right ),
	\end{align}
	and $f_{n} = \ol{f_{-n}}$ and $F_{n} = \overline{F}_{m-n},$ for $0 \leq n \leq m$.
	
	Let $\bbf$ be a real valued symmetric $m$-tensor, with integrable components of compact support in $\BR^2$,
	%Let $m \geq 0$ be an integer.   For a symmetric $m$-tensor fields $\bbf$ on $\BR^2$,
	and $a \in L^1(\BR^2)$ a real valued function.
	The attenuated $X$-ray transform of $\bbf$ is given by 
	\begin{equation}\label{Xaf} 
		\begin{aligned}
			X_{a}\bbf(x,\btheta) &:= 
			%\int_{-\infty}^{\infty} f_{i_1 \cdots i_m}(x+t\btheta) \btheta^{i_1} \cdots \btheta^{i_m} \exp\left\{{-\int_t^\infty a(x+s\btheta)ds}\right\}dt \\
			\int_{-\infty}^{\infty} \langle \bbf(x+t\btheta), \btheta^m \rangle \exp\left\{{-\int_t^\infty a(x+s\btheta)ds}\right\}dt,
		\end{aligned}
	\end{equation}
	where $x \in \BR^2$, $\btheta \in \sph$, and $\langle\cdot,\cdot \rangle$ is the inner product in \eqref{innerprod}. For the non attenuated case $(a\equiv 0)$, we use the notation $X\bbf$.
	
	In here, we consider the tensor field $\bbf$ be defined on a strongly convex bounded set $\Omega \subset \BR^2$ with vanishing trace at the boundary $\Gamma$; further regularity and the order of vanishing will be specified in the theorems.  
	In the statements below we use the notations in \cite{sharafutdinov_book94}:
	\begin{align*}
		C^\mu(\mathbf{S}^m; \OM) =\left \{\bbf =(f_{i_1 \cdots i_m}) \in \mathbf{S}^m( \OM): f_{i_1 \cdots i_m}\in C^\mu(\OM)\right \}
	\end{align*}
	$0<\mu<1$, for the space of real valued, symmetric tensor fields of order $m$ with 
	%integrable components.  Similarly,  $C^\mu(\mathbf{S}^m; \OM)$, $0<\mu<1$, denotes the tensor fields of order $m$ with 
	locally H\"older continuous components.  %We also use the notation  $\jpn=(1+|n|^2)^{1/2}$. 
	Similarly, $L^1(\mathbf{S}^m; \OM)$ denotes the tensor fields of order $m$ with integrable components.
	
%		In the statements below we use the notations in \cite{vladimirBook}:
%	\begin{align*}
%		L^1(\mathbf{S}^m; \OM) =\left \{\bbf =(f_{i_1 \cdots i_m}) \in \mathbf{S}^m( \OM): f_{i_1 \cdots i_m}\in L^1(\OM)\right \}
%	\end{align*}
%	for the space of real valued, symmetric tensor fields of order $m$ with integrable components.
%	Similarly, 
%	$C^\mu(\mathbf{S}^m; \OM)$, $0<\mu<1$, denotes the tensor fields of order $m$ with locally H\"older continuous components, and we use the notation $\jpn=(1+|n|^2)^{1/2}$.
	
	For any $(x,\btheta)\in\ol\OM\times \sph$, let $\tau(x,\btheta)$
	be  length of the chord  passing through $x$ in the direction of $\btheta$. 
	Let  also consider the incoming $(-)$, respectively outgoing $(+)$ submanifolds of the unit bundle restricted to the boundary
	\begin{align}\label{GammaPM}
		\Gamma_\pm :=\{(x,\btheta)\in \Gamma\times \sph: \pm\btheta\cdot \nu(x)>0\},
	\end{align} and the variety 
	\begin{align}\label{GammaZero}
		\Gamma_0 :=\{(x,\btheta)\in \Gamma\times \sph: \btheta\cdot \nu(x)=0\},
	\end{align} where $\nu(x)$ denotes outer normal.
	
	The $a$-attenuated $X$-ray transform of $\bbf$ is realized as a function on $\Gam_+$ by
	\begin{align}\label{ARTmT}
		X_{a}\bbf(x,\btheta) = \int_{-\tau(x,\btheta)}^{0} \langle \bbf(x+t\btheta)\, , \btheta^m \rangle \,
		e^{-\int_t^0 a(x+s\btheta)ds} \;dt, \; (x,\btheta) \in \Gam_{+}.
	\end{align} 
	
	We approach the range characterization via the well-known connection with the transport model as follows:
	%We approach the range characterization through its connection with the transport model as follows: 
	The boundary value problem
	\begin{subequations}\label{bvp_transport}
		\begin{align} \label{TransportEq2Tensor}
			&  \btheta\cdot\nabla u(x,\btheta) +a(x) u(x,\btheta) = \langle  \bbf(x) , \btheta^m \rangle,  \quad  (x,\btheta)\in \OM \times \sph, \\ \label{UGamMinus}
			&   u|_{\Gam_-}= 0 ,
		\end{align}
	\end{subequations}
	has a unique solution in $\OM \times \sph$ and
	\begin{align}\label{u_Gam+}
		u \lvert_{\Gam_{+}} (x,\btheta) = X_a \bbf (x,\btheta), \quad (x,\btheta)\in \Gam_{+}.
	\end{align}
	
	The range characterization is given in terms of the trace
	\begin{align}\label{g_trace}
		g:=u \lvert_{\Gam \times \sph} = \left \{ \begin{array}{ll}
			X_a \bbf, & \text{ on } \Gam_{+}, \\
			0 , &  \text{ on } \Gam_{-} \cup \Gam_{0}. \\
		\end{array}
		\right.
	\end{align}
	
	We note that from \eqref{eq:innerprod_fThetaM}, the expression $\langle \bbf, \btheta^m \rangle $ in the transport equation \eqref{TransportEq2Tensor} is represented in the Fourier decomposition in $\btheta$ as in terms of the following Fourier modes: 
	\begin{align*}
		\langle \bbf, \btheta^m \rangle  = \begin{cases}
			f_0 +f_{\pm 2}e^{\mp 2 \i \tta}+f_{\pm 4}e^{\mp 4 \i \tta} + \cdots +f_{\pm m}e^{\mp m \i \tta}& (m \;\text{even}), \\
			f_{\pm 1}e^{\mp \i \tta}+f_{\pm 3}e^{\mp 3 \i \tta} + \cdots +f_{\pm m}e^{\mp m \i \tta} & (m \;\text{odd}). \\
		\end{cases}
	\end{align*}
	
	\section{Ingredients from  $A$-analytic theory}\label{sec:A-analytic} 
	In this section we briefly introduce the properties of $A$-analytic maps needed later.\\
	%	The theorems below comprise some results in \cite{sadiqTamasan01,sadiqTamasan02,sadiqTamasan22}. 
	For $0<\mu<1$, $p=1,2$, we consider the  Banach spaces:
	\begin{equation}\label{spaces}
		\begin{aligned} 
			l^{1,p}_{\INF}(\Gam) &:= \left \{ \bg= \langle g_{0}, g_{-1}, g_{-2},...\rangle\; : \lnorm{\bg}_{l^{1,p}_{\INF}(\Gam)}:= \sup_{\xi \in \Gam}\sum_{j=0}^{\INF}  \jpj^p \lvert g_{-j}(\xi) \rvert < \INF \right \},\\
			C^{\mu}(\Gam; l_1) &:= \left \{ \bg= \langle g_{0}, g_{-1}, g_{-2},...\rangle:
			\sup_{\xi\in \Gam} \lVert \bg(\xi)\rVert_{\ds l_{1}} + \underset{{\substack{
						\xi,\eta \in \Gam \\
						\xi\neq \eta } }}{\sup}
			\frac{\lVert \bg(\xi) - \bg(\eta)\rVert_{\ds l_{1}}}{|\xi - \eta|^{ \mu}} < \INF \right \}, \\
			Y_{\mu}(\Gam) &:= \left \{ \bg: \bg \in  l^{1,2}_{\INF}(\Gam) \; \text{and} \;
			\underset{{\substack{
						\xi,\eta \in \Gam \\
						\xi\neq \eta } }}{\sup} \sum_{j=0}^{\INF}  \jpj 
			\frac{\lvert g_{-j}(\xi) - g_{-j}(\eta)\rvert }{|\xi - \eta|^{ \mu}} < \INF \right \},
		\end{aligned}
	\end{equation} where $l_\INF (,l_1)$ is the space of bounded (, respectively summable) sequences, and for brevity, we use the notation $\jpj=(1+|j|^2)^{1/2}$.
	Similarly,  we consider $ C^{\mu}(\ol \OM; l_1) $, and $ C^{\mu}(\ol \OM; l_\INF) $.
	
	%Let $l_\INF (,l_1)$ be the space of bounded (, respectively summable) sequences,
%	Let $L:l_\infty\to l_\infty$ be the left shift $\ds L \langle v_{0}, v_{-1}, v_{-2}, \cdots  \rangle =  \langle v_{-1}, v_{-2},  \cdots \rangle, $  and $L^{k}=\underbrace{L\circ \cdots \circ L}_{k}$ be its $k$-th composition.\\

	A sequence valued map $\OM \ni z\mapsto  \bv(z): = \langle v_{0}(z), v_{-1}(z),v_{-2}(z),... \rangle$
	in $C(\ol\OM;l_\INF)\cap C^1(\OM;l_\INF)$
	is called {\em $L^k$-analytic} (in the sense of Bukhgeim), $k=1,2$,  if
	\begin{equation}\label{Aanalytic}
		\ol{\del} \bv (z) + L^k \del \bv (z) = 0,\quad z\in\OM,
	\end{equation} where $L$ is the left shift operator $\ds L \langle v_{0}, v_{-1}, v_{-2}, \cdots  \rangle =  \langle v_{-1}, v_{-2},  \cdots \rangle, $ and $L^{2}=L\circ L$.
% be its $k$-th composition.
%, and $\dba, \del$ are the Cauchy-Riemann operators in \eqref{eq:ddbar}.
	
	Bukhgeim's original  theory in \cite{bukhgeimBook}  shows that solutions of \eqref{Aanalytic},  satisfy a Cauchy-like integral formula,
	\begin{align}\label{Analytic}
		\bv (z) = \B [\bv \lvert_{\Gam}](z), \quad  z\in\OM,
	\end{align} where $\B$ is 
	the Bukhgeim-Cauchy operator  acting on $\bv \lvert_{\Gam}$. We use the formula in \cite{finch}, where
	$\B$ is defined component-wise for $n\geq 0$ by
	\begin{equation} \label{BukhgeimCauchyFormula}
		\begin{aligned} 
			% u_{-n}(z) &=  (\B \bg)_{-n}(z) \\
			(\B \bg)_{-n}(z) := \frac{1}{2\pi \i} \int_{\Gam}
			\frac{ g_{-n}(\zeta)}{\zeta-z}d\zeta   + \frac{1}{2\pi \i}\int_{\Gam} \left \{ \frac{d\zeta}{\zeta-z}-\frac{d \ol{\zeta}}{\ol{\zeta}-\ol{z}} \right \} \sum_{j=1}^{\infty}  
			g_{-n-j}(\zeta)
			\left( \frac{\ol{\zeta}-\ol{z}}{\zeta-z} \right) ^{j},\; z\in\OM.
		\end{aligned}
	\end{equation}

	The following regularity result in \cite[Proposition 4.1]{sadiqTamasan01} is needed. 
	\begin{prop}\cite[Proposition 4.1]{sadiqTamasan01}\label{functoseq_regularityprop} 
		Let $\mu>1/2$ and $\bg= \langle g_{0}, g_{-1}, u_{-2}, ... \rangle $ be the sequence valued map of non-positive Fourier modes of $g$.
		
		(i) If $g\in C^\mu(\Gam; C^{1,\mu}(\sph))$, then $\bg\in l^{1,1}_\infty(\Gam)\cap C^\mu(\Gam; l_1)$.
		
		(ii) If $g\in C^\mu(\Gam; C^{1,\mu}(\sph))\cap C(\Gam;C^{2,\mu}(\sph))$, then $\bg\in Y_\mu(\Gam)$.
	\end{prop} 
		
	Similar to the analytic maps, the traces of $L$-analytic maps  on the boundary must satisfy some constraints, which can be expressed in terms of a corresponding Hilbert-like transform introduced in  \cite{sadiqTamasan01}. More precisely, the Bukhgeim-Hilbert transform $\HT$ acting on  $\bg$, 
	\begin{align}\label{boldHg}
		\Gam \ni z\mapsto  (\HT \bg)(z)& = \langle (\HT \bg)_{0}(z), (\HT \bg)_{-1}(z),(\HT \bg)_{-2}(z),... \rangle
	\end{align}
	is defined component-wise for $n\geq 0$ by
	
	\begin{equation} \label{BHtransform}
		\begin{aligned} 
			(\HT \bg)_{-n}(z)&=\frac{1}{\pi }\int_\Gam \frac{ g_{-n}(\zeta)}{\zeta-z}d\zeta + \frac{1}{\pi }\int_{\Gam} \left \{ \frac{d\zeta}{\zeta-z}-\frac{d \ol{\zeta}}{\ol{\zeta}-\ol{z}} \right \} \sum_{j=1}^{\infty}  
			g_{-n-j}(\zeta)
			\left( \frac{\ol{\zeta}-\ol{z}}{\zeta-z} \right) ^{j},\; z\in \Gam,
		\end{aligned}
	\end{equation}
	and we refer to \cite{sadiqTamasan01} for its mapping properties. 

%    Note that the Bukhgeim-Cauchy integral formula in \eqref{BukhgeimCauchyFormula}  %, and Bukhgeim-Hilbert transform in \eqref{BHtransform} 
%    above is restated in terms of $L$-analytic maps as opposed to $L^{2}$-analytic as in \cite{sadiqTamasan01}. 
%    Only the index relabelling 
%    The only change is the index relabeling. In particular, 
%    the index $g_{-n-j}$ will change to $g_{-n-2j}$ to account for $L^{2}$-analytic.
     %due to the difference between $L^{2}$-analytic and $L$-analytic. 
%     In particular, the index $g_{-n-j}$ will change to $g_{-n-2j}$ to account for $L^{2}$-analytic.
%    Moreover, the Bukhgeim-Hilbert transform in \eqref{BHtransform} also account for the index relabelling due to the difference between 
%    $L$-analytic and $L^{2}$-analytic.

    Note that the Bukhgeim-Cauchy integral formula in \eqref{BukhgeimCauchyFormula} above is restated in terms of $L$-analytic maps as opposed to $L^{2}$-analytic as in \cite{sadiqTamasan01}.  
    The only change is the index relabeling. In particular, 
    the index $g_{-n-j}$ will change to $g_{-n-2j}$ therein to account for $L^{2}$-analytic. Moreover, the same index relabelling in the Bukhgeim-Hilbert transform formula \eqref{BHtransform} is made to account for the difference between 
    $L$-analytic and $L^{2}$-analytic.
	
	The following result recalls the necessary and sufficient conditions for a sufficiently regular map to be the boundary value of an  $L^k$-analytic function, $k=1,2$.
	%Key to the proof of the reconstruction is the following characterization of traces of $L^2$-analytic maps.
	\begin{theorem}\label{NecSuf_BukhgeimHilbert_Thm}
		Let $0<\mu<1$, and $k=1,2$. Let $\B$ be the Bukhgeim-Cauchy operator in \eqref{BukhgeimCauchyFormula}.  \\
		Let $\bg = \langle g_{0}, g_{-1}, g_{-2},...\rangle\in Y_{\mu}(\Gam)$ for $\mu>1/2$ be defined on the boundary $\Gamma$, 
		and let $\HT$ be the Bukhgeim-Hilbert transform acting on $\bg$ as in \eqref{BHtransform}.
		
		(i) If $\bg$  is the boundary value of an $L^k$-analytic function, 
		then $\HT \bg\in C^{\mu}(\Gam;l_1)$ and satisfies 
		\begin{align} \label{NecSufEq}
			(I+ \i \HT) \bg = {\bf {0}}.
		\end{align}
		(ii)  If $\bg$  satisfies \eqref{NecSufEq}, then there exists an $L^k$-analytic function $ \bv :=\B \bg \in C^{1,\mu}(\OM;l_1)\cap C^{\mu}(\ol \OM;l_1)\cap C^2(\OM;l_\infty)$, such that
		\begin{align}\label{gdata_defn}
			\bv \lvert_{\Gam} = \bg.
		\end{align}
	\end{theorem}
	For the proof  of  Theorem \ref{NecSuf_BukhgeimHilbert_Thm}  we refer to \cite[Theorem 3.2,  Corollary 4.1, and Proposition 4.2]{sadiqTamasan01} and \cite[Proposition 2.3]{sadiqTamasan02}.
	
	%%%%%%%%%%%%%% Another Ingredient 
	
	Another ingredient, in addition to $L^2$-analytic maps,  consists in the one-to-one relation between solutions
	$ \bu: = \langle u_{0}, u_{-1},u_{-2},... \rangle $
	satisfying
	\begin{align}\label{beltrami}
		\dba u_{-n}(z) +\del u_{-n-2}(z) + a(z)u_{-n-1}(z) &=0,\quad z\in \OM, \; n\geq 0,
	\end{align}
	%for attenuation $a\in C^{2,\mu}(\ol \OM)$, $\mu>1/2$, with $\ds \underset{\ol{\OM}}{\min}\, a > 0$, 
	and the $L^2$-analytic map $\bv = \langle v_{0}, v_{-1},v_{-2},... \rangle $ satisfying
	\begin{align}\label{Analytic1}
		\dba v_{-n}(z) +\del v_{-n-2}(z) &=0,\quad z\in \OM, \; n\geq 0;
	\end{align} via a special function $h$, see \cite[Lemma 4.2]{sadiqScherzerTamasan} for details.
	The function $h$ is defined as 
	\begin{align}\label{hDefn}
		h(z,\btheta) := Da(z,\btheta) -\frac{1}{2} \left( I - \i H \right) Ra(z\cdot \btheta^{\perp}, \btheta^{\perp}),
	\end{align} where $\btheta^\perp$ is  the counter-clockwise rotation of $\btheta$ by $\pi/2$,
	$Ra(s, \btheta^{\perp}) = \ds \int_{-\INF}^{\INF} a\left( s \btheta^{\perp} +t \btheta \right)dt$ is the Radon transform in $\BR^2$  of the attenuation $a$,
	$Da(z,\btheta) =\ds \int_{0}^{\INF} a(z+t\btheta)dt$ is the divergent beam transform of the attenuation $a$, 
	and $\ds H h(s) = \ds \frac{1}{\pi} \int_{-\INF}^{\INF} \frac{h(t)}{s-t}dt $ is the classical Hilbert transform \cite{muskhellishvili}, 
	taken in the first variable and evaluated at $s = z \cdotp \btheta^{\perp}$. 
	The function $h$ appeared first in \cite{nattererBook} and enjoys the crucial property of having vanishing negative Fourier modes yielding the expansions
	%We refer the following Lemma \ref{hproperties} for the properties of $h$ used in here.
	%Let $a\in C^{2,\mu}(\ol \OM)$, $\mu>1/2$, with $\ds \underset{\ol{\OM}}{\min}\, a > 0$, 
	%The function $h$  appeared first in the work of Natterer \cite{naterrerBook}; see also \cite{bomanStromberg} for elegant arguments that show how $h$ extends analytically (in the angular variable on the unit circle $\sph$) inside the unit disc.  we use the special integrating factor proposed by Finch in \cite{finch}, which enjoys the crucial property of having vanishing negative Fourier modes.
	\begin{align}\label{ehEq}
		e^{- h(z,\btheta)} := \sum_{k=0}^{\INF} \alpha_{k}(z) e^{\i k\tta}, \quad e^{h(z,\btheta)} := \sum_{k=0}^{\INF} \beta_{k}(z) e^{\i k\tta}, \quad (z, \btheta) \in \ol\OM \times \sph.
	\end{align}
	%We refer \cite[Lemma 4.1]{sadiqScherzerTamasan} for the properties of $h$ used in here.
	Using the Fourier coefficients of  $e^{\pm h}$,  define the integrating operators $e^{\pm G} \bu $ component-wise for each $n \leq 0$, by 
	\begin{align}\label{eGop}
		(e^{-G} \bu )_n = (\balpha \ast \bu)_n = \sum_{k=0}^{\infty}\alpha_{k} u_{n-k}, \quad \text{and} \quad 
		(e^{G} \bu )_n = (\bbeta \ast \bu)_n = \sum_{k=0}^{\infty}\beta_{k} u_{n-k},
	\end{align} where $\balpha$ and $\bbeta$ is given by
	\begin{align*}%\label{balpha_seq}
		&\ol \OM \ni z\mapsto \balpha(z) := \langle \alpha_{0}(z), \alpha_{1}(z), \alpha_{2}(z),  ... , \rangle,  %\in C^{p,\mu}(\OM ; l_{1})\cap C(\ol\OM ; l_{1}), \\ %\label{bbeta_seq}
		\quad \ol \OM \ni z\mapsto \bbeta(z) := \langle \beta_{0}(z), \beta_{1}(z), \beta_{2}(z),  ... , \rangle. %\in C^{p,\mu}(\OM ; l_{1})\cap C(\ol\OM ; l_{1}).%\label{betasDecay}
		%%%%%%%%%%%%%%%%%%%%%%%%%%%%%%%%%%%%%%%%
	\end{align*}
	Note that $e^{\pm G} $ can also be written in terms of left translation operator as
	\begin{align}\label{eGop_leftshift}
		e^{-G} \bu = \sum_{k=0}^{\infty}\alpha_{k}L^{k} \bu , \quad \text{and} \quad e^{G} \bu = \sum_{k=0}^{\infty}\beta_{k}L^{k} \bu,
	\end{align} where $L^{k}$ is the $k$-th composition of left translation operator.
	It is important to note that the operators $e^{\pm G}$ commute with the left translation, $ [e^{\pm G}, L]=0$.
	%%%%%%%%%%%%%%%%%%%%%%%%%%%%%%%%%%%%%%%%%%%%%%%%%%%%%%%%%%%%
	We refer \cite[Lemma 4.1]{sadiqScherzerTamasan} for the properties of $h$, and we restate the following result \cite[Proposition 5.2]{sadiqTamasan01} to incorporate the operators $e^{\pm G}$ notation used in here.
	%%%%%%%%%%%%%%%%%%%%%%%%%%%%%%%%%%%%%%%%
	\begin{prop}\cite[Proposition 5.2]{sadiqTamasan01}\label{eGprop}
		Let $a\in C^{1,\mu}(\ol \OM)$, $\mu>1/2$. Then $\balpha , \del \balpha, \bbeta , \del \bbeta \in  l^{1,1}_{\INF}(\ol \OM)$, and 
		% for any  $p\in \{0,\frac{1}{2},1\}$, $q\in \{0,1\}$, 
		the operators
		\begin{equation}\label{eGmaps}
			\begin{aligned}
				&(i) \,e^{\pm G}:C^{\mu} (\ol\OM ; l_{\INF}) \to C^{\mu} (\ol\OM ; l_{\INF}); \\
				&(ii) \, e^{\pm G}:C^{\mu} (\ol\OM ; l_{1}) \to C^{\mu} (\ol\OM ; l_{1}); \\
				&(iii) \, e^{\pm G}: Y_{\mu}(\Gam) \to Y_{\mu}(\Gam).
			\end{aligned}
		\end{equation}
	\end{prop}
	%%%%%%%%%%%%%%%%%%%%%%%%%%%%%%%%%%%%%%%%
	\begin{lemma}\cite[Lemma 4.2]{sadiqTamasan02}\label{beltrami_reduction}
		Let $a\in C^{1,\mu}(\ol \OM)$, $\mu>1/2$, and $e^{\pm G}$ be operators as defined in \eqref{eGop}. 
		
		(i) If $\bu \in C^1(\OM, l_1)$ solves $\ds \dba \bu +L^2 \del \bu+ aL\bu = 0$, then  $\ds \bv= e^{-G} \bu \in C^1(\OM, l_1)$ solves $\dba \bv + L^2\del \bv =0$.
		
		(ii) Conversely, if $\bv \in C^1(\OM, l_1)$ solves $\dba \bv + L^2\del \bv =0$, then $\ds \bu= e^{G} \bv \in C^1(\OM, l_1)$ solves $\ds \dba \bu +L^2 \del \bu+ aL\bu = 0$.
	\end{lemma}
	
	%%%%%%%%%%%%%%%%%%%%%%%%%%%%%%%%%%%%%%%%%%%%%%%%%%%%%%%
	%\section{$m$-tensor (Even Case $\geq 0$)-non-attenuated case~($a = 0$)}
	\section{Even order $m$-tensor - non-attenuated case}\label{sec:RangeNART_mEven}

	%In this section we assume $a \equiv 0$.
	We establish necessary and sufficient conditions for a sufficiently smooth function on $\Gam \times \sph$ to be the non-attenuated $X$-ray data of some sufficiently smooth real valued symmetric tensor field $\bbf$ of even order  $m=2q, \; q \geq 0$.
	In this non-attenuated case, %using \eqref{eq:innerprod_fThetaM}
	the transport equation \eqref{TransportEq2Tensor} becomes \begin{align}\label{NART_mEvenTEQ}
		&\btheta\cdot\nabla u(x,\btheta)= \sum_{k=-q}^{q} f_{2k}(x) e^{- \i (2k) \tta}, \quad x\in\OM,
	\end{align} 
	where $f_{2k}$ defined in \eqref{eq:Evenmodes_fk}, and $ f_{2k} = \ol{f_{-2k}}, \;$ for 
	$ -q \leq k \leq q, \, q \geq 0$.  Note that $f_0$ is real-valued while other modes are complex conjugates.
	
	For $z=x_1+\i x_2\in \OM$,  the advection operator $\btheta \cdot\nabla$ in complex notation becomes $e^{-\i \tta}\dba + e^{\i \tta}\del$, 
	where $\btheta=(\cos\tta,\sin\tta)$, and  $\dba, \del$ are the Cauchy-Riemann operators in \eqref{eq:ddbar}.

	If $\ds \sum_{n \in \BZ} u_{n}(z) e^{\i n\tta}$ is the Fourier series  expansion in the angular variable $\btheta$ of a solution $u$ of \eqref{NART_mEvenTEQ}, then, provided some sufficient decay (to be specified later) of $u_n$ to allow regrouping, the equation \eqref{NART_mEvenTEQ} reduces to the system:
	\begin{align} \label{NART_mEven_fmEq}
		&\ol{\del} u_{-(2n-1)}(z) + \del u_{-(2n+1)}(z) = f_{2n}(z), && 0 \leq n \leq q, \; q \geq 0, \\ 
		\label{NARTmEvenAnalyticEq_Odd}
		&\ol{\del} u_{-(2n-1)}(z) + \del u_{-(2n+1)}(z) = 0, && n \geq q+1, \; q \geq 0, \\ 
		\label{NARTmEvenAnalyticEq_Even}
		&\ol{\del} u_{-2n}(z) + \del u_{-(2n+2)}(z) = 0, &&  n \geq 0.
	\end{align}
	
	Recall that the trace $u \lvert_{\Gam \times \sph} := g$ as in \eqref{g_trace}, with $g= X \bbf \text{ on }\Gam_+ \mbox{ and  }g= 0 \mbox{ on }\Gam_- \cup \Gam_{0}.$
%	\begin{align*}
%		g:=u \lvert_{\Gam \times \sph} = \left \{ \begin{array}{ll}
%			X \bbf , & \text{on } \Gam_{+}, \\
%			0 , &  \text{on }  \Gam_{-} \cup \Gam_{0}. \\
%		\end{array}
%		\right. 
%	\end{align*}

	The range characterization is given in terms of the Fourier modes of $g$ in the angular variables:
	\begin{align*}%\label{FourierData}
		g(\zeta,\btheta) = \sum_{n =-\infty}^{\infty} g_{n}(\zeta) e^{\i n\tta}, \quad \zeta \in \Gam.
	\end{align*}
	%Since $\bbf$ is real valued, the solution $\ds u$ and its trace are also real valued and its 
	Since the trace $g$ is also real valued, its Fourier modes will satisfy 
	$\ds g_{-n} = \ol{g}_{n}$, for $ n \geq 0. $ \\
	From the non-positive Fourier modes, we built the sequences
	\begin{align}\label{gEvenOdd_mEven}
		\bg^{\text{even}}&:=\langle g_{0}, {g}_{-2}, g_{-4}, ... \rangle, \quad \text{ and }%\\ \label{gOdd}
		\quad \bg^\text{{odd }}:=\langle g_{-1}, g_{-3}, g_{-5},  ... \rangle.
	\end{align}
	
	From the negative odd modes starting from mode $(2q+1)$, we built the sequence
	\begin{align}\label{mEven_Lm2_gOdd}
		%\bg^{odd < -m}:= 
		L^{q}\bg^\text{{odd }}:= \langle {g}_{-(2q+1)}, g_{-(2q+3)}, g_{-(2q+5)}, ... \rangle, \quad q \geq 0,
	\end{align} where $L^{q}$ is the $q$-th composition of left translation operator.
	%$L^{q}=\underbrace{L\circ \cdots \circ L}_{q}$ be the $q$-th composition of left translation operator.
	
%	We characterize next the $X$-ray data $g$ in terms of the Bukhgeim-Hilbert Transform $\HT$ in \eqref{BHtransform}. 
%	We will construct simultaneously the right hand side of the transport equation \eqref{NART_mEvenTEQ} and the solution $u$ whose trace matches the boundary data $g$. The construction of solution $u$ is via its Fourier modes. 
%	We first construct the non-positive Fourier modes and then the positive Fourier modes are constructed by conjugation. 

	We characterize next the non-attenuated $X$-ray data $g$ in terms of the Bukhgeim-Hilbert Transform $\HT$ in \eqref{BHtransform}. 
	We will construct the solution $u$ of the transport equation \eqref{NART_mEvenTEQ}, whose trace matches the boundary data $g$, and also construct the right hand side of the \eqref{NART_mEvenTEQ}.
	The construction of solution $u$ is in terms of its Fourier modes in the angular variable. We first construct the non-positive Fourier modes and then the positive Fourier modes are constructed by conjugation. 
	For even $m =2q$, $q \geq 1$, apart from $q$ many Fourier modes $u_{-1}, u_{-3}, \cdots u_{-(2q-1)}$, all non-positive Fourier modes are defined by Bukhgeim-Cauchy integral formula \eqref{BukhgeimCauchyFormula} using boundary data. 
	Other than having the traces $u_{-(2j-1)} \big \lvert_{\Gam} = g_{-(2j-1)}, \; 1 \leq j \leq q,\, q \geq 1$, on the boundary, the $q$ many Fourier modes $u_{-(2j-1)}, \; 1 \leq j \leq q, \, q \geq 1$, are unconstrained. They are chosen arbitrarily from the class $\Psi_{g}^{\text{even}}$ of functions of cardinality $q=\frac{m}{2}$ with prescribed trace on the boundary $\Gam$ defined as
	\begin{align}\nonumber %\label{NART_mEvenPsiClass}\nonumber
		\Psi_{g}^{\text{even}}:=
		&\left \{ \left( \psi_{-1}, \psi_{-3}, \cdots ,\psi_{-(2q-1)}\right) \in  \left(C^{1, \mu}(\ol\OM; \BC)\right)^{q}, 2\mu >1:
		% \left(C^{1}(\ol\OM;\mathbb{C})\right)^{q}  : 
		\right. \\  \label{NART_mEvenPsiClass}
		&\left. \quad \vphantom{\int} 
		\psi_{-(2j-1)} \big \lvert_{\Gam}= g_{-(2j-1)}, \; 1 \leq j \leq q, \; q \geq 1 \right\}.
		%&\left \{ \underbrace{\psi_{-1}, \psi_{-3}, \cdots ,\psi_{-(2q-1)}}_{q} \in C^{1}(\ol\OM;\mathbb{C}): \psi_{-(2j-1)} \big \lvert_{\Gam}= g_{-(2j-1)}, \; 1 \leq j \leq q, \; q \geq 1 \right\}.
	\end{align}   
	%By Sobolev embedding, $\psi_{-1}, \psi_{-3}, \cdots, \psi_{-(2q-1)}\in C^{1, \mu}(\ol\OM)$,with 
	%$\mu =1-\frac{2}{p}>\frac{1}{2}$. In particular their traces are in $C^{\mu}(\Gam)$.   
	
%	In the statements below we use the notations in \cite{sharafutdinov_book94}:
%	\begin{align*}
%		C^\mu(\mathbf{S}^m; \OM) =\left \{\bbf =(f_{i_1 \cdots i_m}) \in \mathbf{S}^m( \OM): f_{i_1 \cdots i_m}\in C^\mu(\OM)\right \}
%	\end{align*}
%	$0<\mu<1$, for the space of real valued, symmetric tensor fields of order $m$ with 
%	%integrable components.  Similarly,  $C^\mu(\mathbf{S}^m; \OM)$, $0<\mu<1$, denotes the tensor fields of order $m$ with 
%	locally H\"older continuous components.  %We also use the notation  $\jpn=(1+|n|^2)^{1/2}$.
	\begin{remark}
		In the 0-tensor case ($m = 0$), there is no class, and the characterization of the $X$-ray data $g$ is in terms of the Fourier modes $\bg$.  
	\end{remark}
	
	%We characterize the range for even order tensors in the non-attenuated case.
	\begin{theorem}	[Range characterization for even order tensors]\label{RTmEvenTensor}
		
		(i) Let $\bbf  \in C_{0}^{1,\mu}(\mathbf{S}^m ; \OM)$,  $\mu > 1/2$, 
		be a  real-valued symmetric  tensor field of even order $m = 2q, \, q\geq 0$, and 
		$$g= X \bbf \text{ on }\Gam_+ \mbox{ and  }g= 0 \mbox{ on }\Gam_- \cup \Gam_{0}.$$
		Then  $\bg^{\emph{even}}, \bg^{\emph{odd}} \in l^{1,1}_{\INF}(\Gamma)\cap C^\mu(\Gamma;l_1)$ satisfy
		\begin{align}\label{RTmEvenTensorCondEven}
			&[I+\i\HT]\bg^{\emph{even}} = {\bf {0}}, \\ \label{RTmEvenTensorCondOdd}
			&[I+\i\HT]L^{\frac{m}{2}}\bg^{\emph{odd} } = {\bf {0}},
			% &[I+i\HT]\bg^{odd < -m} = {\bf {0}},
		\end{align} where $\bg^{\emph{even}}, \bg^{\emph{odd}}$ are sequences in \eqref{gEvenOdd_mEven}, and 
		$\HT$ is the Bukhgeim-Hilbert operator in \eqref{BHtransform}.
		
		(ii) Let $g\in C^{\mu} \left(\Gam; C^{1,\mu}(\sph) \right)\cap C(\Gam;C^{2,\mu}(\sph))$ be real valued with $g \lvert_{\Gam_{-} \cup \Gam_{0}}=0$.
		For $q=0$, if the corresponding sequences $\bg^{\emph{even}}, \bg^{\emph{odd}} \in Y_{\mu}(\Gam)$ satisfies \eqref{RTmEvenTensorCondEven} and  \eqref{RTmEvenTensorCondOdd}, 
		then there is a unique real valued symmetric $0$-tensor $\bbf$ such that $g\lvert_{\Gam_{+}} = X\bbf$.
		Moreover, for $q \geq 1$, if $\bg^{\emph{even}}, \bg^{\emph{odd}}\in Y_{\mu}(\Gam)$ satisfies \eqref{RTmEvenTensorCondEven} and  \eqref{RTmEvenTensorCondOdd}, and for each element $\left( \psi_{-1}, \psi_{-3}, \cdots ,\psi_{-(2q-1)}\right)   \in \Psi_{g}^{\emph{even}}$,
		then there is a unique real valued symmetric $m$-tensors $\bbf_{\Psi} \in C^\mu(\mathbf{S}^m; \OM)  $ such that $g\lvert_{\Gam_{+}} = X\bbf_{\Psi}$.
	\end{theorem}
	
	\begin{proof} (i) {\bf Necessity: }		
		Let $\bbf =(f_{i_1 \cdots i_m})  \in C_{0}^{1,\mu}(\mathbf{S}^m;\OM)$.
		Since all components $f_{i_1 \cdots i_m}\in C_{0}^{1,\mu}(\OM)$ are compactly supported inside $\OM$, then for any point at the boundary there is a cone of lines which do not meet the support.
		Thus $g \equiv 0$ in the neighborhood of the variety $\Gam_0$
		which yields $g \in C^{1,\mu}(\Gam \times \sph)$. Moreover, $g$ is the trace on $\Gam \times \sph$ of a solution $u \in C^{1,\mu}(\ol{\OM} \times \sph)$ of the transport equation \eqref{NART_mEvenTEQ}.
		By \cite[Proposition 4.1]{sadiqTamasan01} $\bg^{\text{even}}, \bg^{\text{odd}} \in l^{1,1}_{\INF}(\Gamma)\cap C^\mu(\Gamma;l_1)$. 
		
		If $u$ solves \eqref{NART_mEvenTEQ} then its Fourier modes satisfy \eqref{NART_mEven_fmEq}, \eqref{NARTmEvenAnalyticEq_Odd}, and \eqref{NARTmEvenAnalyticEq_Even}.
		Since the negative even Fourier modes $u_{2n}$ for $n \leq 0$, satisfies the system \eqref{NARTmEvenAnalyticEq_Even}, then
		the sequence valued map
		$$\OM \ni z\mapsto  \bu^{\text{even}}(z):= \langle u_{0}(z), u_{-2}(z), u_{-4}(z), u_{-6}(z), \cdots \rangle$$
		is $L$-analytic in $\OM$ and the necessity part in Theorem  \ref{NecSuf_BukhgeimHilbert_Thm} yields the condition \eqref{RTmEvenTensorCondEven}.
		
		The equation \eqref{NARTmEvenAnalyticEq_Odd} for negative odd Fourier modes starting from negative $2q+1$ mode, yield that the sequence valued map
		$$z\mapsto  \langle u_{-(2q+1)}, u_{-(2q+3)}, u_{-(2q+5)}, ... \rangle$$
		is $L$-analytic in $\OM$ and the necessity part in Theorem  \ref{NecSuf_BukhgeimHilbert_Thm} gives the condition  \eqref{RTmEvenTensorCondOdd}.

		(ii) {\bf Sufficiency: }
		Let $g\in C^{\mu} \left(\Gam; C^{1,\mu}(\sph) \right)\cap C(\Gam;C^{2,\mu}(\sph))$  be real valued with $g \lvert_{\Gam_{-} \cup \Gam_{0}}=0$.
		Since $g$ is real valued, its Fourier modes in the angular variable occurs in conjugates
		\begin{align}\label{reality_g}
			g_{-n}(\zeta) = \ol{g}_{n}(\zeta), \quad \text{ for }  \; n \geq 0, \; \zeta \in \Gam.
		\end{align}
		Let the corresponding sequences  $\bg^{\text{even}}$  satisfying \eqref{RTmEvenTensorCondEven} and $\bg^{\text{odd}}$  satisfying \eqref{RTmEvenTensorCondOdd}. 
		By Proposition (\ref{functoseq_regularityprop}), $\bg^{\text{even}}, \bg^{\text{odd}}\in Y_{\mu}(\Gam)$.
		
		Let $m =2q$, $q \geq 0$, be an even integer.
		To prove the sufficiency we will construct a real valued symmetric $m$-tensor $\bbf$ in $\OM$ and a real valued function $u \in C^{1}(\OM\times \sph)\cap C(\ol\OM \times \sph)$
		such that $u \lvert_{\Gam \times \sph}=g$ and $u$ solves \eqref{NART_mEvenTEQ} in $\OM$.
		The construction of such $u$ is in terms of its Fourier modes in the angular variable and it is done in several steps.

		{\bf Step 1: The construction of even modes $u_{2n}$ for $n \in \BZ$.}

		Apply the Bukhgeim-Cauchy Integral operator \eqref{BukhgeimCauchyFormula} to 
		construct the negative even Fourier modes:
		\begin{align}\label{construction_EVENS}
			\langle u_{0}(z), u_{-2}(z), u_{-4}(z), u_{-6}(z), ... \rangle := \B \bg^{\text{even}}(z), \quad z\in \OM.
		\end{align}
		By Theorem \ref{NecSuf_BukhgeimHilbert_Thm}, the sequence valued map
		\begin{align*}
			z\mapsto \langle u_{0}(z), u_{-2}(z), u_{-4}(z),  ... \rangle \in C^{1,\mu}(\OM;l_1)\cap C^{\mu}(\ol \OM;l_1),
		\end{align*}is $L$-analytic in $\OM$, thus the equations
		\begin{align}\label{uEvenL1}
			\ol{\del} u_{-2n} + \del u_{-2n-2} = 0,
		\end{align} are satisfied for all $n \geq 0$.
		Moreover, the hypothesis \eqref{RTmEvenTensorCondEven} and the sufficiency part of Theorem \ref{NecSuf_BukhgeimHilbert_Thm} yields that they extend continuously to $\Gam$ and $\ds 			u_{-2n}|_\Gam=g_{-2n},$ for all $n\geq 0$. 
		
		Construct the positive even Fourier modes by conjugation:
		$ \ds u_{2n}:=\ol{u_{-2n}}$, for all $n\geq 1$.
		
		By conjugating \eqref{uEvenL1} we note that the positive even Fourier modes also satisfy
		\begin{align*}%\label{uEvenL1+}
			\ol{\del} u_{2n+2} + \del u_{2n} = 0,\quad n\geq 0.
		\end{align*}Moreover, by reality of $g$ in \eqref{reality_g} they extend continuously to $\Gam$ and
		\begin{align*}%\label{trace_uposEven_mEvenNART}
			u_{2n}|_\Gam=\ol{u_{-2n}}|_\Gam=\ol{g_{-2n}}=g_{2n},\quad n\geq 1.
		\end{align*}
		Thus, as a summary from above equations,
		%\eqref{uEvenL1} -- \eqref{trace_uposEven_mEvenNART}, 
		we have shown that the even modes $u_{2n}$ satisfy
		\begin{equation} \label{PostiveEvens_Trace}
			\begin{aligned}
				\ol{\del} u_{2n} + \del{u_{2n-2}} &= 0, 
				%\quad\forall n\in \mathbb{Z},  %\label{PostiveEvensTrace}
				\quad \text { and } \quad 
				u_{2n}\big \lvert_{\Gam} = g_{2n}, \quad \text{ for all } n\in \mathbb{Z}.
			\end{aligned}
		\end{equation}
		
		{\bf Step 2: The construction of odd modes $u_{2n-1}$ for $|n| \geq q, \, q \geq 0$.}
		
		Apply the Bukhgeim-Cauchy Integral operator  \eqref{BukhgeimCauchyFormula} to 
		construct the other odd negative  modes:
		\begin{align}\label{constructionODDSnegative}
			\langle u_{-(2q+1)}(z), u_{-(2q+3)}(z),  \cdots \rangle
			:= \B L^{q}\bg^{\text{odd}} (z), \quad z\in \OM.
		\end{align} 
		By Theorem \ref{NecSuf_BukhgeimHilbert_Thm}, the sequence valued map
		\begin{align*}
			z\mapsto \langle u_{-(2q+1)}(z), u_{-(2q+3)}(z), u_{-(2q+5)}(z),  ..., \rangle \in C^{1,\mu}(\OM;l_1)\cap C^{\mu}(\ol \OM;l_1),
		\end{align*}is $L$-analytic in $\OM$, thus the equations
		\begin{align} \label{negODDS}
			\ol{\del} u_{-(2n+1)} + \del u_{-(2n+3)} &= 0,  
		\end{align} are satisfied for all $n \geq q, \, q \geq 0$.
		Moreover, the hypothesis \eqref{RTmEvenTensorCondOdd} and the sufficiency part of Theorem \ref{NecSuf_BukhgeimHilbert_Thm} yields that they extend continuously to $\Gam$ and
		\begin{align}\label{NegOddsTrace}
			u_{-(2n+1)}|_\Gam=g_{-(2n+1)},\quad \forall n\geq q, \, q \geq 0.
		\end{align}
		Construct the positive odd Fourier modes by conjugation:
		$\ds u_{2n+1}:=\ol{u}_{-(2n+1)},$ for all $n\geq q, \, q \geq 0$.
%		\begin{align*}%\label{construct_odd+}
%  $\ds 	u_{2n+1}:=\ol{u}_{-(2n+1)},\quad n\geq q, \, q \geq 0.
%		\end{align*}
	
	By conjugating \eqref{negODDS} we note that the positive odd Fourier modes also satisfy
		\begin{align}\label{PostiveOdds}
			\ol{\del} u_{2n+3} + \del{u_{2n+1}} = 0, \quad \forall n \geq q, \, q\geq 0.
		\end{align} Moreover, by \eqref{reality_g} they extend continuously to $\Gam$ and
		\begin{align}\label{PostiveOddsTrace}
			u_{2n+1}|_\Gam=\ol{u}_{-(2n+1)}|_\Gam=\ol{g}_{-(2n+1)}=g_{2n+1},\quad n \geq q, \, q\geq 0.
		\end{align}

		{\bf Step 3: The construction of the tensor field $\bbf$ in the $q=0$ case.}
		In the case of the 0-tensor, $\bbf=f_0$, and $f_0$ is uniquely determined   from the odd Fourier mode $u_{-1}$ in \eqref{constructionODDSnegative}, by
		\begin{align}\label{NART_0tensor_reconstruction}
			&f_0 :=2 \re \del u_{-1}, \quad (\text{for }q= 0 \text{ case}).
			%   &\bbf (z):=2 \re \del \left(\B \bg^{odd}\right)_{-1}(z), \quad z \in \OM, \quad q= 0.
		\end{align}
%		and 
%		$\displaystyle u(z, \btheta):= \sum_{n\in\mathbb{Z}} u_{2n}(z)e^{\i (2n\tta)} +  \sum_{n\in\mathbb{Z}} u_{(2n+1)}(z)e^{\i (2n+1)\tta}$  solves \eqref{NART_mEvenTEQ} for $q=0$. Using equations  \eqref{PostiveEvens_Trace},  \eqref{NegOddsTrace}, and \eqref{PostiveOddsTrace}, $u$ defined above  satisfies $\ds u(\cdot,\btheta)\lvert_{\Gam}= g(\cdot,\btheta).$ Thus $g \lvert_{\Gam_{+}} = X\bbf$ for $q=0.$
		
		We consider next the case $q \geq 1$ of tensors of order 2 or higher. 
		In this case the construction of the tensor field $\bbf_{\Psi} $  is in terms of the Fourier mode $u_{-(2q+1)}$ in \eqref{constructionODDSnegative} and the class $\Psi_{g}^{\text{even}}$ in \eqref{NART_mEvenPsiClass}.
		
		{\bf Step 4: The construction of odd modes $u_{\pm(2n-1)}$,  for $1 \leq n \leq q, \; q \geq 1$.}
		
		Recall the non-uniqueness class $\Psi_{g}^{\text{even}}$ in \eqref{NART_mEvenPsiClass}. 
		
		For $\left( \psi_{-1}, \psi_{-3}, \cdots ,\psi_{-(2q-1)}\right)  \in \Psi_{g}^{\text{even}}$ arbitrary, define the modes $u_{\pm 1}, u_{\pm 3}, ..., u_{\pm (2q-1)}$ in $\OM$ by
		\begin{equation}\label{NART_mEven_uPsi}
			%\begin{aligned}
			u_{-(2n-1)} := \psi_{-(2n-1)} \text{ and } u_{2n-1} := \ol{\psi}_{-(2n-1)}, \quad 1 \leq n \leq q, \; q \geq 1.
			%\end{aligned}
		\end{equation}
		
		By the definition of the class \eqref{NART_mEvenPsiClass}, and the reality of $g$ in \eqref{reality_g}, we have
		%Using traces $\psi_{-(2n-1)} \big \lvert_{\Gam}= g_{-(2n-1)}, \; 1 \leq n \leq q, \; q \geq 1,$ from the class $\Psi_{g}^{\text{even}}$, and the fact that $g$ is real valued, we have
		\begin{equation}\label{u1trace}
			\begin{aligned}
				u_{-(2n-1)} \lvert_{\Gam} &= g_{-(2n-1)}, \quad \text { and } \quad 
				u_{2n-1} \lvert_{\Gam} = \ol{g}_{-(2n-1)}= g_{2n-1}, \quad 1 \leq n \leq q, \; q \geq 1.
			\end{aligned}
		\end{equation}
		
		{\bf Step 5: The construction of the tensor field $\bbf_{\Psi} $ whose $X$-ray data is $g$.}
		
		The components of the $m$-tensor $\bbf_{\Psi} $ are defined via the one-to-one correspondence between  the pseudovectors $\langle \tilde{f}_0, \tilde{f}_{1}, \cdots , \tilde{f}_{m} \rangle $  and the functions  $\{f_{2n}:\; -q\leq n\leq q\}$ as follows. 
		
		For $q\geq 1$, we define $f_{2q}$  by using $\psi_{-(2q-1)}$ from the non-uniqueness class \eqref{NART_mEvenPsiClass}, and Fourier mode $u_{-(2q+1)}$ from the Bukhgeim-Cauchy formula \eqref{constructionODDSnegative}. Then, define $\{f_{2n}:\; 0 \leq n \leq q-1\}$ solely from the information in the non-uniqueness class. Finally, define $\{f_{-2n}:\; 1\leq n\leq q\}$ by conjugation. 
		%From $\left( \psi_{-1}, \psi_{-3}, \cdots ,\psi_{-(2q-1)}\right)  \in \Psi_{g}^{\text{even}}$ for $q\geq 1$, and Fourier mode $u_{-(2q+1)}$ in \eqref{constructionODDSnegative}, 
		%we define first $f_{2q}$ depending on $m=2q,\; q\geq1$, then define $f_{2(q-1)}$ depending on $ q\geq2$, and then continue as follows: 
		\begin{equation} \label{NART_mEven_fmPsiEq}
			\begin{aligned} %\label{NART_mEven_fzeroPsiEq}
				&f_{2q}:= \ol{\del} \psi_{-(2q-1)}+ \del u_{-(2q+1)},  && q\geq 1,\\
				&f_{2n}:= \ol{\del} \psi_{-(2n-1)} + \del \psi_{-(2n+1)}, && 1 \leq n \leq q-1, \; q\geq 2, \\
				&f_{0}:=2 \re \del \psi_{-1}, && q\geq 1, \quad \text{and} \\
				% &f_{0}:= \begin{cases}
				% \ds 2 \re \del \psi_{-1}, \quad q\geq 1, \\
				% \ds 2 \re \del u_{-1}, \quad q\geq 0, \\
				% \end{cases} \\
				&f_{-2n}:=\ol{f_{2n}}, && 1 \leq n \leq q, \; q\geq 1,
			\end{aligned}
		\end{equation}
		By construction, $f_{2n}\in C^\mu(\OM)$, for $-q\leq n\leq q$, as  $\psi_{-1},\cdots,\psi_{-2q+1}\in C^{1,\mu}(\OM)$.
		We use these Fourier modes $ f_0 , f_{\pm 2}, f_{\pm 4}, \cdots , f_{\pm 2q}$ for $q\geq 1,$ and equations \eqref{eq:Evenmodes_fk},  \eqref{reality_cond} and \eqref{eq:tilde_fk} to construct the pseudovectors $\langle \tilde{f}_0, \tilde{f}_{1}, \cdots , \tilde{f}_{m} \rangle $, and thus the $m$-tensor field $\bbf_{\Psi}  \in C^\mu(\mathbf{S}^m; \OM)  $.
		
		%In the remaining of the proof  we check that, for the tensor $\bbf defined above, 
		In order to show $g \lvert_{\Gam_{+}} = X\bbf_{\Psi}$ for $ q \geq 1$, with $\bbf_{\Psi}$ being constructed as in \eqref{NART_mEven_fmPsiEq}, we define the real valued function $u$ via its Fourier modes for $q \geq 1$,
		\begin{align}\label{NAT_evensolution}
			\footnotesize u(z,\btheta) =\sum_{n = -\INF}^{\INF} u_{2n}e^{\i 2n\tta}  +\sum_{|n| \geq q} u_{2n+1}e^{\i (2n+1)\tta} 
			+ \sum_{n=1}^{q} \psi_{-(2n-1)} e^{ - \i (2n-1)\tta} + \sum_{n = 1}^{q} \ol{\psi}_{-(2n-1)}e^{\i (2n-1)\tta}.
		\end{align}
		
		Since $g\in C^{\mu} \left(\Gam; C^{1,\mu}(\sph) \right)\cap C(\Gam;C^{2,\mu}(\sph))$, we use Proposition \ref{functoseq_regularityprop} (ii) and 
		\cite[Proposition 4.1 (iii)]{sadiqTamasan01} to conclude that $u$ defined in \eqref{NAT_evensolution} belongs to $C^{1,\mu}(\OM \times \sph)\cap C^{\mu}(\ol{\OM}\times \sph)$. 
		Using \eqref{PostiveEvens_Trace}, \eqref{NegOddsTrace}, \eqref{PostiveOddsTrace},  \eqref{u1trace}, and definition of $\left( \psi_{-1}, \psi_{-3}, \cdots ,\psi_{-(2q-1)}\right)  \in \Psi^{\text{even}}_{g}$ for  $ q \geq 1$, the trace
		$u(\cdot,\btheta)$ in \eqref{NAT_evensolution} extends to the boundary, $$\ds u(\cdot,\btheta)\lvert_{\Gam}= g(\cdot,\btheta).$$

		Since $u\in C^{1,\mu}(\OM \times \sph)\cap C^{\mu}(\ol{\OM}\times \sph)$, 
		then the term by term differentiation in \eqref{NAT_evensolution} is now justified, and $u$ satisfy \eqref{NART_mEvenTEQ}:
		%the transport equation 
		%for $q \geq 1$:
		\begin{align*}
			\btheta \cdot \nabla u &= \ol{\del} \; \ol{\psi_{-1}} +\del \psi_{-1}  
			+ \sum_{n=1}^{q-1} (\ol{\del} \psi_{-(2n-1)} + \del \psi_{-(2n+1)} )e^{ -\i (2n)\tta} + \sum_{n=1}^{q-1} (\ol{\del} \; \ol{\psi}_{-(2n+1)} +\del \ol{\psi}_{-(2n-1)})e^{ \i (2n)\tta} \\
			&\qquad + e^{-\i (2q) \tta} (\ol{\del}\psi_{-(2q-1)}+\del u_{-(2q+1)})   + e^{\i (2q) \tta} (\del \ol{\psi}_{-(2q-1)} + \ol{\del} \; \ol{u}_{-(2q+1)}) \\
			&=  \sum_{n=-q}^{q} f_{2n}(z) e^{- \i (2n) \tta} = \langle \bbf,  \btheta^{2q} \rangle,
			% &\quad +\sum_{n = -\INF}^{\INF} ( \ol{\del}u_{2n}+\del u_{2n-2}) e^{i (2n-1) \tta}+\sum_{n = m}^{\INF} (\ol{\del}u_{-n}+ \del u_{-n-2})e^{-i (n+1) \tta} + \sum_{n = m}^{\INF} ( \ol{\del}u_{n+2}+\del u_{n}) e^{i (n+1) \tta}.
		\end{align*}
		where the cancellation uses equations \eqref{PostiveEvens_Trace}, \eqref{negODDS}, \eqref{PostiveOdds},  \eqref{NART_mEven_uPsi},
		and the second equality uses the definition of $f_{2k}$'s in \eqref{NART_mEven_fmPsiEq}.
		% the real valued $u$ defined in \eqref{NAT_evensolution}, satisfies 
		
		%Now using the definition of $f_{2k}$'s in   \eqref{NART_mEven_fmPsiEq},  the real valued 
		%$u$ satisfies the transport equation \eqref{NART_mEvenTEQ} for $q \geq 1$:
		%\begin{align*}
		%% \tta \cdot \nabla u &= e^{-2i \tta} f_2 + e^{2i \tta} \ol{f_2}+f_0 = \langle \bF_{\psi} \tta ,  \tta \rangle.
		%\btheta \cdot \nabla u &=  \sum_{k=-q}^{q} f_{2k}(x) e^{- \i (2k) \tta}.
		%% \langle \bbf_{\Psi},  \btheta^m \rangle.
		%\end{align*}
	\end{proof}
	
	%%%%%%%%%%%%%%%%%%%%%%%%%%%%%%%%%%%%%%%%%%%%%%%%%%%%%%%%%%%%%%%%%%%
	
	\section{Odd order $m$-tensor - non-attenuated case}\label{sec:RangeNART_mOdd}
	% assume $a \equiv 0$. We 
	In this section we establish necessary and sufficient conditions for a sufficiently smooth function on $\Gam \times \sph$ to be the non-attenuated $X$-ray data of some sufficiently smooth real valued symmetric tensor field $\bbf$ of odd order $ m =2q+1, \; q \geq 0$.
	
	%By using the Cauchy-Riemann operators $\dba, \del$ in \eqref{eq:ddbar}, and $\theta=\arg{\btheta}\in (-\pi,\pi]$,  the transport equation \eqref{TransportEq2Tensor} becomes
	%\begin{align}\label{NART_mEvenTEQ}
	%&[e^{-i\tta}\dba + e^{i\tta}\del] u(z,\btheta)= \sum_{k=-q}^{q} f_{2n}(z) e^{- \i (2n) \tta}, \quad (z,\btheta)\in \OM\times\sph,
	%\end{align}
	%where $f_{-2k}$ as in \eqref{eq:Evenmodes_fk}, for 
	%%and $ f_{2k} = \ol{f_{-2k}}, \;$
	%$ -q \leq k \leq q, \, q \geq 0$. Note that $f_0$ is real-valued while other modes are complex conjugates.
	%
	%If $\ds \sum_{n \in \BZ} u_{n}(z) e^{\i n\tta}$ is the Fourier series  expansion in the angular variable $\btheta$ of a solution $u$ of \eqref{NART_mEvenTEQ}, then, 
	%	By identifying the Fourier modes of the same order, the equation \eqref{TransportEq_mEven} reduces to the system:
	%\begin{align} \label{NART_mEven_fmEq}
	%&\ol{\del} u_{-(2n-1)}(z) + \del u_{-(2n+1)}(z) = f_{2n}(z), && 0 \leq n \leq q,\\ 
	%\label{NARTmEvenAnalyticEq_Odd}
	%&\ol{\del} u_{-(2n-1)}(z) + \del u_{-(2n+1)}(z) = 0, && n \geq q+1.
	%%		\label{NARTmEvenAnalyticEq_Even}
	%%		&\ol{\del} u_{-2n}(z) + \del u_{-(2n+2)}(z) = 0, &&  n \geq 0.
	%\end{align}
	
	In the non-attenuated odd $m$-tensor case, the transport equation \eqref{TransportEq2Tensor} becomes 
	\begin{align}\label{NART_mOddTEQ}
		&\btheta\cdot\nabla u(z,\btheta)= \sum_{n=0}^{q} \left( f_{2n+1}(z) e^{- \i (2n+1) \tta} + f_{-(2n+1)}(z) e^{ \i (2n+1) \tta} \right) , \quad (z,\btheta)\in \OM\times\sph,
	\end{align} where $f_{2n+1}$ defined in \eqref{eq:Oddmodes_fk}, and $ f_{2n+1} = \ol{f_{-2n-1}}, \;$ for $ 0\leq n \leq q, \, q \geq 0$.
	
	If $\ds \sum_{n \in \BZ} u_{n}(z) e^{\i n\tta}$ is the Fourier series  expansion in the angular variable $\btheta$ of a solution $u$ of \eqref{NART_mOddTEQ}, then,  by identifying the Fourier modes of the same order, the equation \eqref{NART_mOddTEQ} reduces to the system:
	%Let $\displaystyle u(z,\btheta) = \sum_{-\infty}^{\infty} u_{n}(z) e^{\i n\tta}$ be the formal Fourier series representation of $u$ in the angular variable $\btheta=(\cos\tta,\sin\tta)$. Since $u$ is real valued, $u_{-n}=\ol{u_n}$ and the angular dependence is completely determined by the sequence of its nonpositive Fourier modes 
	%\begin{align*}%\label{boldu}
	%\OM \ni z\mapsto  \bu(z)&: = \langle u_{0}(z), u_{-1}(z),u_{-2}(z),... \rangle.
	%\end{align*}
	%
	%Provided appropriate convergence of the series (given by smoothness in the angular variable) we see that if $u$ solves \eqref{NART_mOddTEQ} then its Fourier modes solve the system
	\begin{align} \label{NART_mOdd_fmEq}
		&\ol{\del} u_{-2n}(z) + \del u_{-(2n+2)}(z) = f_{2n+1}(z), && 0 \leq n \leq q, \; q \geq 0, \\ \label{NARTmOddAnalyticEq_Even}
		&\ol{\del} u_{-2n}(z) + \del u_{-(2n+2)}(z) = 0, 
		&& n \geq q+1, \; q \geq 0,  \\ \label{NARTmOddAnalyticEq_Odd}
		&\ol{\del} u_{-(2n-1)}(z) + \del u_{-(2n+1)}(z) = 0, &&  n \geq 0.
	\end{align}
	
	In the odd $m$-tensor case, the even and odd Fourier modes of $u$ plays a different role, unlike  the even $m$-tensor case in the previous section. To emphasize this difference we separate the non-positive even  modes $\bu^{\text{even}} :=\langle u_{0}, u_{-2},u_{-4}...\rangle$, and  
	negative odd  modes $\bu^{\text{odd}} :=\langle u_{-1},u_{-3},...\rangle$, 
	%\begin{align*}%\label{uEvenOdd_mOdd_NART}
	%\bu^{\text{even}} := \langle u_{0}, u_{-2},u_{-4}...\rangle, \quad \text{and} \quad
	%\bu^{\text{odd}} := \langle u_{-1},u_{-3},...\rangle,
	%\end{align*} and 
	and note that if $\langle u_{0}(z), u_{-1}(z),u_{-2}(z),... \rangle$ is $L^{2}$-analytic, then $\bu^{\text{even}},\bu^{\text{odd}}$ are $L$-analytic.
	
	Let us consider the sequence $\{\bu^{2k-1}\}_{k\geq1}\subset C(\ol\OM;l_\INF)\cap C^1(\OM;l_\INF)$ given by
	\begin{align}\label{uM_mOdd_NART}
		\bu^{2k-1}:=  \langle u_{2k-1}, u_{2k-3},....,u_{1}, u_{-1},u_{-3}, u_{-5}, ... \rangle, \quad k \geq 1,
	\end{align}
	obtained by augmenting the sequence of negative odd indices
	$\langle u_{-1},u_{-3}, u_{-5}, ... \rangle$ by $k$ many terms in the order $ \underbrace{u_{2k-1}, u_{2k-3},....,u_{1}}_{k} $.
	
	One of the ingredients in our characterization of the odd $m$-tensor is the following simple property of $L$-analytic maps, shown in \cite[Lemma 2.6]{sadiqTamasan01}.
	
	\begin{lemma} \cite[Lemma 2.6]{sadiqTamasan01} \label{UconjLemma_mOdd_NART}
		Let $\{\bu^{2k-1}\}_{k\geq1}$ be the sequence of $L$-analytic maps defined in \eqref{uM_mOdd_NART}. Assume that
		\begin{align}\label{uTraceConj_mOdd_NART}
			u_{2k-1} \lvert_{\Gam} = \ol{{u}_{-(2k-1)}} \lvert_{\Gam}, \quad \forall k\geq 1.
		\end{align}
		Then, for each $k \geq 1$,
		\begin{align}\label{identity_mOdd_NART}
			u_{2k-1}(z) = \ol{{u}_{-(2k-1)}(z)}, \quad z \in \OM.
		\end{align}
	\end{lemma}
	
	%\begin{proof}
	%For each $k\geq 2$, since $\bu^{2k-1}$ is $L$-analytic,
	%\begin{align*}
	%\ol{\del} u_{2j-1} + \del u_{2j-3} = 0, \quad 2 \leq j
	%\leq k.
	%\end{align*} By taking the conjugate and using the fact that $k$ is arbitrary, we get
	%\begin{align*}
	%\ol{\del} \ol{u_{2k-3}} + \del \ol{u_{2k-1}} = 0, \quad \forall k \geq 2.
	%\end{align*} In other words the sequence
	%$\langle \ol{u_1}, \ol{u_3}, \ol{u_5}, \cdots \rangle$ is $L$-analytic, and by \eqref{uTraceConj_mOdd_NART} has the same trace as the sequence $\langle u_{-1}, u_{-3}, u_{-5}, \cdots \rangle$.
	%Since the latter is also $L$-analytic and
	%$L$-analytic map are uniquely determined by their traces  the identity \eqref{identity_mOdd_NART} holds.
	%\end{proof}
	
	The range characterization of data $g$ will be given in terms of its Fourier modes:% in the angular variables:
	\begin{align*}%\label{FourierData_mOdd_NART}
		g(\zeta,\theta) = \sum_{n=-\infty}^{\infty} g_{n}(\zeta) e^{\i n\fii}, \quad \zeta \in \Gam.
	\end{align*}
    Since the trace $g$ is also real valued, its Fourier modes will satisfy
	$\ds g_{-n}= \ol{g}_{n}, $ for $ n \geq 0$.
	From the non-positive even modes, we build the sequence
	\begin{align}\label{gEvenNozero_mOdd_NART}
		\bg^{\text{even}}:=\langle {g}_{0},{g}_{-2}, g_{-4}, g_{-6},... \rangle.
	\end{align}
	For each $k \geq 1 $, we use the odd modes $\{ g_{-1},g_{-3}, g_{-5}, ... \}$ to build the sequence
	\begin{align}\label{gM_mOdd_NART}
		\bg^{2k-1}:= \langle g_{2k-1}, g_{2k-3},....,g_{1}, g_{-1},g_{-3}, g_{-5}, ... \rangle
	\end{align}
	by augmenting the negative odd indices by $k$-many terms in the order $ \ds \underbrace{g_{2k-1}, g_{2k-3},....,g_{1}}_{k}. $
	
	Similar to the non-attenuated even $m$-tensor case before, we will construct the solution $u$ of the transport equation \eqref{NART_mOddTEQ}, whose trace matches the boundary data $g$, and also construct the right hand side of the \eqref{NART_mOddTEQ}.
	The construction of solution $u$ is in terms of its Fourier modes in the angular variable. 
	%We first construct the non-positive Fourier modes and then the positive Fourier modes are constructed by conjugation. 
	%we construct simultaneously the right hand side of the transport equation \eqref{NART_mOddTEQ} together with the solution $u$. Construction of $u$ is via its Fourier  modes. We first construct the negative modes and then the positive modes are constructed by conjugation. 
	%Similar to the even case, the characterization of the boundary data $g$ is in terms of the Bukhgeim-Hilbert Transform.  The solution $u$ of the transport equation \eqref{NART_mOddTEQ} is constructed via its Fourier modes. 
	%We will construct simultaneously the right hand side of the transport equation \eqref{NART_mOddTEQ} and the solution $u$ whose trace matches the boundary data $g$. Construction of $u$ is via its Fourier  modes. We first construct the negative modes and then the positive modes are constructed by conjugation. 
	Except for non-positive modes $u_{0}, u_{-2}, \cdots, u_{-2q}$, all non-positive modes are defined by Bukhgeim-Cauchy integral formula in \eqref{BukhgeimCauchyFormula} using boundary data. 
	Other than having the traces $u_{-2j} \big \lvert_{\Gam} = g_{-2j}, \; 0 \leq j \leq q, \; q \geq 0$, on the boundary, the $q+1$ many Fourier modes $u_{-2j}, \; 0 \leq j \leq q, \; q\geq 0$, are unconstrained. They are chosen arbitrarily from the class of functions
	\begin{align}\nonumber 
		\Psi_{g}^{\text{odd}}:=
		&\left \{ \left(\psi_0,  \psi_{-2},  \cdots ,\psi_{-2q}\right) \in 
		%  W^{2,p}(\OM; \BR)\times\left(W^{2,p}(\OM; \BC)\right)^{q}, 2 \mu >1:
		C^{1,\mu}(\ol\OM ;\BR)\times \left(C^{1,\mu}(\ol\OM ;\BC)\right)^{q} : 2 \mu >1:
		\right. \\  \label{NART_mOddPsiClass}
		&\left. \quad \vphantom{\int} 
		\psi_{-2j} \big \lvert_{\Gam}= g_{-2j}, \; 0 \leq j \leq q, \; q\geq 0 \right\}.
	\end{align} 
	
	\begin{remark}
		In the 1-tensor case ($m = 1$), only Fourier mode $u_0$ be an \emph{arbitrary} function in $C^1(\OM)\cap C(\ol\OM)$ with
		$u_0|_{\Gam}=g_0.$
		The arbitrariness of $u_0$ characterizes the non-uniqueness (up to the gradient field of a function which vanishes at the boundary) in the reconstruction of a vector field from its Doppler data.
	\end{remark}

	\begin{theorem}[Range characterization for odd tensors.]\label{NARTmOddTensor}
		Let $\bbf  \in C_{0}^{1,\mu}(\mathbf{S}^m ; \OM)$,  $\mu > 1/2$, 
		be a  real-valued symmetric  tensor field of odd order $m = 2q+1, \, q\geq 0$, and  
			$$g= X \bbf \text{ on }\Gam_+ \mbox{ and  }g= 0 \mbox{ on }\Gam_- \cup \Gam_{0}.$$
		Then  $\bg^{\emph{even}}, \bg^{2k-1} \in l^{1,1}_{\INF}(\Gamma)\cap C^\mu(\Gamma;l_1)$   for $k \geq 1$, and satisfy
		\begin{align}\label{RTmOddTensorCondEven}
			&[I+\i\HT]L^{\frac{m+1}{2}}\bg^{\emph{even}} = {\bf {0}}, \\ \label{RTmOddTensorCondOdd}
			&[I+\i\HT]\bg^{2k-1} = {\bf {0}}, \quad \forall k \geq 1,
		\end{align} 
	    where $\bg^{\emph{even}}$ is the sequence in     \eqref{gEvenOdd_mEven}, 
	    $\bg^{2k-1}$ for $k \geq 1$ is the sequence in \eqref{gM_mOdd_NART}, and $\HT$ is the Bukhgeim-Hilbert operator in \eqref{BHtransform}.
		
		(ii) Let $g\in C^{\mu} \left(\Gam; C^{1,\mu}(\sph) \right)\cap C(\Gam;C^{2,\mu}(\sph))$ be real valued with $g \lvert_{\Gam_{-} \cup \Gam_{0}}=0$.
		If the corresponding sequence $\bg^{\emph{even}} \in Y_{\mu} (\Gam)$ satisfies \eqref{RTmOddTensorCondEven}, $\bg^{2k-1} \in Y_{\mu} (\Gam)$ for $k \geq 1$, satisfies \eqref{RTmOddTensorCondOdd},
		and for each element $\left(\psi_0,  \psi_{-2},  \cdots ,\psi_{-2q}\right)  \in \Psi_{g}^{\emph{odd}}$, then there is a unique real valued symmetric $m$-tensor $\bbf_{\Psi} \in C^\mu(\mathbf{S}^m; \OM)  $ such that
		$g\lvert_{\Gam_{+}} = X\bbf_{\psi}$.
	\end{theorem}
	
	\begin{proof} (i) {\bf Necessity:}
		Let $\bbf =(f_{i_1 \cdots i_m})  \in  C_{0}^{1,\mu}(\mathbf{S}^m;\OM)$.
		Since all components $f_{i_1 \cdots i_m}\in C_{0}^{1,\mu}(\OM)$, 
		$X \bbf \in C^{1,\mu}(\Gam_+)$, and, thus, the solution $u$ to the transport equation \eqref{NART_mOddTEQ} is in $C^{1,\mu}(\ol{\OM} \times \sph)$. Moreover, its trace $g=u \lvert_{\Gam \times \sph} \in C^{1,\mu}(\Gam \times \sph)$.
%		 are compactly supported inside $\OM$, then for any point at the boundary there is a cone of lines which do not meet the support.
%		Thus $g \equiv 0$ in the neighborhood of the variety $\Gam_0$
%		which yields $g \in C^{1,\mu}(\Gam \times \sph)$. Moreover, $g$ is the trace on $\Gam \times \sph$ of a solution $u \in C^{1,\mu}(\ol{\OM} \times \sph)$ of the transport equation \eqref{NART_mOddTEQ}.
		By \cite[Proposition 4.1]{sadiqTamasan01} $\bg^{\text{even}}, \bg^{2k-1} \in l^{1,1}_{\INF}(\Gamma)\cap C^\mu(\Gamma;l_1)$ for all $k \geq 1$.
		
		If $u$ solves \eqref{NART_mOddTEQ} then its Fourier modes satisfy 
		\eqref{NART_mOdd_fmEq}, \eqref{NARTmOddAnalyticEq_Even}, and \eqref{NARTmOddAnalyticEq_Odd}. Since the negative even Fourier modes $u_{-2n}$ for $n \geq \frac{m+1}{2}$, satisfies the system \eqref{NARTmOddAnalyticEq_Even}, then
		the sequence valued map
		$$\OM \ni z\mapsto  \langle u_{-(m+1)}(z), u_{-(m+3)}(z), u_{-(m+5)}(z), \cdots \rangle$$
		is $L$-analytic in $\OM$ and the necessity part in Theorem  \ref{NecSuf_BukhgeimHilbert_Thm} yields the condition \eqref{RTmOddTensorCondEven}.
		
		The system \eqref{NARTmOddAnalyticEq_Odd}  
		yield that the sequence valued map
		\begin{align*}
		\OM \ni	z \mapsto \bu^{1}(z) := \langle u_{1}(z), u_{-1}(z), u_{-3}(z) \cdots \rangle
		\end{align*}
		is $L$-analytic in $\OM$ with the trace satisfying $\ds 			u_{2k-1} \lvert_{\Gam} = g_{2k-1}$, for all $ k\leq 1$. \\
		By  Theorem  \ref{NecSuf_BukhgeimHilbert_Thm} necessity part, the sequence $\bg^{1}=\langle g_{1}, g_{-1},g_{-3}, ... \rangle$ must satisfy
		\begin{align*}
			[I+\i\HT]\bg^1 =\bzero.
		\end{align*}
		
		Recall that $u$ is real valued so that its Fourier modes occur in conjugates $u_{n} = \ol{u_{-n}}$ for all $n \geq 0$.
		Consider now the equation \eqref{NARTmOddAnalyticEq_Odd} for $n= 1$ and take its conjugate to yield
		\begin{align}\label{delu3Eq_mOddTEQ}
			\ol{\del}u_{3}+ \del u_{1}=0.
		\end{align} 
	    Equation \eqref{delu3Eq_mOddTEQ} together with \eqref{NARTmOddAnalyticEq_Odd} yield that the sequence valued map
		\begin{align*}
		 \OM \ni	z \mapsto \bu^{3}(z) := \langle u_{3}(z), u_{1}(z), u_{-1}(z), u_{-3}(z) \cdots \rangle
		\end{align*}
		is $L$-analytic in $\OM$ with the trace satisfying 
		$\ds 	u_{2k-1} \lvert_{\Gam} = g_{2k-1}$ for all $ k\leq 2$. \\
%		\begin{align*}
%			u_{2k-1} \lvert_{\Gam} = g_{2k-1}, \quad k\leq 2.
%		\end{align*}
		By the necessity part in Theorem  \ref{NecSuf_BukhgeimHilbert_Thm}, it must be that $\bg^{3}=\langle g_{3},g_{1}, g_{-1},g_{-3}, ... \rangle$ satisfies
		\begin{align*}
			[I+\i\HT]\bg^3 =\bzero.
		\end{align*}
		Inductively, the argument above holds for any odd index $2k-1$ to yield that the sequence 
		\begin{align*}
			\OM \ni z \mapsto \bu^{2k-1}(z) := \langle u_{2k-1}(z),u_{2k-3}(z),...,u_{1}(z), u_{-1}(z), u_{-3}(z) \cdots \rangle
		\end{align*}
		is $L$-analytic in $\OM$. Then, again by the necessity part in Theorem \ref{NecSuf_BukhgeimHilbert_Thm}, its trace $\bu^{2k-1} \lvert_{\Gam} = \bg^{2k-1}$ must satisfy the condition 
		\eqref{RTmOddTensorCondOdd}:
		\begin{align*}
			[I+\i\HT]\bg^{2k-1} =\bzero, \quad \text{ for all } \; k \geq 1.
		\end{align*}
		
		(ii) {\bf Sufficiency:}
			Let $g\in C^{\mu} \left(\Gam; C^{1,\mu}(\sph) \right)\cap C(\Gam;C^{2,\mu}(\sph))$  be real valued with $g \lvert_{\Gam_{-} \cup \Gam_{0}}=0$.
		Since $g$ is real valued, its Fourier modes in the angular variable occurs in conjugates
		\begin{align}\label{reality_g1}
			g_{-n}(\zeta) = \ol{g}_{n}(\zeta), \quad \text{ for }  \; n \geq 0, \; \zeta \in \Gam.
		\end{align}
		Let the corresponding sequences  $\bg^{\text{even}}$  satisfying \eqref{RTmEvenTensorCondEven} and $\bg^{\text{odd}}$  satisfying \eqref{RTmEvenTensorCondOdd}. 
		By Proposition (\ref{functoseq_regularityprop}), $\bg^{\text{even}}, \bg^{\text{odd}}\in Y_{\mu}(\Gam)$.
		
		Let $m =2q+1, \; q\geq 0$, be an odd integer. 
		To prove the sufficiency we will construct a real valued symmetric $m$-tensor $\bbf$ in $\OM$ and a real valued function $u \in C^{1}(\OM\times \sph)\cap C(\ol\OM \times \sph)$
		such that $u \lvert_{\Gam \times \sph}=g$ and $u$ solves \eqref{NART_mOddTEQ} in $\OM$.
		The construction of such $u$ is in terms of its Fourier modes in the angular variable and it is done in several steps.
		
		{\bf Step 1: The construction of  even modes $u_{2n}$ for $|n| \geq 2q+1, \; q\geq 0$.}
			
		Apply the Bukhgeim-Cauchy integral formula \eqref{BukhgeimCauchyFormula} to
		construct the negative even Fourier modes:
		\begin{align}\label{construction_EVENS_mOdd_NART}
			\langle u_{-2(q+1)}, u_{-2(q+2)}, u_{-2(q+3)}, ..., \rangle := \B L^{q+1}\bg^{\text{even}}.
		\end{align}%where $\B$ is the operator in \eqref{BukhgeimCauchyFormula}.
		By Theorem \ref{NecSuf_BukhgeimHilbert_Thm}, the sequence valued map
		\begin{align*}
			\OM \ni z\mapsto \langle u_{-2(q+1)}(z), u_{-2(q+2)}(z), u_{-2(q+3)}(z),  ... \rangle \in C^{1,\mu}(\OM;l_1)\cap C^{\mu}(\ol \OM;l_1),
		\end{align*}is $L$-analytic in $\OM$, thus the equations
		\begin{align}\label{uEvenL1_mOddNART}
			% \ol{\del} u_{-2k} + \del u_{-2k-2} = 0,
			\ol{\del} u_{-2n} + \del u_{-(2n+2)} = 0, 
		\end{align} are satisfied for all $n \geq q+1, \, q \geq 0$.
		Moreover, the hypothesis \eqref{RTmOddTensorCondEven} and the sufficiency part of Theorem \ref{NecSuf_BukhgeimHilbert_Thm} yields that they extend continuously to $\Gam$ and
		\begin{align}\label{trace_uEven_mOddNART}
			u_{-2n}|_\Gam=g_{-2n},\quad n \geq q+1, \, q \geq 0.
		\end{align}
		
		Construct the positive even Fourier modes by conjugation:
		$ \ds u_{2n}:=\ol{u_{-2n}}$, for all $n \geq q+1, \, q \geq 0$.
		
%		Construct the positive even Fourier modes by conjugation:
%		\begin{align}\label{construct_even+_mOddNART}
%			u_{2n}:=\ol{u_{-2n}},\quad n \geq q+1, \, q \geq 0.
%		\end{align}
	By conjugating \eqref{uEvenL1_mOddNART} we note that the positive even Fourier modes also satisfy
		\begin{align}\label{uEvenL1+_mOddNART}
			\ol{\del} u_{2n+2} + \del u_{2n} = 0,\quad n \geq q+1, \, q \geq 0.
		\end{align}Moreover, by reality of $g$ in \eqref{reality_g1}, they extend continuously to $\Gam$ and
		\begin{align}\label{trace_uposEven_mOddNART}
			u_{2n}|_\Gam=\ol{u_{-2n}}|_\Gam=\ol{g_{-2n}}=g_{2n},\quad n \geq q+1, \, q \geq 0.
		\end{align}

		{\bf Step 2: The construction of even modes $u_{2n}$,  
			for $|n| \leq 2q, \; q\geq 0$.}
		
		Recall the non-uniqueness class $\Psi_{g}^{\text{odd}}$ in \eqref{NART_mOddPsiClass}. 
		
		For $\left(\psi_0,  \psi_{-2},  \cdots ,\psi_{-2q}\right) \in \Psi_{g}^{\text{odd}}$ arbitrary, define 
		the modes $u_0, u_{\pm 2}, u_{\pm 4}, ..., u_{\pm 2q}$ in $\OM$ by
		\begin{equation}\label{NART_mOdd_uPsi}
			\begin{aligned}
				u_{-2n} &:= \psi_{-2n}, \quad \text{and} \quad u_{2n} := \ol{\psi_{-2n}}, \quad 0 \leq n \leq q.
			\end{aligned}
		\end{equation}
		
		By the definition of the class \eqref{NART_mOddPsiClass}, and reality of $g$ in \eqref{reality_g1}, we have
		%Since the trace $\psi_{-2n}   \big \lvert_{\Gam}= g_{-2n}$ for  $0 \leq n \leq q, \; q\geq 0$,
		% the trace of $u_{-2n}$ as in \eqref{NART_mOdd_uPsi} satisfies
		% Since $g$ is real valued, we have 
		\begin{align}\label{u1trace_mOddNART}
			u_{2n} \lvert_{\Gam}= \ol{g_{-2n}}= g_{2n}, \quad 0 \leq n \leq q.
		\end{align}
		
		{\bf Step 3: The construction of negative modes $u_{2n-1}$ for $n \in \BZ$.}
		
		Use the Bukhgeim-Cauchy Integral formula \eqref{BukhgeimCauchyFormula} to 
		construct the negative odd Fourier modes:
		\begin{align}\label{constructionODDSnegative_mOddNART}
			\langle u_{-1}(z), u_{-3}(z), u_{-5}(z), ... \rangle := \B \bg^{\text{odd}}(z), \quad z\in \OM.
		\end{align}
		By Theorem \ref{NecSuf_BukhgeimHilbert_Thm}, the sequence valued map
		\begin{align*}
			\OM \ni	z\mapsto \langle u_{-1}(z), u_{-3}(z), u_{-5}(z), ... \rangle \in C^{1,\mu}(\OM;l_1)\cap C^{\mu}(\ol \OM;l_1),
		\end{align*}is $L$-analytic in $\OM$, thus the equations
		\begin{align} \label{negODDS_mOddNART}
			\ol{\del} u_{-2n-1} + \del u_{-2n-3} = 0,
		\end{align} are satisfied for all $n \geq 0$.
		
		%
		%Next we construct the odd negative indexed modes from
		%$$\bg^{\text{odd}}:=\langle g_{-1}, g_{-3},...\rangle$$
		%via the Bukhgeim-Cauchy integral formula \eqref{BukhgeimCauchyFormula} by
		%\begin{align}\label{constructionODDSnegative_mOddNART}
		%\bu^{\text{odd}}= \langle u_{-1}, u_{-3},  \cdots \rangle
		%:= \B \bg^{\text{odd}}.
		%\end{align}By Theorem \ref{NecSuf_BukhgeimHilbert_Thm} the sequence $\bu^{odd}$ is $L$-analytic, i.e.,
		%\begin{align} \label{negODDS_mOddNART}
		%\ol{\del} u_{-2n-1} + \del u_{-2n-3} &= 0, \quad \forall n\geq 0,
		%\end{align}
		Note that $L \bg^1=\bg^{\text{odd}}$. By hypothesis \eqref{RTmOddTensorCondOdd}, $[I+\i\HT] \bg^1 =\bzero$. Since $\HT$ commutes with the left translation $L$, then
		\begin{align*}
			\bzero=L[I+\i\HT] \bg^1 =[I+\i\HT]L \bg^1 =[I+\i\HT] \bg^{\text{odd}}.
		\end{align*} By applying Theorem \ref{NecSuf_BukhgeimHilbert_Thm}  sufficiency part,  
		we have that each $u_{2n-1}$ extends continuously to $\Gamma$:
		\begin{align*}
			u_{-2n-1}|_\Gam=g_{-2n-1},\quad n\geq 1.
		\end{align*}
		
		If we were to define the positive odd index modes by conjugating the negative ones (as we did for the non-attenuated even tensor case) it would not be clear why
		the equation \eqref{NARTmOddAnalyticEq_Odd} for $n=0$:
		\begin{align*}
			&\ol{\del} u_{1} + \del u_{-1} = 0, 
		\end{align*} should hold. To solve this problem we will define the positive odd modes by using the Bukhgeim-Cauchy integral formula \eqref{BukhgeimCauchyFormula} inductively.
		
		Let $\bu^1=\langle u_{1}, u^1_{-1}, u^1_{-3},  \cdots \rangle$ be the $L$-analytic map defined by
		\begin{align}\label{construct_u1_mOddNART}
			\bu^{1}:= \B \bg^1.
		\end{align}
		The hypothesis \eqref{RTmOddTensorCondOdd} for $k=1$,
		\begin{align*}
			[I+\i\HT] \bg^1 =\bzero,
		\end{align*} allows us to apply the sufficiency part of Theorem \ref{NecSuf_BukhgeimHilbert_Thm} to yield that $\bu^1$ extends continuously to $\Gam$ and has trace $\bg^{1}$ on $\Gam$. However, $L \bu^1 = \bu^{\text{odd}}$ is also $L$-analytic with the same trace $\bg^{\text{odd}}$ as $\bu^{\text{odd}}$. By the uniqueness of $L$-analytic maps with the given trace we must have the equality
		$$\langle  u^1_{-1}, u^1_{-3},  \cdots \rangle = \langle u_{-1}, u_{-3},  \cdots \rangle.$$
		In other words the formula \eqref{construct_u1_mOddNART} constructs only one new function $u_1$ and recovers the previously defined negative odd functions $u_{-1}, u_{-3},...$. In particular $\bu^1=\langle u_{1}, u_{-1}, u_{-3},  \cdots \rangle$ is $L$-analytic, and the equation $\ds \ol{\del}u_1+ \del u_{-1} =0$ 
		holds in $\OM$. We stress here that, at this stage, we do not know that $u_1$ is the complex conjugate of $u_{-1}$. %$u_1=\ol{u_{-1}}$.
		
		Inductively, for $k\geq 1$, the formula
		\begin{align}\label{define_u^2k-1_mOddNART}
			\bu^{2k-1}= \langle u_{2k-1}, u^{2k-1}_{2k-3}, ..., u^{2k-1}_1, u^{2k-1}_{-1}, \cdots \rangle:=\B \bg^{2k-1}
		\end{align} defines a sequence $\{\bu^{2k-1}\}_{k\geq1}$ of $L$-analytic maps with $\bu^{2k-1}\lvert_{\Gam} = \bg^{2k-1}$.
		By the uniqueness of $L$-analytic maps with the given trace, a similar reasoning as above shows
		\begin{align*}
			L \bu^{2k-1} = \bu^{2k-3}, \quad \forall k \geq 2.
		\end{align*}In particular for all $k \geq 1$, the sequence 
		\begin{align*}
			\bu^{2k-1} = \langle u_{2k-1}, u_{2k-3}, ..., u_1, u_{-1}, \cdots \rangle
		\end{align*} is $L$-analytic. 	
		Note that the sequence $\{\bu^{2k-1}\}_{k\geq 1}$ constructed above satisfies the hypotheses of the Lemma \ref{UconjLemma_mOdd_NART}, and therefore
		for each $k \geq 1$,
		\begin{align}\label{exceptional_mOddNART}
			u_{2k-1}(z) = \ol{u}_{-(2k-1)}(z), \quad z \in \OM.
		\end{align}
		We stress here that the identities \eqref{exceptional_mOddNART} need the hypothesis \eqref{RTmOddTensorCondOdd} for all $k\geq 1$, cannot be inferred directly from the Bukhgeim-Cauchy integral formula \eqref{BukhgeimCauchyFormula} for finitely many $k$'s.
		
		We have shown that
		\begin{equation}\label{PostiveOdds_mOddNART}
			\begin{aligned}
				\ol{\del} u_{2n-1} + \del{u_{2n-3}} &= 0, \quad \text{and} \quad u_{2n-1}|_\Gamma = g_{2n-1}, \quad \forall n\in \mathbb{Z}.
			\end{aligned}
		\end{equation}

		{\bf Step 4: The construction of the tensor field $\bbf_{\psi} $ whose $X$-ray data is $g$.}
		
		The components of the $m$-tensor $\bbf_{\Psi} $ are defined via the one-to-one correspondence between  the pseudovectors $\langle \tilde{f}_0, \tilde{f}_{1}, \cdots , \tilde{f}_{m} \rangle $  and the functions  $\{f_{\pm(2n+1)}:\; 0\leq n\leq q\}$ as follows. 
		
		For $q\geq 0$, we define $f_{2q+1}$  by using $\psi_{-2q}$ from the non-uniqueness class in \eqref{NART_mOddPsiClass}, and Fourier mode $u_{-(2q+2)}$ from the Bukhgeim-Cauchy formula \eqref{construction_EVENS_mOdd_NART}. Then, define $\{f_{2n+1}:\; 0 \leq n \leq q-1\}$ solely from the information in the non-uniqueness class. Finally, define $\{f_{-(2n+1)}:\; 0\leq n\leq q\}$ by conjugation. 
		%From $\left(\psi_0,  \psi_{-2},  \cdots ,\psi_{-2q}\right) \in \Psi_{g}^{\text{odd}}$ as in \eqref{NART_mOddPsiClass}, and Fourier mode $u_{-2(q+1)}$ in \eqref{construction_EVENS_mOdd_NART}, 
		%we define first $f_{2q+1}$ depending on $m=2q+1,\; q\geq0$, then define $f_{2q-1}$ depending on $ q\geq1$, and then continue as follows:		
		\begin{equation} \label{NART_mOdd_fmPsiEq}
			\begin{aligned} %\label{NART_mEven_fzeroPsiEq}
				&f_{2q+1}:=\ol{\del} \psi_{-2q} + \del u_{-(2q+2)},   && q\geq 0,\\
				&f_{2n+1}:=\ol{\del} \psi_{-2n} + \del \psi_{-(2n+2)}, && 0 \leq n \leq q-1,\; q\geq 1, \quad \text{and} \\
				&f_{-(2n+1)}:=\ol{f_{2n+1}}, &&0 \leq n \leq q, \; q\geq 0,
			\end{aligned}
		\end{equation}
		By construction, $f_{\pm(2n+1)}\in C^\mu(\OM)$, for $0\leq n\leq q$, as $\psi_{0},\psi_{-2}, \cdots,\psi_{-2q}\in C^{1,\mu}(\OM)$.
		We use these Fourier modes $ f_{\pm 1}, f_{\pm 3}, \cdots , f_{\pm m}$ for $m =2q+1, \; q\geq 0$, and equations \eqref{eq:Oddmodes_fk},  ~\eqref{reality_cond} and \eqref{eq:tilde_fk} to construct the 
		pseudovectors $\langle \tilde{f}_0, \tilde{f}_{1}, \cdots , \tilde{f}_{m} \rangle $, and thus the $m$-tensor field $\bbf_{\Psi}  \in C^\mu(\mathbf{S}^m; \OM)  $.

		In order to show $g \lvert_{\Gam_{+}} = X\bbf_{\Psi}$ with $\bbf_{\Psi}$ being constructed from pseudovectors via  Fourier modes as in \eqref{NART_mOdd_fmPsiEq} from class $\Psi_{g}^{\text{odd}}$, 
		we define the real valued function $u$ via its Fourier modes\\
		\begin{equation}\label{NART_mOdd_defnUPsi1}
			\begin{aligned}
				\footnotesize u(z, \btheta):= \sum_{n = -\INF}^{\INF} u_{2n-1}(z)e^{\i (2n-1)\tta} +\sum_{|n| \geq q+1} u_{2n}(z)e^{\i 2n\tta} 
				+ \sum_{n=0}^{q} \psi_{-2n}(z)e^{ -\i 2n\tta} + \sum_{n = 0}^{q} \ol{\psi}_{-2n}(z)e^{\i 2n \tta}. 
			\end{aligned} 
		\end{equation} 
		Since $g\in C^{\mu} \left(\Gam; C^{1,\mu}(\sph) \right)\cap C(\Gam;C^{2,\mu}(\sph))$, we use Proposition \ref{functoseq_regularityprop} (ii) and \cite[Proposition 4.1 (iii)]{sadiqTamasan01} to conclude that $u$ defined in \eqref{NART_mOdd_defnUPsi1} belongs to $C^{1,\mu}(\OM \times \sph)\cap C^{\mu}(\ol{\OM}\times \sph)$. 
		
		Using \eqref{trace_uEven_mOddNART}, \eqref{trace_uposEven_mOddNART}, \eqref{u1trace_mOddNART}, \eqref{PostiveOdds_mOddNART}, and element $\left(\psi_0,  \psi_{-2},  \cdots ,\psi_{-2q}\right) \in \Psi_{g}^{\text{odd}}$, the
		$u(\cdot,\btheta)$ in \eqref{NART_mOdd_defnUPsi1} extends to the boundary
		\begin{align*}
			u(\cdot,\btheta)\lvert_{\Gam} = g(\cdot,\btheta),
		\end{align*}
		
		%Since $u\in C^{1,\mu}(\OM \times \sph)\cap C^{\mu}(\ol{\OM}\times \sph)$ in \eqref{NART_mOdd_defnUPsi1}, calculation shows that for $q \geq 0:$ 
		
		Since $u\in C^{1,\mu}(\OM \times \sph)\cap C^{\mu}(\ol{\OM}\times \sph)$, 
		then the term by term differentiation in \eqref{NART_mOdd_defnUPsi1} is now justified,
		satisfying the transport equation \eqref{NART_mOddTEQ}:
		\begin{align*}
			\btheta \cdot \nabla u &= 2 \re \left \{ (\ol{\del} \psi_{-2q}+ \del u_{-(2q+2)})e^{\i (2q+1) \tta} \right\}+ 2 \re \left \{ \sum_{n = 0}^{q-1} (\ol{\del} \psi_{-2n}+ \del \psi_{-(2n+2)})e^{\i (2n+1) \tta}\right\}\\
			% &\quad + \sum_{n=1}^{q-1} (\ol{\del} \psi_{1-2n} + \del \psi_{-1-2n} )e^{ -\i (2n)\tta} + \sum_{n=1}^{q-1} (\ol{\del} \; \ol{\psi_{-1-2n}} +\del \ol{\psi_{1-2n}})e^{ \i (2n)\tta} \\
			% &\quad + e^{-\i (2q) \tta} (\ol{\del}\psi_{-(2q-1)}+\del u_{-(2q+1)})   + e^{\i (2q) \tta} (\del \ol{\psi_{-(2q-1)}} + \ol{\del} \; \ol{u_{-(2q+1)}}) \\
%			&\quad + (\del \ol{\psi}_{-2q}+ \ol{\del u}_{-2(q+1)})e^{-\i (2q+1) \tta} +\sum_{n = 0}^{q-1} (\del \ol{\psi}_{-2n}+ \ol{\del \psi}_{-(2n+2)})e^{-\i (2n+1) \tta}
			% \\&\quad + \sum_{|n| \geq  2q+1} ( \ol{\del}u_{-n}+\del u_{-(n+2)}) e^{ - \i (n+1) \tta}.
%			\\
			&=  \sum_{n=0}^{q} \left( f_{2n+1} e^{- \i (2n+1) \tta} + f_{-(2n+1)}e^{ \i (2n+1) \tta} \right) = \langle \bbf,  \btheta^{2q+1} \rangle,
		\end{align*}
		where the cancellation uses equations \eqref{uEvenL1_mOddNART}, \eqref{uEvenL1+_mOddNART}, \eqref{PostiveOdds_mOddNART},
		and the second equality uses the definition of $f_{2k+1}$'s in \eqref{NART_mOdd_fmPsiEq}.

	\end{proof}

	%%%%%%%%%%%%%%%%%%%%%%%%%%%%%%%%%%%%%%%%%%%%%
	%%%%%%%%%%%%%%%%%% ATTENUATED %%%%%%%%%%%%%%%%%%%
	%%%%%%%%%%%%%%%%%%%%%%%%%%%%%%%%%%%%%%%%%%%%%

	\section{Even order $m$-tensor - attenuated case}\label{sec:RangeART_mEven}
	%\section{$m$-tensor (Even Case $\geq 0$)-attenuated case}
	Let $a\in C^{2,\mu}(\ol \OM)$, $\mu>1/2$, with $\ds \underset{\ol{\OM}}{\min}\, a > 0$. 
	We now establish necessary and sufficient conditions for a sufficiently smooth function on $\Gam \times \sph$ to be the attenuated $X$-ray data of some sufficiently smooth real valued symmetric tensor field $\bbf$ of even order $m=2q, \; q \geq 0$.
	In this case $a \neq 0$, the transport equation \eqref{TransportEq2Tensor} becomes
	\begin{align}\label{eq:Transport_fmEven}
		\btheta \cdot\nabla u(x,\btheta) + a(x) u(x,\btheta) = 
		\sum_{k=0}^{q}f_{-2k} e^{\i (2k)\tta}  + \sum_{k=1}^{q} f_{2k} e^{-\i (2k)\tta},
	\end{align} 
%where $f_{2k}=\ol{f_{-2k}}$, for $-q \leq k \leq q$.
where $f_{2k}$ defined in \eqref{eq:Evenmodes_fk}, and $ f_{2k} = \ol{f_{-2k}}, \;$ for 
$ -q \leq k \leq q, \, q \geq 0$.  
%Note that $f_0$ is real-valued while other modes are complex conjugates.
	
	If $\ds \sum_{n \in \BZ} u_{n}(z) e^{\i n\tta}$ is the Fourier series  expansion in the angular variable $\btheta$ of a solution $u$ of \eqref{eq:Transport_fmEven}, then by identifying the Fourier coefficients of the same order, equation \eqref{eq:Transport_fmEven} reduces to the system:
	\begin{align} \label{ART_mEven_fmEq}
		&\ol{\del} u_{-(2n-1)}(z) + \del u_{-(2n+1)}(z) +a u_{-2n}(z) = f_{2n}(z), && 0 \leq n \leq q, \; q \geq 0, \\ 
		%%%%%%%%%%%%%%%%%%%%
		\label{eq:Beltrami_mEven_un<mby2}
		&\ol{\del} u_{-2n}(z) + \del u_{-(2n+2)}(z) +a u_{-2n-1}(z) = 0, &&  0 \leq n \leq q-1, \; q \geq 1, \\ 
		\label{BeltramiEq_mEven_un>=m}
		&\ol{\del} u_{-n}(z) + \del u_{-(n+2)}(z) +a u_{-(n+1)}(z)= 0, && n \geq 2q, \;q \geq 0.
	\end{align}

	Recall that the trace $u \lvert_{\Gam \times \sph} := g$ as in \eqref{g_trace}, with $g= X_a \bbf \text{ on }\Gam_+ \mbox{ and  }g= 0 \mbox{ on }\Gam_- \cup \Gam_{0}.$ \\
	We expand the attenuated $X$-ray data $g$ in terms of its Fourier modes in the angular variables:
	\begin{align*}%\label{FourierData}
		g(\zeta,\btheta) = \sum_{n=-\infty}^{\infty} g_{n}(\zeta) e^{\i n \tta}, \quad \zeta \in \Gam.
	\end{align*}Since the trace $g$ is also real valued, its Fourier modes will satisfy
	$\ds g_{-n}= \ol{g}_{n},$ for $ n \geq 0.$
	From the negative modes, we built the sequence
	%\begin{align}\label{gEvenOdd}
	$\ds  \bg:=\langle g_{0}, {g}_{-1}, g_{-2}, g_{-3},... \rangle.$
	%\end{align} 
	From the special function $h$ defined in \eqref{hDefn} and the data $g$, we built the sequence
	\begin{align*}%\label{g_h}
		\bg_h:= e^{-G} \bg:= \langle \gamma_{0}, \gamma_{-1}, \gamma_{-2}, ... \rangle,
	\end{align*} where $e^{\pm G}$ as defined in \eqref{eGop}. 
	From the negative even, respectively, negative odd Fourier modes, we built the sequences
	\begin{align}\label{gHEvenOdd}
		\bg^{\text{even}}_h = \langle \gamma_{0}, \gamma_{-2},\gamma_{-4},...\rangle, \quad \text{and} \quad
		\bg^{\text{odd}}_h = \langle \gamma_{-1}, \gamma_{-3},\gamma_{-5},...\rangle.
	\end{align}
	
	Next we characterize the attenuated $X$-ray data $g$ in terms of its Fourier modes $\underbrace{ g_{0}, g_{-1}, g_{-2}, \cdots g_{-(m-1)}}_{m}$, and the Fourier modes 
	\begin{align*}%\label{Lmg_h}
		L^{m}\bg_h:= L^{m}e^{-G} \bg:= \langle \gamma_{-m}, \gamma_{-(m+1)}, \gamma_{-(m+2)}, ... \rangle.
	\end{align*} 

	Similar to the non-attenuated case as before, we construct simultaneously the right hand side of the transport equation \eqref{eq:Transport_fmEven} together with the solution $u$ via its Fourier  modes. 
	%	Similar to the non-attenuated case in previous sections, we construct the solution $u$ (via its Fourier  modes) of the transport equation \eqref{eq:Transport_fmEven}, whose trace matches the boundary data $g$, and also construct the right hand side of the \eqref{eq:Transport_fmEven}.
%	The construction of solution $u$ is in terms of its Fourier modes in the angular variable. 
		%As before we construct simultaneously the right hand side of the transport equation \eqref{eq:Transport_fmEven} together with the solution $u$. Construction of $u$ is via its Fourier  modes. We first construct the negative modes and then the positive modes are constructed by conjugation. 
	For $m =2q$, $q \geq 1$, apart from modes $\underbrace{ u_{0}, u_{-1}, u_{-2}, \cdots u_{-(2q-1)}}_{2q}$, all Fourier modes are constructed uniquely from the data $L^{2q}\bg_h$.
	The modes $ u_{0}, u_{-2}, u_{-4}, \cdots u_{-(2q-2)}$ will be chosen arbitrarily from the class $\Psi_{a,g}^{\text{even}}$ of cardinality $q=\frac{m}{2}$ with prescribed trace and gradient on the boundary $\Gam$ defined as
	%\begin{equation}%\label{ART_mEvenPsiClass}
	\begin{align} \nonumber
		\Psi_{a,g}^{\text{even}}:=
		&\left \{ \left(\psi_0,  \psi_{-2},  \cdots ,\psi_{-2 (q-1)}\right) \in C^{2}(\ol\OM ;\BR)\times\left(C^{2}(\ol\OM ;\BC))\right)^{q}: 
		\right. \\  \nonumber
		&\left. \quad \vphantom{\int}
		\psi_{-2j} \big \lvert_{\Gam}= g_{-2j}, \quad 0 \leq j \leq q-1, \; q \geq 1,
		\right. \\  \nonumber
		&\left. \quad \vphantom{\int}  \ol{\del} \psi_{-2(q-1)}  \big \lvert_{\Gam} = - \del (e^G \B e^{-G} \bg )_{-2q}  \big \lvert_{\Gam}  -  a \big \lvert_{\Gam} \,  g_{-(2q-1)},  \quad q \geq 1,
		\right. \\  \label{ART_mEvenPsiClass}
		&\left. \quad \vphantom{\int} 
		\ol{\del} \psi_{-2j}  \big \lvert_{\Gam} = - \del \psi_{-(2j+2)} \big\lvert_{\Gam}  -  a \big \lvert_{\Gam} \,   g_{-(2j+1)}, \quad  0 \leq j \leq q-2, \;q \geq 2
		\right \}
	\end{align}
	%\end{equation}
	where $\B$ be the Bukhgeim-Cauchy operator in  \eqref{BukhgeimCauchyFormula}, and the operators $e^{\pm G}$ as defined in \eqref{eGop}.
	
	\begin{remark}
		In the 2-tensor case ($m = 2$), apart from zeroth mode $u_0$ and negative one mode $u_{-1}$, all Fourier modes are constructed uniquely from the data $L^{2}\bg_h$.
		The mode $u_0$ will be chosen arbitrarily from the class $\Psi_{a,g}^{m=2}$. We rewrite the above class $\Psi_{a,g}^{\emph{even}}$ explicitly for $m = 2$, as
		\begin{align}\label{ART_m2PsiClass}
			\Psi_{a,g}^{m=2}:=
			&\left \{  \psi_0 \in C^{2}(\ol\OM ;\BR):  \psi_{0} \big \lvert_{\Gam}= g_{0}, 
			\quad \vphantom{\int}  \ol{\del} \psi_{0}  \big \lvert_{\Gam} = - \del (e^G \B e^{-G} \bg )_{-2}  \big \lvert_{\Gam}  -  a \lvert_{\Gam} \,  g_{-1} 
			\right \}.
		\end{align}  
		In the 0-tensor case ($m = 0$), there is no class, and the characterization of the attenuated $X$-ray data $g$ is in terms of the Fourier modes $\bg_h:= e^{-G} \bg$. 
	\end{remark}

	Next, we characterize the range for even $m =2q$, $q \geq 0$, in the attenuated case.
	\begin{theorem}[Range characterization for even order tensors ]\label{ARTmEvenTensor}
		
		Let $a \in C^{2,\mu}(\ol\OM)$, $\mu>1/2$ with $\underset{\ol{\OM}}{\min}\, a >0$. 
		(i) Let $\bbf  \in C_{0}^{1,\mu}(\mathbf{S}^m ; \OM)$,  
		be a  real-valued symmetric  tensor field of even order $m = 2q, \, q\geq 0$, and 
		$g= X_a \bbf \text{ on }\Gam_+ \mbox{ and  }g= 0 \mbox{ on }\Gam_- \cup \Gam_{0}.$	
		Then $\bg^{\emph{even}}_h,  \bg^{\emph{odd}}_h \in l^{1,1}_{\INF}(\Gamma)\cap C^\mu(\Gamma;l_1)$ satisfy
		\begin{align}\label{ARTmEvenTensorCondEvenOdd}
			&[I+\i \HT] L^{\frac{m}{2}}
			\bg^{\emph{even}}_h  =\bzero, \quad [I+\i \HT] L^{\frac{m}{2}}  \bg^{\emph{odd}}_h =\bzero. %\quad \text{and} 
		\end{align} 
		where $\bg^{\text{even}}_h,  \bg^{\emph{odd}}_h$ are sequences in \eqref{gHEvenOdd}, and $\HT$ is the Bukhgeim-Hilbert operator in \eqref{BHtransform}.
		
		(ii) Let $g\in C^{\mu} \left(\Gam; C^{1,\mu}(\sph) \right)\cap C(\Gam;C^{2,\mu}(\sph))$ be real valued with $g \lvert_{\Gam_{-} \cup \Gam_{0}}=0$.
		For $q=0$, if the corresponding sequences $\bg^{\emph{even}}_h,  \bg^{\emph{odd}}_h \in Y_{\mu}(\Gam)$ satisfies \eqref{ARTmEvenTensorCondEvenOdd}, 
		then there is a unique real valued symmetric $0$-tensor $\bbf$ such that $g\lvert_{\Gam_{+}} = X_a\bbf$.
		Moreover, for $q \geq 1$, if $\bg_h^{\emph{even}}, \bg_h^{\emph{odd}}\in Y_{\mu}(\Gam)$ satisfies \eqref{ARTmEvenTensorCondEvenOdd}, and for each element $ \left(\psi_0,  \psi_{-2},  \cdots ,\psi_{-2 (q-1)}\right)   \in \Psi_{a,g}^{\emph{even}}$,
		then there is a unique real valued symmetric $m$-tensor $\bbf_{\Psi} \in C(\mathbf{S}^m; \OM)  $ such that $g\lvert_{\Gam_{+}} = X_a\bbf_{\Psi}$.
	\end{theorem}

	\begin{proof}
		(i) {\bf Necessity:} Let $\bbf =(f_{i_1 \cdots i_m})  \in  C_{0}^{1,\mu}(\mathbf{S}^m;\OM)$.
		Since all components $f_{i_1 \cdots i_m}\in C_{0}^{1,\mu}(\OM)$ are compactly supported inside $\OM$, then for any point at the boundary there is a cone of lines which do not meet the support.
		Thus $g \equiv 0$ in the neighborhood of the variety $\Gam_0$  which yields $g \in C^{1,\mu}(\Gam \times \sph)$. Moreover, $g$ is the trace on $\Gam \times \sph$ of a solution $u \in C^{1,\mu}(\ol{\OM} \times \sph)$ of the transport equation
		\eqref{eq:Transport_fmEven}.
		By Proposition \ref{functoseq_regularityprop}(i) and Proposition \ref{eGprop}, 
		$\bg_h= e^{-G} \bg \in l^{1,1}_{\INF}(\Gamma)\cap C^\mu(\Gamma;l_1)$. 
				
		If $u$ solves \eqref{eq:Transport_fmEven} then its Fourier modes satisfies \eqref{ART_mEven_fmEq}, \eqref{eq:Beltrami_mEven_un<mby2} and \eqref{BeltramiEq_mEven_un>=m}.
		In particular, the sequence valued map $\ds  \bu:=\langle u_{0}, u_{-1}, u_{-2},  \cdots \rangle, $  satisfies 
		$\ds \dba L^m \bu +L^2 \del L^m \bu+ aL^{m+1}\bu = 0$. 
		%where $L^m \bu =\langle u_{-m}, u_{-m-1}, u_{-m-2},  \cdots \rangle$ .
		
		Let $\ds \bv := e^{-G} L^m \bu $, then by Lemma \ref{beltrami_reduction}, and the fact that the operators $e^{\pm G}$ commute with the left translation, $ [e^{\pm G}, L]=0$, 
		the sequence $\ds \bv = L^m e^{-G}  \bu$ solves $\dba \bv + L^2\del \bv =0$, i.e $\bv$ is $L^2$ analytic.
		Thus, the negative even subsequence $\langle v_{0}, v_{-2},   \cdots \rangle$, and negative odd subsequence $\langle v_{-1}, v_{-3},  \cdots \rangle$, respectively, are $L$ analytic, with traces 
		$L^{\frac{m}{2}}\bg_{h}^{\text{even}}$, respectively, $L^{\frac{m}{2}}\bg_{h}^{\text{odd}}$.
		The necessity part in Theorem  \ref{NecSuf_BukhgeimHilbert_Thm} yields \eqref{ARTmEvenTensorCondEvenOdd}:
		\begin{align*}
			&[I+\i\HT] L^{\frac{m}{2}} \bg_{h}^{\text{even}} =\bzero, \quad [I+\i\HT] L^{\frac{m}{2}} \bg_{h}^{\text{odd}} =\bzero. 
		\end{align*} This proves part (i) of the theorem.
		
		(ii) {\bf Sufficiency:} 
		Let $g\in C^{\mu} \left(\Gam; C^{1,\mu}(\sph) \right)\cap C(\Gam;C^{2,\mu}(\sph))$ be real valued
		with $g \lvert_{\Gam_{-} \cup \Gam_{0}}=0$.
		Let the corresponding sequences $\bg_h^{\text{even}}, \bg_h^{\text{odd}}$ as in \eqref{gHEvenOdd} satisfying \eqref{ARTmEvenTensorCondEvenOdd}. By Proposition \ref{functoseq_regularityprop}(ii) and Proposition \ref{eGprop}(iii), we have $\bg_h^{\text{even}}, \bg_h^{\text{odd}}\in Y_{\mu}(\Gam)$.
		
		Let $m =2q$, $q \geq 0$, be an even integer.
		To prove the sufficiency we will construct a real valued symmetric $m$-tensor $\bbf$ in $\OM$ and a real valued function $u \in C^{1}(\OM\times \sph)\cap C(\ol\OM \times \sph)$
		such that $u \lvert_{\Gam \times \sph}=g$ and $u$ solves \eqref{eq:Transport_fmEven} in $\OM$.
		The construction of such $u$ is in terms of its Fourier modes in the angular variable and it is done in several steps.

		{\bf Step 1: The construction of modes $u_{-n}$ for $|n| \geq 2q, \; q\geq 0$.}

		Use the Bukhgeim-Cauchy Integral formula \eqref{BukhgeimCauchyFormula} to define the $L$-analytic maps
		\begin{align*}
			&\bv^{\text{even}}(z)= \langle v_{0}(z), v_{-2}(z), v_{-4}(z), ... \rangle:= \B L^{q}\bg_{h}^{\text{even}}(z), \quad z\in \OM,\\
			&\bv^{\text{odd}}(z)= \langle v_{-1}(z), v_{-3}(z), v_{-5}(z), ... \rangle:= \B L^{q}\bg_{h}^{\text{odd}}(z),\quad z\in \OM.
		\end{align*}By intertwining the above $L$-analytic maps, define also the $L^2$-analytic map
		\begin{align*}
			\bv(z):=\langle v_{0}(z), v_{-1}(z), v_{-2}(z),v_{-3}(z), ... \rangle , \quad z\in \OM.
		\end{align*}
		By Theorem \ref{NecSuf_BukhgeimHilbert_Thm} (ii),
		\begin{align}\label{smothness__v_j}
			\bv, \bv^{\text{even}}, \bv^{\text{odd}}\in C^{1,\mu}(\OM; l_{1})\cap C^{\mu}(\ol\OM;l_1)\cap C^2(\OM;l_\infty).
		\end{align}Moreover, since $\bg_h^{\text{even}}, \bg_h^{\text{odd}}$ satisfy the hypothesis \eqref{ARTmEvenTensorCondEvenOdd}, by  Theorem  \ref{NecSuf_BukhgeimHilbert_Thm} sufficiency part, we have
		\begin{align*}
			\bv^{\text{even}} \lvert_{\Gam} = L^{q}\bg_h^{\text{even}} \quad \text{and}\quad  \bv^{\text{odd}} \lvert_{\Gam} = L^{q}\bg^{\text{odd}}_h.
		\end{align*} 
		In particular, $\bv$ is $L^2$-analytic map with trace:
		\begin{align}\label{vn_intermsof_gn}
			\bv \lvert_{\Gam} = L^{2q}\bg_h = L^{2q} e^{-G} \bg,  
		\end{align} where $\bg_h$ is formed by intertwining $\bg_h^{\text{even}} $ and $\bg_h^{\text{odd}}$.
		
		Define the sequence valued map
		\begin{align}\label{eq:LmU_eGv}
			&\OM \ni z\mapsto  L^{2q} \bu(z)=\langle u_{-2q}(z), u_{-2q-1}(z), u_{-2q-2}(z),  \cdots \rangle := e^{G} \bv (z),
		\end{align} where the operator $e^{ G}$ as defined in \eqref{eGop}.
		Since convolution preserves $l_1$, by Proposition \ref{eGprop}, 
		\begin{align}\label{eq:LmU_regularity}
			L^{2q} \bu \in C^{1,\mu}(\OM; l_{1})\cap C^{\mu}(\ol\OM;l_1).
		\end{align}
		Moreover, since $\bv\in C^2(\OM;l_\infty)$ as in \eqref{smothness__v_j}, we also conclude from convolution that $L^{2q} \bu \in C^{2}(\OM; l_\infty).$
		
		As $\bv$ is $L^2$ analytic, by Lemma \ref{beltrami_reduction}, $L^{2q} \bu$ satisfies 
		\begin{align*}
			\dba L^{2q} \bu +L^2 \del L^{2q} \bu+ aL^{2q+1}\bu = 0,
		\end{align*} which in component form is written as:
		\begin{align}\label{uL2sys_mEven_a>0}
			\ol{\del} u_{-n} + \del u_{-n-2} +a u_{-n-1} = 0, \quad n\geq 2q, \; q\geq 0.
		\end{align}
		The trace satisfy
		\begin{align}\label{Lmu_trace}
			L^{2q} \bu \lvert_{\Gam}    = e^G \bv \lvert_{\Gam} = e^G L^{2q} e^{-G} \bg = L^{2q} \bg, 
		\end{align} where the second equality follows from \eqref{vn_intermsof_gn} and in the last equality  we use the fact that the operators $e^{\pm G}$ commute with the left translation, $ [e^{\pm G}, L]=0$.

		Construct the positive Fourier modes by conjugation: 
		$ \ds u_{n}:=\ol{u_{-n}},$ for all $n\geq 2q, \; q\geq 0$. Moreover using \eqref{Lmu_trace}, the traces $u_n \lvert_{\Gam}$ for each $n \geq 2q, \, q\geq 0$, satisfy
		\begin{align}\label{Trace_un_pos_gn}
			u_{n}\lvert_{\Gam} = \ol{u_{-n}}\lvert_{\Gam}= \ol{g_{-n}}= g_{n}, \quad n \geq 2q, \; q\geq 0.
		\end{align}
		By conjugating \eqref{uL2sys_mEven_a>0} we note that the positive Fourier modes also satisfy
		\begin{align}\label{uL2sys_mEven_a>0_Pos}
			\ol{\del} u_{n+2} + \del u_{n} +a u_{n+1} = 0, \quad n\geq 2q, \; q\geq 0.
		\end{align}
		
		%%%%%%%%%%%%%%%%%%%%%%%%
		{\bf Step2: The construction of the tensor field $\bbf$ in the $q=0$ case.}
		
		In the case of the 0-tensor, $\bbf=f_0$, and $f_0$ is uniquely determined   from the odd Fourier mode $u_{-1}$, and the zeroth mode $u_{0}$ in  \eqref{eq:LmU_eGv}, by
		\begin{align}\label{ART_0tensor_reconstruction}
			\bbf:= 2 \re \del u_{-1} +a u_0,  \quad (\text{for }q= 0 \text{ case}).
		\end{align} 
%		and 
%		$\displaystyle u(z, \btheta):= \sum_{n\in\mathbb{Z}} u_{n}(z)e^{\i (n\tta)} $  solves \eqref{eq:Transport_fmEven}. Using  \eqref{Trace_un_pos_gn}, $u$ defined above  satisfies $\ds u(\cdot,\btheta)\lvert_{\Gam}= g(\cdot,\btheta).$ Thus $g \lvert_{\Gam_{+}} = X_a\bbf$ for $q=0.$

		We consider next the case $m =2q, q \geq 1$ of tensors of order 2 or higher. 
		In this case the construction of the tensor field $\bbf_{\psi} $  is in terms of the  mode $u_{-2q}$ in \eqref{eq:LmU_eGv}  and the class $\Psi_{a,g}^{\text{even}}$ in \eqref{ART_mEvenPsiClass}.
		
		%The remaining steps require $q \geq 1$, and the construction of the tensor field $\bbf_{\psi} $  is in terms of the Fourier mode $u_{-2q}$ in \eqref{eq:LmU_eGv}  and the class $\Psi_{a,g}^{\text{even}}$ in \eqref{ART_mEvenPsiClass}.

		%%%%%%%%%%%%%%%%%%%%%%%%%%%%%%%%%%%%%%%%%%%%%%
		%\commentK{Change it}

		{\bf Step 3: The construction of modes $u_{n}$ for $|n| \leq 2q-1\; q \geq 1$.}%$\underbrace{u_{-(m-1)}, u_{-(m-2)}, \cdots ,u_{-1}, u_0, u_1, \cdots , u_{m-2}, u_{m-1}}_{2m-1}$.}
		
		%Recall that data $g$ is real valued, so its Fourier modes satisfy $\ds g_{2n-1} = \ol{g}_{-(2n-1)}$, for $1 \leq n \leq q. $
		Recall that $a \in C^{2,\mu}(\ol\OM)$, $\mu>1/2$ with $\underset{\ol{\OM}}{\min}\, a >0$, and  the non-uniqueness class $\Psi_{a,g}^{\text{even}}$  in \eqref{ART_mEvenPsiClass}.
		
		For $ \left(\psi_0,  \psi_{-2},  \cdots ,\psi_{-2 (q-1)}\right) \in \Psi_{a,g}^{\text{even}}$ arbitrary, define the modes $u_0, u_{\pm 2},  ..., u_{\pm (2(q-1))}$ in $\OM$ by
		\begin{equation}\label{defn_uminus_2j}
			\begin{aligned}
				u_{-2j}&:= \psi_{-2j},       \quad \text{and} \quad u_{2j} := \ol{\psi_{-2j}},  \quad 0 \leq j \leq q-1, \; q \geq 1.
			\end{aligned} 
		\end{equation}

		Using the mode $u_{-2q}$ from \eqref{eq:LmU_eGv} and $\psi_{-2(q-1)} $, define the modes $u_{\pm(2q-1)} $ by
		\begin{equation} \label{defn_u_oneminusm}
			\begin{aligned}
				u_{-(2q-1)}&:= -\frac{\ol{\del}\psi_{-2(q-1)}+\del u_{-2q}}{a}, \quad  \text{and} \quad 				u_{2q-1} := \ol{u}_{-(2q-1)}, \; \text{ for all } q \geq 1.
			\end{aligned} 
		\end{equation}
		% From $\psi_{-2j} \in \Psi_{a,g}^{\text{even}}$, for  $0 \leq j \leq \frac{m}{2}-2$, 
		As $\psi_0 \in C^{2}(\ol\OM ;\BR)$ and $\psi_{-(2j+2)} \in C^{2}(\ol\OM;\mathbb{C})$, for  $0 \leq j \leq q-2, \; q \geq 2$, define modes
		\begin{equation}\label{defn_uminus_2j+1}
			\begin{aligned}
				u_{-(2j+1)}&:= -\frac{\ol{\del}\psi_{-2j}+\del \psi_{-(2j+2)}}{a},\, \text{ and }  \, 				u_{2j+1} := \ol{u}_{-(2j+1)}, \; \text { for all } 0 \leq j \leq q-2, \; q \geq 2.
			\end{aligned} 
		\end{equation}
		%\commentK{Make sure the above $q$ are ok.}
		% Define the mode $u_{-(m-1)} $ by
		By the construction in \eqref{defn_uminus_2j}, \eqref{defn_u_oneminusm}, and \eqref{defn_uminus_2j+1}:  %$u_{-2j} \in C^{2}(\OM; l_{\INF})$ for  $0 \leq j \leq \frac{m}{2}-1$, and $u_{-(2j+1)} \in C^{1}(\OM; l_{\INF})$, for  $0 \leq j \leq \frac{m}{2}-1$, and
		\begin{equation}\label{BeltramiEq_mEven_umPhiag}
			\begin{aligned}
				&u_{-2j} \in C^{2}(\OM; l_{\INF}),              && \text{for} \quad 0 \leq j \leq q-1, \; q \geq 1,\\
				&u_{-(2j+1)} \in C^{1}(\OM; l_{\INF}),          && \text{for} \quad 0 \leq j \leq q-1, \; q \geq 1,\quad \text{and} \\
				&\ol{\del} u_{-2j} + \del u_{-(2j+2)} +a u_{-(2j+1)}= 0, &&\text{for} \quad 0 \leq j \leq q-1, \; q \geq 1,
			\end{aligned}
		\end{equation}
		are satisfied. Moreover, by conjugating the last equation in \eqref{BeltramiEq_mEven_umPhiag} yields
		\begin{equation}\label{BeltramiEq_mEven_umPhiag_conj}
			\begin{aligned}
				&\del u_{2j} + \ol{\del} u_{(2j+2)} +a u_{(2j+1)}= 0, \quad \text{for} \quad 0 \leq j \leq q-1, \; q \geq 1.
			\end{aligned}
		\end{equation}
		
		By the definition of the class \eqref{ART_mEvenPsiClass}, and  reality of $g$, we have
		%Since $ \left(\psi_0,  \psi_{-2},  \cdots ,\psi_{-2 (q-1)}\right) \in \Psi_{a,g}^{\text{even}}$,
		the trace of $u_{-2j}$ in \eqref{defn_uminus_2j} satisfies
		\begin{equation}\label{trace_uminus_2j}
			\begin{aligned}
				u_{-2j} \lvert_{\Gam} &= g_{-2j},   \quad \text{and} \quad u_{2j} \lvert_{\Gam} = \ol{g_{-2j}} = g_{2j} , \quad  0 \leq j \leq q-1, \; q \geq 1.
			\end{aligned} 
		\end{equation} 
		We check next that the trace of $u_{-(2j+1)}$ is $g_{-(2j+1)}$ for  $0 \leq j \leq q-2, \; q \geq 2$:
		\begin{equation}\label{trace_uminus_twojone}
			\begin{aligned}
				u_{-(2j+1)} \big \lvert_{\Gam} &= \left. -\frac{\ol{\del}\psi_{-2j}+\del \psi_{-(2j+2)}}{a} \right \lvert_{\Gam} = g_{-(2j+1)},
			\end{aligned} 
		\end{equation} where the last equality uses 
		% the condition \eqref{ARTmEvenTensorCondClass} and 
		the condition in class \eqref{ART_mEvenPsiClass}.
		Similar calculation to \eqref{trace_uminus_twojone} for mode $u_{-(2q-1)}$ give the trace 
		\begin{align}\label{trace_u_oneminus_m}
			u_{-(2q-1)} \big \lvert_{\Gam} &= \left. -\frac{\ol{\del}\psi_{-2(q-1)}+\del u_{-2q}}{a} \right \lvert_{\Gam} =  g_{-(2q-1)}.
		\end{align} %where the last equality uses the condition in class \eqref{ART_mEvenPsiClass}.
		%\commentK{We should do the calculation.}
		Thus, from \eqref{trace_uminus_2j} - \eqref{trace_u_oneminus_m}, we have the traces: %for modes $u_{n}$ for $|n| \leq m-1$the traces shown that
		\begin{equation} \label{trace_umodn<m-1}
			\begin{aligned}
				% \ol{\del} u_{2k} + \del{u_{2k-2}} &= 0, \quad \forall k\in \mathbb{Z}, \\ %\label{PostiveEvensTrace}
				u_{n}\big \lvert_{\Gam} &= g_{n}, \quad \forall |n| \leq 2q-1.
			\end{aligned}
		\end{equation}
		
		{\bf Step 4: The construction of the tensor field $\bbf_{\Psi} $ whose attenuated $X$-ray data is $g$.}

		The components of the $m$-tensor $\bbf_{\Psi} $ are defined via the one-to-one correspondence between  the pseudovectors $\langle \tilde{f}_0, \tilde{f}_{1}, \cdots , \tilde{f}_{m} \rangle $  and the functions  $\{f_{2n}:\; -q\leq n\leq q\}$ as follows. 
		
		We define first $f_{2q}$  by using $\psi_{-2(q-1)}$ from the non-uniqueness class, and Fourier modes $u_{-2q}, u_{-(2q+1)}  \in C^{2}(\OM;l_{\INF})$ from \eqref{eq:LmU_eGv}. Then, next define $f_{2q-2}$  by using $\psi_{-2(q-1)}, \psi_{-2(q-2)} $ from the non-uniqueness class, and Fourier mode $u_{-2q}$ from \eqref{eq:LmU_eGv}. Then, define
		$\{f_{2n}:\; 0 \leq n \leq q-2\}$ solely from the information in the non-uniqueness class. Finally, define $\{f_{-2n}:\; 1\leq n\leq q\}$ by conjugation.

		%From $ \left(\psi_0,  \psi_{-2},  \cdots ,\psi_{-2 (q-1)}\right)  \in \Psi^{\text{even}}_{a,g}$,
		%and Fourier modes $u_{-2q}, u_{-(2q+1)}  \in C^{2}(\OM;l_{\INF})$ from \eqref{eq:LmU_eGv}, 
		%%and Fourier mode $ u_{-(2q+1)} \in C^{2}(\OM;l_{\INF})$ from \eqref{eq:LmU_regularity},
		%we define first $f_{2q}$ depending on $m=2q,\; q\geq1$, then define $f_{2q-2}$ depending on $ q\geq2$, and then continue as follows:
		\begin{equation} \label{ART_mEven_fmPsiEq}
			\begin{aligned} %\label{NART_mEven_fzeroPsiEq}
				&f_{2q}:= -\ol{\del} \left( \frac{\ol{\del}\psi_{-2(q-1)}+\del u_{-2q}}{a} \right) + \del u_{-(2q+1)} +a u_{-2q},  \quad q\geq 1,\\
				&f_{2q-2}:= -\ol{\del} \left( \frac{\ol{\del}\psi_{-2(q-2)}+\del \psi_{-2(q-1)}}{a} \right) - \del  \left( \frac{\ol{\del}\psi_{-2(q-1)}+\del u_{-2q}}{a} \right) +a \psi_{-2(q-1)}, \quad q\geq 2,\\
				&f_{2n}:= -\ol{\del} \left( \frac{\ol{\del}\psi_{-2(n-1)}+\del \psi_{-2n}}{a} \right) - \del  \left( \frac{\ol{\del}\psi_{-2n}+\del \psi_{-2(n+1)}}{a} \right) +a \psi_{-2n}, \quad 1 \leq n \leq q-2, \quad q\geq 3, \\
				% &f_{0}:=-2 \re \del \left( \frac{\ol{\del}\psi_{0}+\del \psi_{-2}}{a} \right)+a \psi_0, \; q\geq 3,\\   
				&f_{0}:= \begin{cases}
					\ds -2 \re \del \left( \frac{\ol{\del}\psi_{0}+\del u_{-2}}{a} \right)+a \psi_0, \quad q= 1,\\ 
					\ds -2 \re \del \left( \frac{\ol{\del}\psi_{0}+\del \psi_{-2}}{a} \right)+a \psi_0, \quad q\geq 2,\\ 
				\end{cases} \\
				&f_{-2n}:=\ol{f_{2n}}, \quad 0 \leq n \leq q, \; q\geq 1,
			\end{aligned}
		\end{equation}
		% \commentK{The above defn is for $m=4$. Have to define $f$ for $m=2$.}
		%$\bbf_{\Psi} \in C^\mu(\mathbf{S}^m; \OM)  $.
		By construction, $f_{2n}\in C(\OM)$, for $0 \leq n \leq q, \; q\geq 1$, as %$u_{-2q} \in C^{2}(\OM;l_{\INF})$ from \eqref{BeltramiEq_mEven_umPhiag}, and $ u_{-(2q+1)} \in C^{2}(\OM;l_{\INF})$ from \eqref{eq:LmU_regularity},  and
		$ \psi_{-2n}\in C^{2}(\OM;l_{\INF})$, for $0 \leq n \leq q-1$, from \eqref{ART_mEvenPsiClass}.
		Note that $f_{2n}$ satisfy \eqref{ART_mEven_fmEq}.
		We use these Fourier modes $\langle f_0 , f_{\pm 2}, f_{\pm 4}, \cdots , f_{\pm m}\rangle $  and equations \eqref{eq:Evenmodes_fk},  ~\eqref{reality_cond} and \eqref{eq:tilde_fk} 
		% from class $\Psi^{even}_{a,g}$ and modes $u_{-2q}$ and $u_{-(2q+1)}$, 
		to construct pseudovectors $\langle \tilde{f}_0, \tilde{f}_{1}, \cdots , \tilde{f}_{m} \rangle $, and thus $m$-tensor field $\bbf_{\Psi} \in C(\mathbf{S}^m; \OM)  $.
		
		In order to show $g \lvert_{\Gam_{+}} = X_{a}\bbf_{\Psi}$ with $\bbf_{\Psi}$ being constructed from pseudovectors via  Fourier modes as in \eqref{ART_mEven_fmPsiEq} from class $\Psi^{\text{even}}_{a,g}$, 
		we define the real valued function $u$ via its Fourier modes
		\begin{equation}\label{ART_mEven_defnUPsi}
			\begin{aligned}
				u(z, \btheta)&:= \sum_{|n| \geq 2q} u_{n}(z)e^{\i n\tta}  + 2 \re \left( -\frac{\ol{\del}\psi_{-2(q-1)}+\del u_{-2q}}{a} \right) e^{ -\i (2q-1)\tta}
				% +\sum_{n=0}^{m-1} u_{n}(z)e^{\i n\tta} +\sum_{n=1}^{m-1} u_{-n}(z)e^{-\i n\tta}
				\\ &\quad + 2 \re \left \{\sum_{n=0}^{q-1} \psi_{-2n}(z)e^{ -\i (2n)\tta}
				\right \}  + 2 \re \left \{\sum_{n=0}^{q-2} \left( -\frac{\ol{\del}\psi_{-2j} +\del \psi_{-(2j+2)}}{a} \right) e^{ -\i (2n+1)\tta} \right \} 
			\end{aligned} 
		\end{equation} and check that it has the trace $g$ on $\Gam$ and satisfies the transport equation \eqref{eq:Transport_fmEven}.
		
		Since $g\in C^{\mu} \left(\Gam; C^{1,\mu}(\sph) \right)\cap C(\Gam;C^{2,\mu}(\sph))$, we use Proposition \ref{functoseq_regularityprop} (ii) and %\cite[Corollary 4.1]{sadiqTamasan01} and
		\cite[Proposition 4.1 (iii)]{sadiqTamasan01} to conclude that $u$ defined in \eqref{ART_mEven_defnUPsi} belongs to $C^{1,\mu}(\OM \times \sph)\cap C^{\mu}(\ol{\OM}\times \sph)$. In particular $u(\cdot,\btheta)$ for $\btheta= (\cos\tta,\sin\tta)$ extends to the boundary and its trace satisfies
		\begin{align*}
			u(\cdot,\btheta)\lvert_{\Gam}% &=\left.\left ( \sum_{|n| \geq m} u_{n}e^{\i n\tta} +\sum_{|n| \leq m-1} u_{n}e^{\i n\tta} \right) \right\lvert_{\Gam}
			&=\sum_{|n| \geq 2q} u_{n} \big \lvert_{\Gam} e^{\i n\tta}  +\sum_{|n| \leq 2q-1} u_{n} \big \lvert_{\Gam} e^{\i n\tta} 
			% \\&
			=\sum_{|n| \geq 2q} g_{n}  e^{\i n\tta}  +\sum_{|n| \leq 2q-1} g_{n} e^{\i n\tta}  
			= g(\cdot,\btheta),
		\end{align*} where in the second equality above we use \eqref{vn_intermsof_gn}, \eqref{Trace_un_pos_gn} and \eqref{trace_umodn<m-1}.
		
		Since $u\in C^{1,\mu}(\OM \times \sph)\cap C^{\mu}(\ol{\OM}\times \sph)$, then %upon simplification and arrangement, and 
		using \eqref{uL2sys_mEven_a>0}, \eqref{uL2sys_mEven_a>0_Pos}, \eqref{defn_u_oneminusm}, \eqref{BeltramiEq_mEven_umPhiag},  \eqref{BeltramiEq_mEven_umPhiag_conj}, 
		% \eqref{ART_mEven_fmPsiEq}, we conclude \eqref{eq:Transport_fmEven}:
		and the definition of $f_{2n}$ for $-q \leq n \leq q, \;q\geq 1$ in \eqref{ART_mEven_fmPsiEq}, the real valued $u$ defined in \eqref{ART_mEven_defnUPsi} satisfies the transport equation \eqref{eq:Transport_fmEven}:
		\begin{align*}
			% \tta \cdot \nabla u &= e^{-2i \tta} f_2 + e^{2i \tta} \ol{f_2}+f_0 = \langle \bF_{\psi} \tta ,  \tta \rangle.
			\btheta \cdot \nabla u +a u &= \langle \bbf_{\Psi},  \btheta^{2q} \rangle, \quad q \geq 1.
		\end{align*}
	\end{proof}

	\section{Odd order $m$-tensor - attenuated case}\label{sec:RangeART_mOdd}
	%\section{$m$-tensor (Odd Case $\geq 1$)-attenuated case}
	In this section, we establish necessary and sufficient conditions for a sufficiently smooth function on $\Gam \times \sph$ to be the attenuated $X$-ray data of some sufficiently smooth real valued symmetric tensor field $\bbf$ of odd order 	$m =2q+1, \; q \geq 0$.
	
	In this case $a \neq 0$, the transport equation becomes
	\begin{align}\label{eq:Transport_fmOddART}
		\btheta \cdot\nabla u(x,\btheta) + a(x) u(x,\btheta) = \sum_{n=0}^{q} \left( f_{2n+1}(x) e^{- \i (2n+1) \tta} + f_{-(2n+1)}(x) e^{ \i (2n+1) \tta} \right) , \quad x\in\OM, 
	\end{align} 
	where $ \ol{f}_{2n+1} = f_{-(2n+1)}, \; 0 \leq n \leq q, \; q \geq 0$. 
	%Note that modes are complex conjugates.   
	
	If $\ds \sum_{n \in \BZ} u_{n}(z) e^{\i n\tta}$ is the Fourier series  expansion in the angular variable $\btheta$ of a solution $u$ of \eqref{eq:Transport_fmOddART}, then by identifying the Fourier coefficients of the same order, the equation \eqref{eq:Transport_fmOddART} reduces to the system:
	\begin{align}
		\label{ART_mOdd_fmEq}
		&\ol{\del} u_{-2n}(z) + \del u_{-(2n+2)}(z) +a u_{-(2n+1)}(z) = f_{2n+1}(z), && 0 \leq n \leq q, \; q \geq 0, \\ 
		\label{eq:Beltrami_mOdd_ueven<mby2}
		&\ol{\del} u_{-(2n-1)}(z) + \del u_{-(2n+1)}(z) +a u_{-2n}(z) = 0, && 0 \leq n \leq q, \; q \geq 0, \\ 
		\label{BeltramiEq_mOdd_un>=m}
		&\ol{\del} u_{-n}(z) + \del u_{-(n+2)}(z) +a u_{-(n+1)}(z)= 0, && n \geq 2q+1, \; q \geq 0,
	\end{align}
%	In particular, the sequence valued map $\ds \bu = \langle u_{0}, u_{-1},u_{-2},... \rangle$ 
%	solves the Beltrami equation
%	\begin{align*}%\label{BeltramiEq_mOddART_Lmu}
%		\dba L^m \bu +L^2 \del L^m \bu+ aL^{m+1}\bu = 0.
%	\end{align*}
	
	Recall that the trace $u \lvert_{\Gam \times \sph} := g$ as in \eqref{g_trace}, with $g= X_a \bbf \text{ on }\Gam_+ \mbox{ and  }g= 0 \mbox{ on }\Gam_- \cup \Gam_{0}.$ \\
	We expand the attenuated $X$-ray data $g$ in terms of its Fourier modes in the angular variables:
$\ds g(\zeta,\btheta) = \sum_{n = -\infty}^{\infty} g_{n}(\zeta) e^{\i n\tta},$ for $\zeta \in \Gam$.
%    Since the trace $g$ is also real valued, its Fourier modes will satisfy 
%	$\ds g_{-n} = \ol{g_{n}}, $ for $n \geq 0. $
	From the non-positive modes of $g$, we built the sequences $\ds \bg:=\langle g_{0}, {g}_{-1}, g_{-2}, ... \rangle,$ and 
	$\ds 	\bg_h:= e^{-G} \bg:= \langle \gamma_{0}, \gamma_{-1}, \gamma_{-2}, ... \rangle,$ where $e^{\pm G}$ as defined in \eqref{eGop}. 
	From the non-positive even, respectively, negative odd Fourier modes, we built the sequences
	\begin{align}\label{gHEvenOdd_mOddART}
		\bg_h^{\text{even}} 
		= \langle \gamma_{0}, \gamma_{-2},\gamma_{-4},...\rangle, \quad \text{and} \quad
		\bg_h^{\text{odd}} = \langle \gamma_{-1}, \gamma_{-3},\gamma_{-5},...\rangle.
	\end{align}
	
	Next we characterize the attenuated $X$-ray data $g$ in terms of its $m$ many modes $g_{0}, g_{-1}, \cdots g_{-(m-1)}$, and the Fourier modes 
	%\begin{align}\label{Lmg_h_mOddART}
	$\ds L^{m}\bg_h:= L^{m}e^{-G} \bg:= \langle \gamma_{-m}, \gamma_{-(m+1)}, \gamma_{-(m+2)}, ... \rangle$. \\
	%\end{align} 
	As before we construct simultaneously the right hand side of the transport equation \eqref{eq:Transport_fmOddART} together with the solution $u$. Construction of $u$ is via its Fourier  modes. We first construct the negative modes and then the positive modes are constructed by conjugation. 
	For $m =2q+1$ (odd integer), $q \geq 1$, the modes will be chosen arbitrarily from the class $\Psi_{a,g}^{\text{odd}}$ of cardinality $q=\frac{m-1}{2}$ with prescribed trace and gradient on the boundary $\Gam$ defined as
	\begin{equation}\label{ART_mOddPsiClass}
		\begin{aligned}
			\Psi_{a,g}^{\text{odd}}:=
			&\left \{ \left( \psi_{-1}, \psi_{-3}, \cdots ,\psi_{-(2q-1)}\right) \in \left(C^{2}(\ol\OM; \BC)\right)^{q}:  
			\right.  \\ 
			&\left. \quad \vphantom{\int}
			\psi_{-(2j-1)} \big \lvert_{\Gam}= g_{-(2j-1)}, \; 1 \leq j \leq q, \; q \geq 1,
			\right.  \\ 
			&\left. \quad \vphantom{\int}  \ol{\del} \psi_{-(2q-1)}  \big \lvert_{\Gam} = - \del (e^G \B e^{-G} \bg )_{-(2q+1)}  \big \lvert_{\Gam}  -  a \big \lvert_{\Gam} \,  g_{-2q}, \quad q \geq 1,
			\right. \\ 
			&\left. \quad \vphantom{\int} 
			\ol{\del} \psi_{-(2j-1)} \big \lvert_{\Gam} = - \del \psi_{-(2j+1)} \big\lvert_{\Gam}  -  a \big \lvert_{\Gam} \,   g_{-2j}, \quad 1 \leq j \leq q-1, \quad q \geq 2,
			\right. \\ 
			&\left. \quad \vphantom{\int}
			2 \left( \re \del \psi_{-1}  \big \lvert_{\Gam} \right ) = -a \lvert_{\Gam} \,  g_{0}, 
			\right \} 
		\end{aligned}
	\end{equation} where $\B$ be the Bukhgeim-Cauchy operator in  \eqref{BukhgeimCauchyFormula}, and the operators $e^{\pm G}$ as defined in \eqref{eGop}.

	\begin{remark}
		In the 1-tensor case ($q = 0$), there is no class, and the characterization of the attenuated $X$-ray data $g$ is in terms of its zero-th mode $g_0=\oint g(\cdot,\tta)d\tta$ and negative Fourier modes of $\bg_h:= e^{-G} \bg$.
	\end{remark}

	\begin{theorem}[Range characterization for odd order tensors]\label{ARTmOddTensor}
		Let $a \in C^{2,\mu}(\ol\OM)$, $\mu>1/2$ with $\underset{\ol{\OM}}{\min}\, a >0$. and $m = 2q+1, \, q\geq 0$.
		(i) Let $\bbf  \in C_{0}^{1,\mu}(\mathbf{S}^m ; \OM)$ 
		be a  real-valued symmetric  $m$-tensor field of odd order  and 
		$$g= X_a \bbf \text{ on }\Gam_+ \mbox{ and  }g= 0 \mbox{ on }\Gam_- \cup \Gam_{0}.$$
		Then $\bg_h^{\emph{even}}, \bg_h^{\emph{odd}} \in l^{1,1}_{\INF}(\Gamma)\cap C^\mu(\Gamma;l_1)$ satisfy
		\begin{align}\label{ARTmOddTensorCondEvenOdd}
			% &[I+\i\HT] L^{\frac{m+1}{2}}\bg_h^{even} =0, \quad [I+\i\HT] L^{\frac{m-1}{2}}\bg_h^{odd} =0, \quad \text {for} \quad q \geq0,
			&[I+\i\HT]  L^{\frac{m+1}{2}}  \bg_h^{\emph{even}}  =\bzero, \quad [I+\i\HT]  L^{\frac{m-1}{2}}\bg_h^{\emph{odd}}  =\bzero, \quad \text {for} \quad q \geq0,
		\end{align} where $\bg_h^{\emph{even}}, \bg_h^{\emph{odd}}$ are sequences in \eqref{gHEvenOdd_mOddART}.
%	, and $\HT$ is the Bukhgeim-Hilbert operator in \eqref{BHtransform}.
		Additionally, in $q=0$ case, for each $\zeta\in\Gam$, the zero-th Fourier mode $g_0$ of $g$ satisfy
		\begin{equation}\label{ART1Tensor_g0Cond}
			\displaystyle g_{0}(\zeta)= 
			\lim_{\OM\ni z \to \zeta\in \Gam} \frac{-2\re \del (e^G \B \bg_h )_{-1} (z)}{a(z)}, \quad \text {for} \quad q =0,
		\end{equation}
		where $\B$ be the Bukhgeim-Cauchy operator in  \eqref{BukhgeimCauchyFormula}, and the operators $e^{\pm G}$ as defined in \eqref{eGop}.
		
		(ii) Let $g\in C^{\mu} \left(\Gam; C^{1,\mu}(\sph) \right)\cap C(\Gam;C^{2,\mu}(\sph))$ be real valued with $g \lvert_{\Gam_{-} \cup \Gam_{0}}=0$.
		For $q=0$, if the corresponding sequences $\bg_h^{\emph{even}}, \bg_h^{\emph{odd}}\in Y_{\mu}(\Gam)$ satisfies \eqref{ARTmOddTensorCondEvenOdd},
		and $g_0$ satisfies \eqref{ART1Tensor_g0Cond}, then there exists a unique real valued vector field ($1$-tensor) 
		$\bbf \in C(\mathbf{S}^m;\OM)$ such that $g \lvert_{\Gam_{+}} = X_a\bbf$. 
		Moreover, for $q \geq 1$, if $\bg_h^{\emph{even}}, \bg_h^{\emph{odd}} \in Y_{\mu}(\Gam)$ satisfies \eqref{ARTmOddTensorCondEvenOdd}, 
		and for each element $\left( \psi_{-1}, \psi_{-3}, \cdots ,\psi_{-(2q-1)}\right)   \in \Psi_{a,g}^{\emph{odd}}$,
		then there is a unique real valued symmetric $m$-tensor $\bbf_{\Psi} \in C(\mathbf{S}^m; \OM)  $ such that
		$g\lvert_{\Gam_{+}} = X_a\bbf_{\Psi}$.
	\end{theorem}

	\begin{proof}
		(i) {\bf Necessity:} 
%		Let $\bbf =(f_{i_1 \cdots i_m})  \in  C_{0}^{1,\mu}(\mathbf{S}^m;\OM)$.
%		Since all components $f_{i_1 \cdots i_m}\in C_{0}^{1,\mu}(\OM)$ are compactly supported inside $\OM$, then for any point at the boundary there is a cone of lines which do not meet the support.
%		Thus $g \equiv 0$ in the neighborhood of the variety $\Gam_0$  which yields $g \in C^{1,\mu}(\Gam \times \sph)$. Moreover, $g$ is the trace on $\Gam \times \sph$ of a solution $u \in C^{1,\mu}(\ol{\OM} \times \sph)$ of the transport equation \eqref{eq:Transport_fmOddART}. By Proposition \ref{functoseq_regularityprop}(i) and Proposition \ref{eGprop}, $\bg_h= e^{-G} \bg \in l^{1,1}_{\INF}(\Gamma)\cap C^\mu(\Gamma;l_1)$. 
		Let $\bbf =(f_{i_1 \cdots i_m})  \in  C_{0}^{1,\mu}(\mathbf{S}^m;\OM)$.
		Since all components $f_{i_1 \cdots i_m}\in C_{0}^{1,\mu}(\OM)$, 
		$X_a \bbf \in C^{1,\mu}(\Gam_+)$, and, thus, the solution $u$ to the transport equation \eqref{eq:Transport_fmOddART} is in $C^{1,\mu}(\ol{\OM} \times \sph)$. Moreover, its trace $g=u \lvert_{\Gam \times \sph} \in C^{1,\mu}(\Gam \times \sph)$. By Proposition \ref{functoseq_regularityprop}(i) and Proposition \ref{eGprop}, $\bg_h= e^{-G} \bg \in l^{1,1}_{\INF}(\Gamma)\cap C^\mu(\Gamma;l_1)$.

		If $u$ solves \eqref{eq:Transport_fmOddART} then its Fourier modes satisfies \eqref{ART_mOdd_fmEq}, \eqref{eq:Beltrami_mOdd_ueven<mby2} and \eqref{BeltramiEq_mOdd_un>=m}.
		%In particular, the sequence valued map
		%\begin{align*}
		% &\OM \ni z\mapsto  \bu(z):=\langle u_{0}(z), u_{-1}(z), u_{-2}(z),  \cdots \rangle,
		%\end{align*} satisfies 
		%\begin{align*}
		% \dba L^m \bu +L^2 \del L^m \bu+ aL^{m+1}\bu = 0,
		%\end{align*} where $L^m \bu =\langle u_{-m}, u_{-m-1}, u_{-m-2},  \cdots \rangle$ .
		In particular, the sequence valued map $\ds \bu = \langle u_{0}, u_{-1},u_{-2},... \rangle$ 
		satisfy $\ds \dba L^m \bu +L^2 \del L^m \bu+ aL^{m+1}\bu = \bzero.$

		Let $\ds \bv := e^{-G} L^m \bu $, then by Lemma \ref{beltrami_reduction}, and the fact that the operators $e^{\pm G}$ commute with the left translation, $ [e^{\pm G}, L]=0$,
		the sequence $\ds \bv = L^m e^{-G}  \bu$ solves $\dba \bv + L^2\del \bv = \bzero$, i.e $\bv$ is $L^2$ analytic.
		The non-positive even subsequence $\langle v_{0}, v_{-2},   \cdots \rangle$, and negative odd subsequence $\langle v_{-1}, v_{-3},  \cdots \rangle$, respectively, are $L$ analytic, with traces 
		$L^{\frac{m+1}{2}}\bg_{h}^{\text{even}}$, respectively, $L^{\frac{m-1}{2}}\bg_{h}^{\text{odd}}$.
		The necessity part in Theorem  \ref{NecSuf_BukhgeimHilbert_Thm} yields \eqref{ARTmOddTensorCondEvenOdd}:
		\begin{align*}
			&[I+\i\HT] L^{\frac{m+1}{2}}\bg_{h}^{\text{even}} =\bzero, \quad [I+\i\HT] L^{\frac{m-1}{2}}\bg_{h}^{\text{odd}} =\bzero, \quad \text {for} \quad m =2q+1, \; q \geq0.
		\end{align*} 
		Additionally, in the $q=0$ case, the Fourier modes $u_{0}, u_{-1}, u_{1}$ of $u$ solve \eqref{eq:Beltrami_mOdd_ueven<mby2} for $n=0$.
		Since $a>0$ in $\OM$, we have
		\begin{align}\label{U0defn}
			u_{0}(z) &= \frac{-2\re  \del u_{-1} (z)  }{a(z)}, \quad z\in \OM.
		\end{align} Since the left hand side of \eqref{U0defn} is continuous all the way to the boundary, so is the right hand side. Moreover, the limit below exists and in the $q=0$ case, we have 
		\begin{align*}
			g_{0}(z_{0}) = \lim_{\OM \ni z \to z_{0}\in \Gam} u_{0}(z) =  \lim_{ \OM \ni z \to z_{0}\in \Gam} \frac{-2\re  \del u_{-1}(z)}{a(z)},
		\end{align*} thus \eqref{ART1Tensor_g0Cond} holds.
		This proves part (i) of the theorem.
		
		(ii) {\bf Sufficiency:} 	Let $g\in C^{\mu} \left(\Gam; C^{1,\mu}(\sph) \right)\cap C(\Gam;C^{2,\mu}(\sph))$ be real valued
		with $g \lvert_{\Gam_{-} \cup \Gam_{0}}=0$.
		Let the corresponding sequences $\bg_h^{\text{even}}, \bg_h^{\text{odd}}$ as in \eqref{gHEvenOdd_mOddART} satisfying \eqref{ARTmOddTensorCondEvenOdd}.  
		By Proposition \ref{functoseq_regularityprop}(ii) and Proposition \ref{eGprop}(iii), $\bg_h^{\text{even}}, \bg_h^{\text{odd}}\in Y_{\mu}(\Gam)$.
		
		Let $m =2q+1, \; q\geq 0$, be an odd integer. To prove the sufficiency we will construct a real valued symmetric $m$-tensor $\bbf$ in $\OM$ and a real valued function $u \in C^{1}(\OM\times \sph)\cap C(\ol\OM \times \sph)$
		such that $u \lvert_{\Gam \times \sph}=g$ and $u$ solves \eqref{eq:Transport_fmOddART} in $\OM$.
		The construction of such $u$ is in terms of its Fourier modes in the angular variable and it is done in several steps.

		{\bf Step 1: The construction of modes $u_{n}$ for $|n| \geq 2q+1, \; q\geq 0$.}

		Use the Bukhgeim-Cauchy Integral formula \eqref{BukhgeimCauchyFormula} to define the $L$-analytic maps
		\begin{align*}
			&\bv^{even}(z)= \langle v_{0}(z), v_{-2}(z), v_{-4}(z), ... \rangle:= \B L^{q+1}\bg_h^{\text{even}}(z), \quad z\in \OM,\\
			&\bv^{odd}(z)= \langle v_{-1}(z), v_{-3}(z), v_{-5}(z), ... \rangle:= \B L^{q}\bg_h^{\text{odd}}(z),\quad z\in \OM.
		\end{align*}By intertwining let also define $L^2$-analytic map
		\begin{align*}
			\bv(z):=\langle v_{0}(z), v_{-1}(z), v_{-2}(z),v_{-3}(z), ... \rangle , \quad z\in \OM.
		\end{align*}
		By Theorem \ref{NecSuf_BukhgeimHilbert_Thm} (ii),
		\begin{align}\label{smothness__v_j_mOddART}
			\bv^{\text{even}}, \bv^{\text{odd}},\bv\in C^{1,\mu}(\OM; l_{1})\cap C^{\mu}(\ol\OM;l_1)\cap C^2(\OM;l_\infty).
		\end{align}Moreover, since $\bg_h^{\text{even}}, \bg_h^{\text{odd}}$ satisfy the hypothesis \eqref{ARTmEvenTensorCondEvenOdd}, by  Theorem  \ref{NecSuf_BukhgeimHilbert_Thm} sufficiency part, we have
		\begin{align*}
			\bv^{\text{even}} \lvert_{\Gam} = L^{q+1}\bg_h^{\text{even}} \quad \text{and}\quad  \bv^{\text{odd}}\lvert_{\Gam} = L^{q}\bg_h^{\text{odd}}, \quad q \geq0.
		\end{align*} 
		In particular, $\bv$ is $L^2$-analytic with trace:
		\begin{align}\label{vn_intermsof_gn_mOddART}
			\bv \lvert_{\Gam} = L^{2q+1}\bg_h = L^{2q+1} e^{-G} \bg,   \quad q \geq0,
		\end{align} where $\bg_h$ is formed by intertwining $\bg_h^{\text{even}}$ and $ \bg_h^{\text{odd}}$.
		
		For $q \geq 0$, define the sequence valued map 
		\begin{align}\label{eq:LmU_eGv_mOddART}
			&\OM \ni z\mapsto  L^{2q+1} \bu(z)=\langle u_{-(2q+1)}(z), u_{-(2q+2)}(z), u_{-(2q+3)}(z),  \cdots \rangle := e^{G} \bv (z).
		\end{align}  
		By Proposition \ref{eGprop}, 
		$ L^{2q+1} \bu \in C^{1,\mu}(\OM; l_{1})\cap C^{\mu}(\ol\OM;l_1)$. 
		Moreover, since $\bv\in C^2(\OM;l_\infty)$ as in \eqref{smothness__v_j_mOddART}, we also conclude from convolution that $L^{2q+1} \bu \in C^{2}(\OM; l_\infty).$
		Thus,  
		\begin{align}\label{smothness_Lbu_mOddART}
			L^{2q+1} \bu\in C^{1,\mu}(\OM; l_{1})\cap C^{\mu}(\ol\OM;l_1)\cap C^2(\OM;l_\infty).
		\end{align}
		As $\bv$ is $L^2$ analytic, by Lemma \ref{beltrami_reduction}, $L^{2q+1} \bu$ satisfies $\ds 
			\dba L^{2q+1} \bu +L^2 \del L^{2q+1} \bu+ a L^{2q+2}\bu = 0,$ for $ q \geq0,$ which in component form is written as:
		\begin{align}\label{uL2sys_mEven_a>0_mOddART}
			\ol{\del} u_{-n} + \del u_{-n-2} +a u_{-n-1} = 0, \quad n\geq 2q+1, \; q\geq 0.
		\end{align}
		The trace satisfy
		\begin{align}\label{Lmu_trace_mOddART}
			L^{2q+1} \bu \lvert_{\Gam}    = e^G \bv \lvert_{\Gam} = e^G L^{2q+1} e^{-G} \bg = L^{2q+1} \bg, \quad q\geq 0,
		\end{align} where the second equality follows from \eqref{vn_intermsof_gn_mOddART} and in the last equality  we use $ [e^{\pm G}, L]=0$.
		
		Construct the positive Fourier modes by conjugation: 
		$\ds u_{n}:=\ol{u_{-n}},$ for all $n\geq 2q+1, \,q\geq 0$. Moreover using \eqref{Lmu_trace_mOddART}, and the reality of $g$, the traces $u_n \lvert_{\Gam}$ satisfy
		\begin{align}\label{Trace_un_pos_gn_mOddART}
			u_{n}\lvert_{\Gam} = \ol{u_{-n}}\lvert_{\Gam}= \ol{g_{-n}}= g_{n}, \quad n \geq 2q+1, \; q\geq 0.
		\end{align}
%		Thus, from \eqref{Lmu_trace_mOddART} and \eqref{Trace_un_pos_gn_mOddART}, we have the traces: 
%		\begin{equation} \label{trace_umodn>=m_mOddART}
%			\begin{aligned}
%				u_{n}\big \lvert_{\Gam} &= g_{n}, \quad \forall |n| \geq 2q+1, \; q\geq 0.
%			\end{aligned}
%		\end{equation}
		
		By conjugating \eqref{uL2sys_mEven_a>0_mOddART}, and from \eqref{Lmu_trace_mOddART} and \eqref{Trace_un_pos_gn_mOddART}, we thus have the Fourier modes satisfy
		\begin{align}\label{uL2sys_mEven_|n|>=m_mOddART}
			\ol{\del} u_{-n} + \del u_{-n-2} +a u_{-n-1} = 0, \quad  \text {and} \quad	u_{n}\big \lvert_{\Gam} = g_{n}, \quad \forall |n| \geq 2q+1, \; q\geq 0.
		\end{align} 
		
%		we note that the positive Fourier modes also satisfy
%		\begin{align*}%\label{uL2sys_mEven_a>0_Pos_mOddART}
%			\ol{\del} u_{n+2} + \del u_{n} +a u_{n+1} = 0, \quad n\geq 2q+1, \; q\geq 0.
%		\end{align*}
%		Thus, from the above equation and \eqref{uL2sys_mEven_a>0_mOddART}, we have the Fourier modes satisfy
%		\begin{align}\label{uL2sys_mEven_|n|>=m_mOddART}
%			\ol{\del} u_{-n} + \del u_{-n-2} +a u_{-n-1} = 0, \quad |n|\geq 2q+1, \; q\geq 0.
%		\end{align}
		
		{\bf Step 2: The construction of 1-tensor  ($q = 0$ case).}
		
		Since $a>0$ in $\OM$, we can define $u_{0}$ (in $q=0$ case) by using the Fourier mode $u_{-1}$ from \eqref{eq:LmU_eGv_mOddART}:
		\begin{align}\label{ConvU0defn}
			u_{0}(z) := \displaystyle -\frac{2\re{\del u_{-1}(z)}}{a(z)},\quad z \in \OM, \quad (\text{for  } q = 0 
			\, \text{   case}).
		\end{align} 
		Note that $u_0$ satisfy \eqref{uL2sys_mEven_|n|>=m_mOddART} for $n=-1$. In particular 
		$\ds \ol{\del} u_{1} + \del u_{-1} +au_{0} = 0 $ holds.\\
		From \eqref{ART1Tensor_g0Cond}, $u_{0}$ defined above extends continuously to the boundary $\Gam$ and 
		\begin{align*}%\label{trace_u_0_mOddART_q=0}
			u_{0} \big \lvert_{\Gam} = g_0, \quad  (\text{for  } q = 0 \, \text{   case}).
		\end{align*}
		Moreover, since $u_{-1}\in C^2(\OM)$ as shown in \eqref{smothness_Lbu_mOddART} and $a\in C^2(\OM)$ we get $u_0\in C^1(\OM)$.
		
		Using the Fourier modes $u_{-1}, u_{-2}$ from \eqref{eq:LmU_eGv_mOddART} and $u_0$ as in \eqref{ConvU0defn}, 
		we next define the real valued vector field $\bbf \in C(\OM;\mathbb{R}^2)$ (for $q=0$ case) by
		\begin{align}\label{onetensor_ART}
			\bbf = \langle 2 \re {f_1}, 2\im{f_1} \rangle, \quad \text{where } \quad
			f_{1}:=\ol{\del}u_{0}+\del u_{-2}+au_{-1}.%,   \quad (\text{for } \; q = 0 \; \text{case}).
		\end{align}

%		The $\displaystyle u(z, \btheta):= \sum_{n\in \BZ} u_{n}(z)e^{\i n\tta}$  satisfies \eqref{eq:Transport_fmOddART} for $q=0$, and has the trace 
%		$\displaystyle u(\cdot,\btheta)\lvert_{\Gam} = g(\cdot,\btheta)$.

		%%%%%%%%%%%%%%%%%%%%%%%%
		%The remaining steps require $q \geq 1$, and the construction of the tensor field $\bbf_{\psi} $  is in terms of the Fourier mode $u_{-(2q+1)}$ in \eqref{eq:LmU_eGv_mOddART} and the class $\Psi_{a,g}^{\text{odd}}$ in \eqref{ART_mOddPsiClass}.

		We consider next the case $q \geq 1$ of tensors of order 3 or higher. 
		In this case the construction of the tensor field $\bbf_{\Psi} $  is in terms of the Fourier modes $u_{-(2q+1)}, u_{-(2q+2)}$ in \eqref{eq:LmU_eGv_mOddART} and the class
		$ \Psi_{a,g}^{\text{odd}}$ as in \eqref{ART_mOddPsiClass}. 
		
		{\bf Step 3: The construction of modes $u_{n}$ for $|n| \leq 2q, \; q\geq 1$.}%$\underbrace{u_{-(m-1)}, u_{-(m-2)}, \cdots ,u_{-1}, u_0, u_1, \cdots , u_{m-2}, u_{m-1}}_{2m-1}$.}
		% \commentK{Have to look at the class again in the attenuated m-Odd case!}
		
		Recall  the non-uniqueness class $ \Psi_{a,g}^{\text{odd}}$ as in \eqref{ART_mOddPsiClass}.

		For $\left( \psi_{-1}, \psi_{-3}, \cdots ,\psi_{-(2q-1)}\right) \in \Psi_{a,g}^{\text{odd}}$ arbitrary,  firstly define the odd modes 
		\begin{equation}\label{defn_uPsiOdd_mOddART}
			\begin{aligned}
				u_{-(2n-1)}&:= \psi_{-(2n-1)},  \quad \text{and} \quad u_{2n-1} := \ol{\psi}_{-(2n-1)}, \quad 1 \leq n \leq q, \; q \geq 1.
			\end{aligned} 
		\end{equation}
		Secondly, by using $\psi_{-1}, \psi_{-(2q-1)} $ and the mode $u_{-(2q+1)}$ from \eqref{eq:LmU_eGv_mOddART}, we define the modes 
		\begin{align} \label{defn_u_zero_mOddART} 
			u_{0}&:= -\frac{2 \re \del \psi_{-1}}{a}, \\ \label{defn_u_oneminus_mOddART}
			u_{-2q}&:= -\frac{\ol{\del}\psi_{-(2q-1)}+\del u_{-(2q+1)}}{a}, \quad \text{and} \quad u_{2q} := \ol{u_{-2q}} \quad \text{for} \quad q \geq 1.
			% \\ \label{defn_u_one_mOddART}
			% u_{m-1} &:= \ol{u_{-(m-1)}}.
		\end{align} 
		
		Lastly, by using $\psi_{-(2n-1)} \in C^{2}(\ol\OM;\mathbb{C})$, for  $1 \leq n \leq q-1, \; q \geq 2$, we define the even modes 
		\begin{equation}\label{defn_uPsiEven_mOddART}
			\begin{aligned}
				u_{-2n}&:= -\frac{\ol{\del}\psi_{-(2n-1)}+\del \psi_{-(2n+1)}}{a}, \quad 1 \leq n \leq q-1, \; q \geq 2, \quad \text{and} \\
				u_{2n} &:= \ol{u}_{-2n}, \quad 1 \leq n \leq q-1, \; q \geq 2.
				% u_{-2n}&:= -\frac{\ol{\del}\psi_{1-2n}+\del \psi_{-(2n+1)}}{a}, \quad \text{and} \quad u_{2n} := \ol{u_{-2n}}, \quad \text{for}\quad  1 \leq n \leq \frac{m-3}{2}.
			\end{aligned} 
		\end{equation}
		By the construction in \eqref{defn_u_zero_mOddART}, \eqref{defn_u_oneminus_mOddART}, and \eqref{defn_uPsiEven_mOddART}, we have 
		\begin{equation}\label{BeltramiEq_mOdd_umPhiag_mOddART}
			\begin{aligned}
				&u_{-(2n-1)} \in C^{2}(\OM; l_{\INF}), \quad \text{for} \quad 1 \leq n \leq q, \; q \geq 1, \\
				&u_{-2n} \in C^{1}(\OM; l_{\INF}), \quad \text{for} \quad 0 \leq n \leq q, \; q \geq 1, \quad \text{and} \\
				% &\ol{\del} u_{-(2j)} + \del u_{-(2j+2)} +a u_{-(2j+1)}= 0, \quad \text{for} \quad 0 \leq j \leq \frac{m}{2}-1,
				&\ol{\del} u_{-(2n-1)} + \del u_{-(2n+1)} +a u_{-2n} = 0, \quad \text{for} \quad 0 \leq n \leq q, \; q \geq 1, \\ 
			\end{aligned}
		\end{equation}
		is satisfied. 
		Moreover, by conjugating the last equation in \eqref{BeltramiEq_mOdd_umPhiag_mOddART}, we have the Fourier modes satisfy
		\begin{align}\label{uL2sys_mEven_|n|<=m_mOddART}
			\ol{\del} u_{-(2n-1)} + \del u_{-(2n+1)} +a u_{-2n} = 0, \quad \text{for} \quad |n| \leq q, \; q \geq 1.
		\end{align}
		By the class \eqref{ART_mOddPsiClass}, and  reality of $g$, we have
		the trace of $u_{-(2n-1)} $ in \eqref{defn_uPsiOdd_mOddART} satisfy
		\begin{equation}\label{trace_uPsiOdd_mOddART}
			\begin{aligned}
				u_{-(2n-1)} \lvert_{\Gam} &= g_{-(2n-1)},  \quad \text{and} \quad 
				u_{2n-1} \lvert_{\Gam} = \ol{g}_{-(2n-1)} = g_{2n-1} , \quad 1 \leq n \leq q, \; q \geq 1.
			\end{aligned} 
		\end{equation} 
		We check next that the trace of $u_{-2n}$ is $g_{-2n}$ for  $1 \leq n \leq q-1, \; q \geq 2$:
		\begin{equation}\label{trace_uminus_twojone_mOddART}
			\begin{aligned}
				u_{-2n} \big \lvert_{\Gam} &= \left. -\frac{\ol{\del}\psi_{-(2n-1)}+\del \psi_{-(2n+1)}}{a} \right \lvert_{\Gam} = g_{-2n},
			\end{aligned} 
		\end{equation} where the last equality uses the condition in class \eqref{ART_mOddPsiClass}.
		Similar calculation to \eqref{trace_uminus_twojone_mOddART} for mode $u_{0}$ in \eqref{defn_u_zero_mOddART}, and mode $u_{-2q}$ in \eqref{defn_u_oneminus_mOddART}, give the trace 
		\begin{equation}
			\begin{aligned}\label{trace_u_oneminus_mOddART}
				u_{0} \big \lvert_{\Gam} &= g_{0}, \quad \text{and} \quad u_{-2q} \big \lvert_{\Gam} = g_{-2q}, \quad q \geq 1.
			\end{aligned}
		\end{equation}
		% \commentK{Check the above calculation and make sure it is fine.}
		Thus, from \eqref{trace_uPsiOdd_mOddART}, \eqref{trace_uminus_twojone_mOddART} and \eqref{trace_u_oneminus_mOddART}, we have the traces: %for modes $u_{n}$ for $|n| \leq m-1$the traces shown that
		\begin{equation} \label{trace_umodn<m_mOddART}
			\begin{aligned}
				u_{n}\big \lvert_{\Gam} &= g_{n}, \quad \forall |n| \leq 2q, \; q \geq 1.
			\end{aligned}
		\end{equation}

		{\bf Step 4: The construction of the tensor field $\bbf_{\Psi} $ whose attenuated $X$-ray data is $g$.}
		
		The components of the $m$-tensor $\bbf_{\Psi} $ are defined via the one-to-one correspondence between  the pseudovectors $\langle \tilde{f}_0, \tilde{f}_{1}, \cdots , \tilde{f}_{m} \rangle $  and the functions  $\{f_{\pm(2n+1)}:\; 0\leq n\leq q\}$ as follows. 
		
		We first define $f_{2q+1}$  by using $\psi_{-(2q-1)}$ from the non-uniqueness class, and the  Fourier modes $u_{-(2q+1)}, u_{-(2q+2)}$ in \eqref{eq:LmU_eGv_mOddART}. Next, define
		$f_{2q-1}$  by using $\psi_{-(2q-1)}, \psi_{-(2q-3)} $ from the non-uniqueness class, and Fourier mode $u_{-(2q+1)}$ in \eqref{eq:LmU_eGv_mOddART}. Then, define
		$\{f_{2n+1}:\; 0 \leq n \leq q-2\}$ solely from the information in the non-uniqueness class. Finally, define $\{f_{-(2n+1)}:\; 0\leq n\leq q\}$ by conjugation. 
		%From $\left( \psi_{-1}, \psi_{-3}, \cdots ,\psi_{-(2q-1)}\right) \in \Psi^{\text{odd}}_{a,g}$ for $q\geq 1$, and Fourier modes $u_{-(2q+1)}$ and $u_{-(2q+2)}$,
		%we define first $f_{2q+1}$ depending on $m=2q+1,\; q\geq1$, then define $f_{2q-1}$ depending on $ q\geq2$, and then continue as follows: 
		\begin{equation} \label{ART_mOdd_fmPsiEq}
			\begin{aligned} %\label{NART_mEven_fzeroPsiEq}
				&f_{2q+1}:= -\ol{\del} \left( \frac{\ol{\del}\psi_{-(2q-1)}+\del u_{-(2q+1)}}{a} \right) + \del u_{-(2q+2)} +a u_{-(2q+1)},  \quad q\geq 1,\\
				&f_{2q-1}:= -\ol{\del} \left( \frac{\ol{\del}\psi_{-(2q-3)}+\del \psi_{-(2q-1)}}{a} \right) - \del  \left( \frac{\ol{\del}\psi_{-(2q-1)}+\del u_{-(2q+1)}}{a} \right) +a \psi_{-(2q-1)}, \quad q\geq 2,\\
				&f_{2n+1}:= -\ol{\del} \left( \frac{\ol{\del}\psi_{-(2n-1)}+\del \psi_{-(2n+1)}}{a} \right) - \del  \left( \frac{\ol{\del}\psi_{-(2n+1)}+\del \psi_{-(2n+3)}}{a} \right) +a \psi_{-(2n+1)}, \,1 \leq n \leq q-2,  \\
				% &f_{1}:=-2  \ol{\del} \left( \frac{ \re \del \psi_{-1} }{a} \right)- \del  \left( \frac{\ol{\del}\psi_{-1}+\del \psi_{-3}}{a} \right) +a \psi_{-1}, \\   %\label{NART_mEven_fmPsiEq}
				&f_{1}:= \begin{cases}
					\ds -2  \ol{\del} \left( \frac{ \re \del \psi_{-1} }{a} \right)- \del  \left( \frac{\ol{\del}\psi_{-1}+\del u_{-3}}{a} \right) +a \psi_{-1}, \quad q= 1,\\ 
					\ds -2  \ol{\del} \left( \frac{ \re \del \psi_{-1} }{a} \right)- \del  \left( \frac{\ol{\del}\psi_{-1}+\del \psi_{-3}}{a} \right) +a \psi_{-1}, \quad q\geq 2,\\ 
				\end{cases} \\
				&f_{-(2n+1)}:=\ol{f_{2n+1}}, \quad 0 \leq n \leq q, \; q\geq 1,
			\end{aligned}
		\end{equation}
		By construction,  $f_{2n+1}\in C(\OM)$ for $0 \leq n \leq q, \; q\geq 1$, as $u_{-(2q+1)} \in C^{2}(\OM;l_{\INF})$ from \eqref{smothness_Lbu_mOddART}, and $ \psi_{-(2n-1)}\in C^{2}(\OM;l_{\INF})$, for $1 \leq n \leq q-1, \;q\geq 1$, from \eqref{ART_mOddPsiClass}.
		We use these $m+1$ Fourier modes $\langle f_{\pm 1}, f_{\pm 3}, \cdots , f_{\pm m}\rangle $,
		and equations \eqref{eq:Oddmodes_fk},  ~\eqref{reality_cond} and \eqref{eq:tilde_fk} to construct the pseudovectors $\langle \tilde{f}_0, \tilde{f}_{1}, \cdots , \tilde{f}_{m} \rangle $, and thus the $m$-tensor field $\bbf_{\Psi} \in C(\mathbf{S}^m; \OM)  $.

%		In order to show $g \lvert_{\Gam_{+}} = X_{a}\bbf_{\Psi}$ with $\bbf_{\Psi}$ being constructed from pseudovectors via  Fourier modes as in \eqref{ART_mOdd_fmPsiEq} from class $\Psi^{\text{odd}}_{a,g}$, 
%		we d
		Define the real valued function $u$ via its Fourier modes
		\begin{equation}\label{ART_mOdd_defnUPsi}
			\begin{aligned}
				u&(z, \btheta):= \sum_{|n| \geq 2q+1} u_{n}(z)e^{\i n\tta}   + 2 \re \left \{ \sum_{n=1}^{q} \psi_{-(2n-1)}(z)e^{ -\i (2n-1)\tta} \right \}  +\frac{-2\re \del \psi_{-1}(z)}{a}\\
				&\quad + 2 \re {\left(-\frac{\ol{\del}{\psi_{-(2q-1)}(z)}+\partial u_{-(2q+1)}(z)}{a}\right)}e^{-\i(2q)\theta} + 2 \re \left \{\sum_{n=1}^{q-1} u_{-2n}e^{-\i (2n\theta)} \right \} .
%				\\
%				&\quad +{\left(-\frac{\del\ol{\psi}_{-(2q-1)}(z)+\ol\del u_{(2q+1)}(z)}{a}\right)}e^{\i(2q)\theta} 
%				+\sum_{n=1}^{q-1} \left(-\frac{\del\ol{\psi}_{-(2n-1)}(z)+ \ol{\del}\ol{ \psi}_{-(2n+1)}(z)}{a}\right)e^{\i(2n\theta)}.
				% \sum_{n=0}^{\frac{m-1}{2}} u_{-2n}(z)e^{-\i 2n\tta} +\sum_{n=0}^{\frac{m-1}{2}} u_{2n}(z)e^{\i 2n\tta} 
				% \\ &\quad  + \sum_{n=0}^{\frac{m-1}{2}} \left( -\frac{\ol{\del}\psi_{-2j}+\del \psi_{-(2j+2)}}{a} \right) e^{ -\i (2n+1)\tta} + \sum_{n = 0}^{\frac{m}{2}-2} \left(-\frac{\del \ol{\psi_{-2j}}+\ol{\del} \; \ol{\psi_{-(2j+2)}}}{a} \right) e^{\i (2n+1)\tta}, 
			\end{aligned} 
		\end{equation} 
		
		Using \eqref{uL2sys_mEven_|n|>=m_mOddART} and \eqref{trace_umodn<m_mOddART},  and definition of $\left( \psi_{-1}, \psi_{-3}, \cdots ,\psi_{-(2q-1)}\right) \in \Psi_{a,g}^{\text{odd}}$  for  $ q \geq 1$, the trace
		$u(\cdot,\btheta)$  in \eqref{ART_mOdd_defnUPsi}  extends to the boundary, and its trace satisfy $\ds u(\cdot,\btheta)\lvert_{\Gam}= g(\cdot,\btheta).$
		
		%Using the fact that the traces $\ds \psi_{-(2n-1)} \big \lvert_{\Gam} = g_{-(2n-1)}$ for $1 \leq n \leq q, \; q \geq 1$, 
		%%as $\left( \psi_{-1}, \psi_{-3}, \cdots ,\psi_{-(2q-1)}\right) \in \Psi_{a,g}^{\text{odd}}$, 
		%along with equations \eqref{trace_umodn>=m_mOddART} and \eqref{trace_umodn<m_mOddART}
		%yields that $u$ defined in \eqref{ART_mOdd_defnUPsi} has the trace $g$ on $\Gam$.
		%More precisely
		%\begin{align*}
		%u(\cdot,\boldsymbol{\theta})\Big|_\Gamma  =g(\cdot,\boldsymbol{\theta}).
		%\end{align*}
		Moreover, by using \eqref{uL2sys_mEven_|n|>=m_mOddART}, \eqref{uL2sys_mEven_|n|<=m_mOddART} and the definition of $f_{2n-1}$ for $|n| \leq q, \;q\geq 1$ in \eqref{ART_mOdd_fmPsiEq}, the real valued $u$ defined in \eqref{ART_mOdd_defnUPsi} satisfies the transport equation \eqref{eq:Transport_fmOddART}:
		\begin{align*}%\label{eq:Transport_fmOddART}
			\btheta \cdot\nabla u + a u =
			%= \sum_{n=0}^{q} \left( f_{2n+1} e^{- \i (2n+1) \tta} + f_{-(2n+1)} e^{ \i (2n+1) \tta} \right), \quad q\geq 1.
			\langle \bbf_{\Psi},  \btheta^{2q+1} \rangle, \quad q \geq 1.
		\end{align*} 
		
		%%%%%%%%%%%%%%%%%%%%%%%%%%%%%

	\end{proof}

	%%%%%%%%%%%%%%%%%%%%%%%%%%%%%%%%%%%%%%%%%%%%%%%%%%%%%%%%%%%%%%%%%%%%%%%
	
	\section*{Acknowledgment}
	The work of D.~Omogbhe and K.~ Sadiq  were supported by the Austrian Science Fund (FWF), Project P31053-N32. The work of K.~ Sadiq  was also supported by the FWF Project F6801–N36 within the Special Research Program SFB F68 “Tomography Across the Scales”. 
	%%%%%%%%%%%%%%%%%%%%%%%%%%%%%%%%%%%%%%%%%%%%%%%%%%%%%%%%%%%%

	%%%%%%%%%%%%%%%%%%%%%%%%%%%%%%%%%%%%%%%%%%%%%%%%%%%%%%%%%%%%%%%


\begin{thebibliography}{99}
		
		\bibitem{aguilarKuchment} V. Aguilar and P. Kuchment, \textit{Range conditions for the multidimensional exponential x-ray transform},  Inverse Problems, \textbf{11 (5)} (1995),977--982.
		
		\bibitem{aguilarEhrenpreisKuchment} V. Aguilar, L. Ehrenpreis and P. Kuchment, \textit{Range conditions for the exponential Radon transform}, J. Anal. Math. \textbf{68} (1996), 1--13.	
		
		\bibitem{ABK} E.~V.~Arbuzov, A.~L.~Bukhgeim and S.~G.~Kazantsev, \textit{Two-dimensional tomography problems and the theory of A-analytic functions}, Siberian Adv. Math., \textbf{8} (1998), 1--20.
		
		\bibitem{AMU} Y.~M.~Assylbekov, F. ~Monard and G.~Uhlmann, \textit{Inversion formulas and range characterizations for the attenuated geodesic ray transform}, Journal de Math\'ematiques Pures et Appliqu\'ees  \textbf{111} (2018), 161--190.
		
		\bibitem{bal04} G.~Bal, \textit{On the attenuated Radon transform with full and partial measurements}, Inverse Problems \textbf{20} (2004), 399--418.
		
		%\bibitem{bal04} G.~Bal, \textit{Inverse transport theory and applications}, Inverse Problems \textbf{25} (2009), 053001.
		
		%\bibitem{begehr05} H.~ Begehr, \textit{Boundary value problems in complex analysis II},
		%Boletin de la Asosiaci\'on Matem\'atica Venezolana,  \textbf{XII} (2005), 217--250.
		
		
		\bibitem{bomanStromberg} J.~Boman and J.-O.~Str\"omberg, \textit{Novikov's inversion formula for the attenuated Radon transform--a new approach}, J. Geom. Anal., \textbf{14} (2004), 185--198.
		
		\bibitem{braunHauk} H. Braun and A. Hauk, \textit{Tomographic reconstruction of vector
			fields}, IEEE Transactions on signal processing \textbf{39} (1991), 464--471.
		
		%\bibitem{bukhgeim^2} A. L. Bukhgeim and A. A. Bukhgeim, \textit{Inversion of the Radon transform, based on the theory of $A$-analytic functions, with application to 3D inverse kinematic problem with local data},  J. Inverse
		%Ill-Posed Probl., \textbf{14}(2006), 219--234.
		
		\bibitem{bukhgeimBook} A.~L.~Bukhgeim, \textit{Inversion Formulas in Inverse Problems}, chapter in Linear Operators and Ill-Posed Problems by M. M. Lavrentiev and L. Ya. Savalev, Plenum, New York, 1995.
		
		
		%\bibitem{clackdoyle13} R.~Clackdoyle, \textit{Necessary and sufficient consistency conditions for fan-beam projections along a line}, IEEE transations on Nuclear Science, \textbf{60} (2013), 1560--1569. 
		%
		
		
		\bibitem{derevtsovPickalov11} E.~Y.~Derevtsov and V.~V.~Pickalov, \textit{Reconstruction of vector fields and their singularities from ray transform} Numerical Analysis and Applications \textbf{4} (2011), 21--35.
		
		\bibitem{derevtsovSvetov15} E.~Derevtsov and I.~Svetov, \textit{Tomography of tensor fields in the plane}, Eurasian J. Math. Comput. Appl., \textbf{3(2)}, (2015), 24--68.
		
		\bibitem{denisiuk}  A.~Denisiuk, \textit{On range condition of the tensor X-ray transforms in $\BR^n$},  Inverse Prob. Imaging, \textbf{14(3)}, (2020), 423--435.
		
		%\bibitem{Finch-IP86} D.~V.~Finch, \textit{Uniqueness for the attenuated X-ray transform in the physical range}, Inverse Problems, \textbf{2} (1986), pp.~197--203.
		
		\bibitem{finch} D.~V.~Finch, \textit{The attenuated x-ray transform: recent developments}, in Inside out: inverse problems and applications, Math. Sci. Res. Inst. Publ., \textbf{47}, Cambridge Univ. Press, Cambridge, 2003, 47--66.
		
		% \bibitem{fujiwaraSadiqTamasan19} H.~Fujiwara, K.~Sadiq and A.~Tamasan, \textit{A Fourier approach to the inverse source problem in an absorbing and anisotropic scattering medium}, Inverse Problems {\bf 36(1)}:015005 (2019).
		
		% \bibitem{fujiwaraSadiqTamasan20} H.~Fujiwara, K.~Sadiq, and A.~Tamasan, \textit{Numerical Realization of the direct source reconstruction in the transport equation}, arXiv:1908.09133, SIAM J. Imaging Sci., {\bf 13(1)} (2020), 535--555. 
		
		%\bibitem{gelfandGindikinGraev80} I. M. Gelfand, S. G. Gindikin and M. I. Graev, \textit{Integral geometry in affine and projective spaces}, Itogi Nauki i Tekh., Ser. Sovrem. Probl. Mat., \textbf{16} (1980), 53--226 (in Russian). (English transl.: J. Sov. Math., \textbf{18} (1980), 39--167.)
		
		\bibitem{gelfandGraev} I. M. Gelfand and M.I. Graev, \textit{Integrals over hyperplanes of basic and generalized functions}, Dokl. Akad. Nauk SSSR {\bf135} (1960), no.6, 1307--1310; English transl., Soviet Math. Dokl. {\bf 1} (1960), 1369--1372.
		
		%\bibitem{gompelDelfriseDyck} G.~Van~Gompel, M.~Defrise, and D.~Van~Dyck, 
		%\textit{Elliptical extrapolation of truncated 2D CT projections using Helgason-
		%	Ludwig consistency conditions}, in Proceedings of the Medical
		%Imaging 2006: Physics of Medical Imaging, February 2006.
		
		\bibitem{helgason65} S. Helgason, \textit{The Radon transform on Euclidean spaces, compact two-point homogenous spaces and Grassmann manifolds}, Acta Math., \textbf{113} (1965), 153--180.
		
		%\bibitem{helgason} S. Helgason, \textit{An analogue of the Paley-Wiener theorem for the Fourier transform on certain symmetric spaces},
		%Math. Ann. {\bf 165} (1966), 297--308.
		
		\bibitem{helgasonBook} S.~Helgason, \textit{The Radon transform}, Progress in Mathematics \textbf{5}, Birkh\"{a}user, Boston,  1980.
		
		%\bibitem{hoopetal} M.~V.~Hoop, T.~Saksala, and J.~Zhai, \textit{Mixed ray transform on simple 2-dimensional Riemannian manifolds}, Proc. Am. Math. Soc., \textbf{147} (2019), 4901--4913.
		
		\bibitem{katznelson} Y.~Katznelson, \textit{An introduction to harmonic analysis}, Cambridge Math. Lib., Cambridge University Press, Cambridge, UK, 2004.
		
		\bibitem{kazantsevBukhgeimJr04} S.~G.~Kazantsev and A.~A.~Bukhgeim, \textit{Singular value decomposition for the 2D fan-beam Radon transform of tensor fields}, J. Inverse Ill-Posed Problems \textbf{12} (2004), 245--278.
		
		\bibitem{kazantsevBukhgeimJr06} S.~G.~Kazantsev and A.~A.~Bukhgeim, \textit{The Chebyshev ridge polynomials in 2D tensor tomography}, J. Inverse Ill-Posed Problems, \textbf{14} (2006), 157--188.
		
		\bibitem{kazantsevBukhgeimJr07}  S.~G.~Kazantsev and A.~A.~Bukhgeim, \textit{Inversion of the scalar and vector attenuated X-ray transforms in a unit disc}, J. Inverse Ill-Posed Probl., \textbf{15} (2007), 735--765.
		
		\bibitem{venke20} V.~P.~Krishnan, R.~Manna, S.~K.~ Sahoo, and V.~A.~ Sharafutdinov, \textit{Momentum ray transforms, II: range characterization in the Schwartz space
		}, Inverse Problems \textbf{36 (4)} (2020) 045009 (33pp).
		
		\bibitem{kuchmentLvin} P. Kuchment, DS. A. L'vin, \textit{Range of the Radon exponential transform}, Soviet Math. Dokl. 42 (1991), no. 1, 183--184
		
		%\bibitem{donsubPreprint} W. Li, K. Ren, and D. Rim, \textit{A range characterization of the single-quadrant ADRT},  arXiv:2010.05360v1, preprint (2020). 
		
		\bibitem{ludwig} D.~Ludwig, \textit{The Radon transform on euclidean space}, Comm. Pure Appl. Math., \textbf{19} (1966), 49--81.
		
		\bibitem{mishra} R.~K.~Mishra, \textit{Full reconstruction of a vector field from restricted Doppler and first integral 	moment transforms in $\BR^n$},  J. Inverse Ill-Posed Problems \textbf{28} (2019), 173--184.
		
		%\bibitem{monard14a} F.~Monard, \textit{Numerical implementation of geodesic X-ray transforms and their inversion}, SIAM J. Imaging Sci., \textbf{7} (2014), 1335--1357.
		
		%\bibitem{monard14b} F.~Monard, \textit{On reconstruction formulas for the X-ray transform acting on symmetric differentials on surfaces}, Inverse Problems, \textbf{30} (2014), 065001.
		
		\bibitem{monard16}  F.~Monard, \textit{Efficient tensor tomography in fan-beam coordinates}, Inverse Probl. Imaging, \textbf{10(2)} (2016), 433--459.
		
		%\bibitem{monard16b}  F.~Monard, \textit{Inversion of the attenuated geodesic X-ray transform over functions and vector fields on simple surfaces}, SIAM J. Math. Anal., \textbf{48(2)} (2016), 1155--1177.
		
		\bibitem{monard17}  F.~Monard, \textit{Efficient tensor tomography in fan-beam coordinates. II: Attenuated transforms}, Inverse Probl. Imaging, \textbf{12(2)} (2018), 433--460.
		
		%\bibitem{monardStefanovUhlmann15} F.~Monard, P.~Stefanov, and G.~Uhlmann, \textit{The geodesic X-ray transform on Riemannian surfaces with conjugate points}, Comm. Math. Phys., \textbf{337} (2015), 1491--1513.
		
		%\bibitem{monardPaternain16} F. Monard, and  G. P. Paternain, \textit{The geodesic X-ray transform with a $GL(n,\BC)$-connection}, (2016) arXiv:1610.09571
		
		%\bibitem{mukhometov75} R.~G.~Mukhometov, \textit{On the problem of integral geometry}, Math. Problems of Geophysics, 
		%Akad. Nauk. SSSR, Sibirs. Otdel., Vychisl. Tsentr, Novosibirsk, \textbf{6} (1975), 212-242, (in russian).
		
		\bibitem{muskhellishvili} N.~I.~Muskhelishvili, Singular Integral Equations, Dover, New York, 2008.
		
		\bibitem{nattererBook} F.~Natterer, {The mathematics of computerized tomography}, Wiley, New York, 1986.
		
		\bibitem{natterer01} F.~Natterer, \textit{Inversion of the attenuated Radon transform}, Inverse Problems \textbf{17} (2001), 113--119.
		
		\bibitem {nattererWubbeling} F.~Natterer and F.~W\"{u}bbeling,  Mathematical methods in image reconstruction. \textit{SIAM Monographs on Mathematical Modeling and Computation}, SIAM, Philadelphia, PA, 2001.
		
		\bibitem{norton88} S. J. Norton, \textit{Tomographic reconstruction of 2-D vector fields: application to flow imaging}, Geophysical Journal \textbf{97} (1988), 161--168.
		
		\bibitem{norton} S. J. Norton, \textit{Unique tomographic reconstruction of vector fields using boundary data}, IEEE Transactions on image processing \textbf{1} (1992), 406--412.
		
		\bibitem{novikov01} R.~G.~Novikov, \textit{Une formule d'inversion pour la transformation d'un rayonnement X att\' enu\' e},  C. R. Acad. Sci. Paris S\' er. I Math., \textbf{332} (2001),  1059--1063.
		
		\bibitem{novikov02} R.~G.~Novikov, \textit{On the range characterization for the two-dimensional attenuated x-ray transformation}, Inverse Problems \textbf{18} (2002), no. 3, 677--700.
		
		% \bibitem{novikov02b} R. G. Novikov, \textit{On determination of a gauge field on $\BR^d$ from its non-abelian Radon transform along oriented straight lines}, J. Inst. Math. Jussieu, \textbf{1} (2002), 559--629.
		
		
		%\bibitem{palamodov09} V.~Palamodov, \textit{Reconstruction of a differential form from doppler transform}, SIAM Journal on Mathematical Analysis, \textbf{41(4)} (2009), 1713--1720.
		
		%\bibitem{pantjukhina} E.~Y.~Pantjukhina, \textit{Description of the image of a ray transform in two-dimensional case,} Akad. Nauk SSSR Sibirsk. Otdel., Inst. Mat., Novosibirsk, \textbf{144} (1990), 80--89.
		
		\bibitem{pantyukhina} E.~Yu.~Pantyukhina, \textit{Description of the image of a ray transformation in the two-dimensional case.}(Russian)  Methods for solving inverse problems (Russian), 80--89, \textbf{144}, Akad. Nauk SSSR Sibirsk. Otdel., Inst. Mat., Novosibirsk,  1990. 
		
		%\bibitem {paternain09} G. Paternain, \textit{Transparent connections over negatively curved surfaces}, Journal of Modern Dynamics, \textbf{3} (2009), 311--333.
		
		
		
		
		%\bibitem {paternain13} G. Paternain, \textit{Inverse problems for connections}, in Inverse Problems and Applications:
		%Inside Out II (edited by G. Uhlmann), MSRI Publications 60, Cambridge University Press (2013). 
		
		%\bibitem {paternainSaloUlmann12} G. Paternain, M. Salo, and G. Uhlmann, \textit{The attenuated ray transform for connections and Higgs fields}, Geom. Funct. Anal., \textbf{22(5)} (2012), 1460--1489.
		
		%\bibitem{paternainSaloUhlmann13-1} G. P. Paternain, M. Salo, and G. Uhlmann, \textit{On the range of the  attenuated Ray transform for unitary connections}, Int. Math. Res. Not., \textbf{4} (2015), 873--897.
		
		\bibitem{paternainSaloUhlmann13} G. P. Paternain, M. Salo, and G. Uhlmann, \textit{Tensor tomography on surfaces}, Invent. Math. \textbf{193(1)} (2013), 229--247.
		
		\bibitem{paternainSaloUhlmann14} G. P. Paternain, M. Salo, and G. Uhlmann, \textit{Tensor Tomography: Progress and Challenges}, Chin. Ann. Math. Ser. B., \textbf {35(3)} (2014), 399--428.
		
		%\bibitem{paternainSaloUhlmann15} G. P. Paternain, M. Salo, and G. Uhlmann, \textit{Invariant distributions, Beurling transforms and tensor tomography in higher dimensions}, Math. Ann., \textbf{363} (2015), 305--362.
		
		%\bibitem{paternainSaloUhlmannZhou16} G. P. Paternain, M. Salo, G. Uhlmann, and H. Zhou, \textit{The geodesic X-ray transform with matrix weights}, (2016) arXiv:1605.07894.
		
		
		
		
		%\bibitem{pestovSharafutdinov88} L.~N.~Pestov, and  V.~A.~Sharafutdinov, \textit{Integral Geometry of tensor fields on a manifold of negative curvature}, (Russian) Sibirsk. Mat. Zh.,\textbf{ 29 (3)} (1988), 114–130, 221; translation in Siberian Math. J., \textbf{29 (3)} (1988), 427–441.
		
		\bibitem{pestovUhlmann04} L. Pestov and G. Uhlmann, \textit{On characterization of the range and inversion formulas for the geodesic X-ray transform}, Int. Math. Res. Not., \textbf{80} (2004), 4331--4347.
		
		
		
		\bibitem{radon1917} J. Radon, \textit{\"{U}ber die Bestimmung von Funktionen durch ihre Integralwerte l\"{a}ngs gewisser Mannigfaltigkeiten},
		%Ber. Verh. Sachs. Akad. Wiss. Leipzig, Math-Nat., {\bf K} 1 \textbf{69} (1917), 262--277.
		Ber. S\"{a}chs. Akad. Wiss. Leipzig, Math.-Phys. Kl., \textbf{69} (1917), 262--277.
		%\bibitem{radon} J. Radon, \textit{\"{U}ber die Bestimmung von Funktionen durch ihre Integralwerte l\"{a}ngs gewisser Mannigfaltigkeiten} [On the determination of functions by their integral values along certain manifolds] 75 years of Radon transform (Vienna, 1992), 324–339, Conf. Proc. Lecture Notes Math. Phys., IV, Int. Press, Cambridge, MA, 1994.
		
		%\bibitem{romanov96} V.~G. Romanov, \textit{A stability estimate in the problem of determining the dispersion index and the relaxation for the transport equation}, Siberian Math. J. \textbf{37}(1996), 308--324.
		
		\bibitem{sadiqTamasan01} K.~Sadiq and A.~Tamasan, \textit{On the range of the attenuated Radon transform in strictly convex sets}, Trans. Amer. Math. Soc., \textbf{367(8)} (2015), 5375--5398.
		
		\bibitem{sadiqTamasan02} K.~Sadiq and A.~Tamasan, \textit{On the range characterization of the two dimensional attenuated Doppler transform},  SIAM J. Math. Anal., \textbf{47(3)} (2015), 2001--2021.
		
		\bibitem{sadiqTamasan22} K.~Sadiq and A.~Tamasan, \textit{On the range of the planar $X$-ray transform on the Fourier lattice of the torus},  preprint (2022). 
		
		
		\bibitem{sadiqScherzerTamasan} K.~Sadiq, O.~Scherzer, and A.~Tamasan, \textit{On the X-ray transform of planar symmetric 2-tensors}, J. Math. Anal. Appl., \textbf{442(1)} (2016),  31--49.
		
		
		\bibitem{saloUhlmann11} M. Salo and G. Uhlmann, \textit{The attenuated ray transform on simple surfaces}, J. Differential Geom., \textbf{88(1)} (2011), 161--187.
		
		\bibitem{schuster08} T. Schuster, \textit{20 years of imaging in vector field tomography: a review}. In Y. Censor,
		M. Jiang, A.K. Louis (Eds.), \textit{Mathematical Methods in Biomedical Imaging and Intensity-Modulated Radiation Therapy (IMRT)}, in: Publications of the Scuola Normale Superiore, CRM \textbf{7} (2008) 389--424.
		
		\bibitem{sharafutdinov86} V.~A.~Sharafutdinov, \textit{A problem of integral geometry for generalized tensor fields on $\BR^n$}, Dokl. Akad. Nauk SSSR \textbf{286} (1986), 305--307.
		
		\bibitem{sharafutdinov_book94} V. A. Sharafutdinov, \textit{Integral geometry of tensor fields}, VSP, Utrecht, 1994.
		
		% \bibitem{sharafutdinov97} V.~A.~Sharafutdinov, \textit{Inverse problem of determining a source in the stationary transport equation on a Riemannian manifold}, Mat. Vopr. Teor. Rasprostr. Voln., 26 (1997), 236--242; translation in J. Math. Sci. (New York) 96 (1999), no. 4, 3430--3433.
		
		%\bibitem{sharafutdinov99} V.~A.~Sharafutdinov, \textit{Ray Transform on Riemannian manifolds}, Lecture Notes at the University of Washington, Spring (1999).
		
		\bibitem {sparSLP95} G. Sparr, K. Str{\aa}hl\'en, K. Lindstr\"om, and H. W. Persson, \textit{Doppler tomography for vector fields,} Inverse Problems, \textbf{11} (1995), 1051--1061.
		
		% \bibitem{tamasanThesis} A.~Tamasan, \textit{An inverse 2D boundary value problem in radiation transport}, Ph.D. thesis, University of Washington, 2002.
		
		% \bibitem{tamasan02} A.~Tamasan, \textit{An inverse boundary value problem in two-dimensional transport}, Inverse Problems \textbf{18} (2002), 209--219.
		
		% \bibitem{tamasan03} A.~Tamasan, \textit{Optical tomography in weakly anisotropic scattering media}, Contemporary Mathematics \textbf{333} (2003), 199--207.
		
		\bibitem{tamasan07} A.~Tamasan, \textit{Tomographic reconstruction of vector fields in variable background media}, Inverse Problems \textbf{23} (2007), 2197--2205.
		
		%\bibitem{WHSWW77} P.~T.~Wells, M.~Halliwell, R.~Skidmore, A.~J.~Webb, and  J.~P.~Woodcock, \textit{Tumour detection by ultrasonic Doppler blood-flow signals}, Ultrasonics \textbf{15(5)} (1977), 231--232.
		
		%		\bibitem{yuWang07} H.~Yu and G.~Wang, \textit{Data consistency based rigid motion artifact reduction in fan-beam CT}, IEEE Transactions on Medical Imaging, \textbf{26 (2)}  (2007), 249--260.
		%
		%\bibitem{yuWangEtall06}
		%H.~Yu,  Y. ~Wei,  J.~Hsieh, and G.~Wang, \textit{Data consistency based
		%	translational motion artifact reduction in fan-beam CT},  IEEE Transactions on Medical Imaging, \textbf{25 (6)}  (2006), 792--803.
		
		
		%\bibitem{xia_etall}
		%Y.~Xia, M.~Berger, S.~Bauer, S.~Hu, A.~Aichert, and A.~Maier,
		%\textit{An Improved Extrapolation Scheme for Truncated CT Data Using 2D Fourier-Based Helgason-Ludwig Consistency Conditions}, 
		%Int. J. Biomed. Imaging 2017 1867025.
		
		
	\end{thebibliography}
\end{document}